\pdfoutput=1
\RequirePackage{ifpdf}
\ifpdf % We~are running pdfTeX in pdf mode
\documentclass[pdftex]{sigma}
\else
\documentclass{sigma}
\fi

\usepackage[usenames, dvipsnames]{xcolor}

\usepackage{tikz-cd}
\numberwithin{equation}{section}

\newtheorem{Theorem}{Theorem}[section]
\newtheorem*{Wseven}{Theorem~\ref{prop:strinhtoW7}}
\newtheorem*{Wsconverse}{Proposition~\ref{prop:obstructlift} and Theorem~\ref{prop:dim9}}
\newtheorem*{orithm}{Theorem~\ref{all_n_stringh_or}}
\newtheorem*{surjective}{Proposition~\ref{prop:surjhomotopy}}
\newtheorem*{thm*}{Theorem}
\newtheorem*{noloopspinc}{Theorem~\ref{no_loop_spinc}}
\newtheorem{Corollary}[Theorem]{Corollary}
\newtheorem*{bordgpsThm}{Corollary~\ref{bordism_ring}}
\newtheorem{Lemma}[Theorem]{Lemma}
\newtheorem{ques}[Theorem]{Question}
\newtheorem{Proposition}[Theorem]{Proposition}

\theoremstyle{definition}
\newtheorem{Definition}[Theorem]{Definition}

\newtheorem{Example}[Theorem]{Example}
\newtheorem{Remark}[Theorem]{Remark}

\usepackage{mathtools}
\DeclarePairedDelimiter{\set}{\{}{\}}
\usepackage{tikz-cd}

\usepackage{quiver}

\usetikzlibrary{calc}

\usepackage[mathscr]{eucal}
\makeatletter
\def\instring#1#2{TT\fi\begingroup
 \edef\x{\endgroup\noexpand\in@{#1}{#2}}\x\ifin@}
\def\isuppercase#1{%
 \instring{#1}{ABCDEFGHIJKLMNOPQRSTUVWXYZ}%
}%
\newcommand{\C@lIfUpper}[1]{
 \if\isuppercase{#1}\mathscr{#1}%
 \else #1%
 \fi
}
\newcommand{\cat}[1]{\mathit{\@tfor\next:=#1\do{\C@lIfUpper{\next}}}}
\makeatother

\usepackage{xparse}
\newcommand{\Z}{\mathbb Z}

\newcommand{\C}{\mathbb C}
\newcommand{\U}{\mathrm U}
\newcommand{\Spin}{\mathrm{Spin}}
\newcommand{\MTSpin}{\mathrm{MSpin}}

\newcommand{\MR}{\mathrm{MR}}
\newcommand{\ko}{\mathrm{ko}}
\newcommand{\Sq}{\operatorname{Sq}}
\newcommand{\abs}[1]{\lvert #1 \rvert}

\newcommand{\SU}{\mathrm{SU}}

\newcommand{\SO}{\mathrm{SO}}
\newcommand{\SL}{\mathrm{SL}}

\renewcommand{\O}{\mathrm O}
\newcommand{\id}{\mathrm{id}}

\newcommand{\CP}{\mathbb{CP}}
\newcommand{\bl}{\text{--}}
\newcommand{\R}{\mathbb R}
\newcommand{\Det}{\mathrm{Det}}

\newcommand{\SH}{\mathit{SH}}

\DeclareRobustCommand*{\RaiseBoxByDepth}{%
 \raisebox{\depth}%
}
\newcommand{\uQ}{\RaiseBoxByDepth{\protect\rotatebox{180}{$Q$}}}

\newcommand{\String}{\mathrm{String}}
\newcommand{\KO}{\mathrm{KO}}
\newcommand{\KU}{\mathrm{KU}}
\newcommand{\TMF}{\mathrm{TMF}}
\newcommand{\Tmf}{\mathrm{Tmf}}
\newcommand{\Sph}{\mathbb S}
\newcommand{\KR}{\mathrm{KR}}

\newcommand{\cP}{\mathcal P}

\newcommand{\Q}{\mathbb Q}

\newcommand{\cA}{\mathcal A}
\newcommand{\Ext}{\mathrm{Ext}}
\newcommand{\Hom}{\mathrm{Hom}}

\newcommand{\MTString}{\mathrm{MString}}
\newcommand{\MString}{\mathrm{MString}}
\newcommand{\MSpin}{\MTSpin}
\newcommand{\MU}{\mathrm{MU}}

\newcommand{\ku}{\mathrm{ku}}
\newcommand{\tmf}{\mathrm{tmf}}
\DeclareMathOperator{\MF}{MF}

\newcommand{\BP}{\mathit{BP}}
\newcommand{\ang}[1]{\langle #1 \rangle}
\newcommand{\cE}{\mathcal E}

\DeclareDocumentCommand{\shortexact}{s O{} O{} mmmm}{
\IfBooleanTF{#1}{ % if star
\begin{tikzcd}[ampersand replacement=\&]
	{1} \& {#4} \& {#5} \& {#6} \& {1#7}
	\arrow[from=1-1, to=1-2]
	\arrow["#2", from=1-2, to=1-3]
	\arrow["#3", from=1-3, to=1-4]
	\arrow[from=1-4, to=1-5]
\end{tikzcd}
}{ % no star
\begin{tikzcd}[ampersand replacement=\&]
	{0} \& {#4} \& {#5} \& {#6} \& {0#7}
	\arrow[from=1-1, to=1-2]
	\arrow["#2", from=1-2, to=1-3]
	\arrow["#3", from=1-3, to=1-4]
	\arrow[from=1-4, to=1-5]
\end{tikzcd}
}}

\usepackage{xspace}
\newcommand{\stringh}{${\rm string}^{h}$}
\newcommand{\spinc}{spin\textsuperscript{$c$}\xspace}
\newcommand{\Aut}{\mathrm{Aut}}

\begin{document}

\allowdisplaybreaks

\newcommand{\arXivNumber}{2411.07299}

\renewcommand{\PaperNumber}{100}

\FirstPageHeading

\ShortArticleName{Type IIA String Theory and ${\rm tmf}$ with Level Structure}

\ArticleName{Type IIA String Theory and $\boldsymbol{{\rm tmf}}$ with Level Structure}

\Author{Arun DEBRAY~$^{\rm a}$ and Matthew YU~$^{\rm b}$}

\AuthorNameForHeading{A.~Debray and M.~Yu}

\Address{$^{\rm a)}$~Department of Mathematics, University of Kentucky,\\
\hphantom{$^{\rm a)}$}~719 Patterson Office Tower, Lexington, KY 40506-0027, USA}
\EmailD{\mail{a.debray@uky.edu}}

\Address{$^{\rm b)}$~Mathematical Institute, University of Oxford, Woodstock Road, Oxford, UK}
\EmailD{\mail{yumatthew70@gmail.com}}

\ArticleDates{Received March 04, 2026, in final form September 20, 2026; Published online October 19, 2026}

\Abstract{We look at a new string$^h$ tangential structure first introduced by Devalapurkar and relate it to the $W_7=0$ condition of Diaconescu--Moore--Witten for type IIA string theory and M-theory. We show that a string$^h$ structure on the target space automatically satisfies the $W_7=0$ condition, and we explain when the $W_7=0$ condition lifts to a string$^h$ structure. Devalapurkar initially constructed ${\rm MTString}^h$ in such a way that it orients ${\rm tmf}_1(3)$; we~extend Devalapurkar's result by showing that ${\rm MTString}^h$ orients ${\rm tmf}_1(n)$. We compute the homotopy groups of ${\rm MTString}^h$ in dimensions relevant to physical applications and apply them to anomaly cancellation for certain compactifications of type IIA string theory.}


\Classification{55N22; 55N34; 57R20; 81T50}

\section{Introduction}
Green--Schwarz in their seminal paper \cite{Green:1984sg} showed that ten-dimensional heterotic string theory is free of perturbative anomalies if the spacetime manifold $M$ comes with a spin structure and a~trivialization of the fractional first Pontryagin form $-\tfrac{1}{8\pi^2}F\wedge F$. Later work due to Killingback~\cite{Killingback:1986rd} refines this constraint: if the fractional Pontrjagin class $\tfrac 12 p_1(M)\in H^4(M;\Z)$ vanishes, the corresponding global anomaly cancels, which is a stronger result.\footnote{Though Killingback correctly observed that one should lift to integral cohomology, he incorrectly stated the condition to be a trivialization of $p_1(M)$. Early references including the correct $1/2$ factor include Freed~\cite[\S4]{Fre86}, Freed--Vafa~\cite{FV87}, and Pilch--Warner~\cite[\S 2]{PW88}.}\textsuperscript{,}\footnote{Strictly speaking, these two perspectives of differential forms and integral cohomology should be taken together: the Green--Schwarz mechanism involves trivializing a refinement of $\tfrac 12p_1$ in \textit{differential cohomology}. This will not play a role in this paper; see~\cite{Fre00, FSS12, SSS12} for more information.} Killingback defined a \textit{string structure} to be a spin structure together with this trivialization data.

More generally, turning on a nontrivial gauge bundle $P\to M$ leads to a refined anomaly cancellation condition for the 10d heterotic string: a trivialization of $\frac{1}{2}p_1(M)-c(P)$, where $c$ is a certain degree-$4$ integral characteristic class of the gauge bundle~\cite{Killingback:1986rd}.
Such a target space is said to have a \textit{twisted} string structure. Later work by Witten \cite{Witten:1987f} showed that it was possible to assign a modular form to a manifold with a string structure by studying the partition function of the heterotic string worldsheet.
More specifically he showed that the Witten genus, a natural invariant valued in a generalized cohomology theory called elliptic cohomology, can be reconstructed from the index of the supercharge on the worldsheet.

Ando--Hopkins--Rezk \cite{AHR10} then provided a highly structured refinement of the Witten genus: Hopkins--Miller~\cite{Hopkins1995,hopkins2002algebraic, HopkinsMiller2014} constructed an $E_\infty$-ring spectrum $\TMF$, called ``topological modular forms,'' and Ando--Hopkins--Rezk~\cite{AHR10} produced a map of $E_\infty$-ring spectra
 \begin{equation*}%\label{AHR_or}
 \sigma\colon \ \MString \longrightarrow \TMF
 \end{equation*}
 together with a comparison map from $\pi_*\TMF$ to the ring of $\mathrm{MF}_*$ modular forms, such that upon taking homotopy groups and tensoring with $\Q$, $\sigma$ becomes the map sending a string manifold to its Witten genus.

Subsequently, the Ando--Hopkins--Rezk map, as well as its physical incarnation due to Stolz--Teichner \cite{stolz2011supersymmetric}, have been used to prove anomaly cancellation results in heterotic string theory by Tachikawa, Yamashita, Yonekura, and Zhang~\cite{Tachikawa:2021mvw,TY23, TY:2021mby, Tachikawa:2023nne, TZ24,Yon22}.

 One can reasonably wonder if there exist relationships between the other superstring theories and $\TMF$, hence strengthening the string theory and elliptic cohomology correspondence. A~natural first question to consider is if there is a different ``string-like'' structure which can be related to elliptic cohomology, and which enforces specific symmetry constraints on the target space.

 The goal of this work is to explore this question in the context of type IIA string theory, and to address the homotopical questions this raises. We will propose
 a ``string-like'' structure that enforces the Diaconescu--Moore--Witten~\cite[Section~6.1]{DMW:E8} symmetry constraint given by $W_7(TM) = 0$ on a 10d manifold $M$, where $W_7(TM)\in H^7(M;\Z)$ is an integral Stiefel--Whitney class. This constraint is needed to resolve a sign ambiguity in the 11-dimensional supergravity partition function on $M\times S^1$.

More specifically, our solution for the ``string-like'' structure is a \textit{\stringh\ structure}, a variant of string structure recently defined by Devalapurkar~\cite{Dev}.
\begin{Definition}[Devalapurkar~\cite{Dev}]
Let $V\to X$ be a \spinc vector bundle with determinant line bundle $L\to X$. A \textit{\stringh\ structure} on $V$ is the data of a trivialization of $\square_\ku(\lambda(V\oplus L))$, where $\lambda$ is the spin characteristic class with $2\lambda = p_1$ and $\square_\ku\colon H^4(X;\Z)\to \ku^7(X)$ is the Bockstein for the cofiber sequence $\Sigma^2 \ku\overset\beta\to \ku\to H\Z$.
\end{Definition}
We in fact give four definitions of \stringh\ structures (Definitions~\ref{triv_stringh_defn}, \ref{lift_stringh_defn}, \ref{ancillary}, and~\ref{4th_defn}) and show they are equivalent in Theorem~\ref{defns_are_equiv}. Then we show that \stringh\ structures answer the call we raised above.
\begin{itemize}\itemsep=0pt
 \item String$^h$ structures are indeed twisted string structures, and if $V$ has a \stringh\ structure then $W_7(V)$ is canonically trivialized.
 \item Moreover, a \spinc structure and a trivialization of $W_7$ is a \emph{good approximation} of a \stringh\ structure in a range of dimensions relevant to string theory, as we explain further below.
 \item There is a map of spectra $\MString^h\bigl[\frac{1}{n}\bigr] \to \tmf_1(n)$, where the latter is the spectrum of \textit{(connective) topological modular forms with level structure} for the subgroup $\Gamma_1(n)\subset \SL_2(\Z)$.\footnote{For $n = 3$, this was first shown a different way by Devalapurkar~\cite{Dev}.}
\end{itemize}

\subsection*{Main results}

The first of our results makes precise the first bullet above. Namely it facilitates the connection between string$^h$ structures
and the structure that corresponds to the $W_7=0$ condition.

\begin{Wseven}
If $V$ is a \stringh\ vector bundle, then $V$ has a canonical \spinc structure and trivialization of $W_7(V)$.
\end{Wseven}

\begin{Wsconverse}
 Let $X$ be a \spinc manifold of dimension $n\le 8$. Every trivialization of~$W_7(X)$ lifts to a \stringh\ structure. If $X$ is closed, this is also true in dimension~$9$.
\end{Wsconverse}

This means that for the purpose of studying compactifications of type IIA string theory in dimensions $8$ and below, there is no loss of generality in upgrading the \spinc structure and trivialization of $W_7$ to a \stringh\ structure. We do not claim that this extra structure is present in the physics -- only that the additional structure helps answer mathematical questions arising from these theories. For example, we will use this to study anomalies of these theories in Examples~\ref{ex:spin} and~\ref{ex:simplyconnectedLie}.

To prove those anomaly cancellation results, we need to calculate groups of reflection-positive invertible field theories on \stringh\ manifolds (possibly with extra data), which by work of Freed--Hopkins~\cite{Freed:2016rqq} and Grady~\cite{Grady:2023sav} reduces to computing \stringh\ bordism groups of spaces. The germ of this calculation is a collection of \stringh\ orientations of the spectra~$\tmf_1(n)$, the spectra of connective topological modular forms with level structure for $\Gamma_1(n)\subset\SL_2(\Z)$, as constructed by Meier~\cite{Mei23}, which we explain in Section~\ref{section:orientation}.

\begin{thm*}[{Devalapurkar~\cite[Theorem 5]{Dev}}]
There is a map of $E_\infty$-ring spectra $\sigma_D\colon \MString^h_{(2)}\!\to\tmf_1(3)_{(2)}$.
\end{thm*}
We generalize this to arbitrary $n$:
\begin{orithm}
For all $n\ge 2$, there are maps of $E_\infty$-ring spectra
\begin{equation*}
	\sigma_1(n)\colon \ \MTString^h[1/n] \longrightarrow \tmf_1(n).
\end{equation*}
\end{orithm}
In Theorem~\ref{real_lift}, we lift the induced map $\MString^h[1/n]\to \Tmf_1(n)$ on mixed $\Tmf$ to a map of $\Z/2$-equivariant ring spectra.

These theorems partially address an open question dating back to Hill--Lawson~\cite[Section~1]{HL16}.
It is not obvious whether $\sigma_D\simeq\sigma_1(3)$, and we would be interested in knowing whether this is the case.

Using Theorem~\ref{all_n_stringh_or}, we computed \stringh\ bordism groups in degrees relevant for physics applications.
\begin{surjective}
For $n = 2,3$, the map on homotopy groups $\sigma_1(n)\colon \Omega_*^{\String^h}[1/n]\to \tmf_1(n)_*$ is surjective.
\end{surjective}
\begin{bordgpsThm}There is a ring isomorphism
\begin{equation*}
 \Omega_*^{\String^h} \overset\cong\longrightarrow \Z[x_2, x_4, x_6, x_8, y_8, x_{10}, x_{12}, y_{12}, x_{14}, \dotsc]/(\dotsb)
\end{equation*}
where $\abs{x_i} = \abs{y_i} = i$ and all generators and relations not listed are in degrees $16$ and above.
\end{bordgpsThm}

Similar to how a string structure on a manifold $M$ induces a spin structure on its free loop space $LM$, we could wonder if \stringh\ has the analogous property for \spinc structures on~$LM$. Huang--Han--Duan~\cite{DHH:19} showed that, because the groups $\Spin^c_n$ are not simply connected, there are multiple, inequivalent notions of a \spinc\ structure on the loop space of a manifold, parametrized by an integer $k$ called the \textit{level}. Unfortunately, a \stringh\ structure does not induce any of these structures!

\begin{noloopspinc}
There is a closed \stringh\ manifold $M$ such that $LM$ is not \spinc for \emph{any} choice of level.
\end{noloopspinc}

\subsection*{Outline}
The structure of the paper is as follows: In Section~\ref{section:definingstringh}, we introduce four equivalent definitions of \stringh\ structure, and prove a number of useful properties of \stringh\ structures and the corresponding Thom spectrum $\MString^h$. In Section~\ref{subsection:stringc}, we explore the relationship between \stringh\ structures and \spinc structures on loop spaces, proving Theorem~\ref{no_loop_spinc}. In Section~\ref{section:orientation}, we show how $\MTString^h$ orients $\tmf_1(n)$ and compute its homotopy groups in degrees less than $16$. For those mainly interested in the relations to physics, Section~\ref{section:orientation} can be skipped and one can proceed to Section~\ref{section:stringhDMW} where we review the Diaconescu--Moore--Witten anomaly. In Section~\ref{subsection:DMWStringh}, we summarize how the $W_7=0$ anomaly cancellation is related to \stringh\ and vice versa. We then give examples in Section~\ref{subsection:applications} of how the \stringh\ structure can be used to more easily compute the anomalies of theories where a $W_7=0$ structure is equivalent to a \stringh\ structure. In Appendix~\ref{appendix:computations}, we prove Theorem~\ref{appthm}, which is a computation needed for the anomaly cancellation result in Example~\ref{ex:simplyconnectedLie}.

\section[The String\^{}h tangential structure]{The String$\boldsymbol{{}^h}$ tangential structure}\label{section:definingstringh}

In this section, we review work of Devalapurkar~\cite{Dev} on the definitions and basic properties of \stringh\ structures. We define and explore several aspects of \stringh\ structures mostly in terms of characteristic classes. In particular, we give four definitions of \stringh\ structures (Definitions~\ref{triv_stringh_defn}, \ref{lift_stringh_defn}, \ref{ancillary}, and~\ref{4th_defn}). We show these definitions are equivalent in Theorem~\ref{defns_are_equiv}.
String\textsuperscript{$h$} structures are analogues of \spinc structures, and we will frequently make this comparison to provide context for a definition or construction. We then define a canonical \stringh\ structure on a direct sum of \stringh\ vector bundles in~Proposition~\ref{direct_sum_stringh}, describe how a string or complex structure on a vector bundle induces a \stringh\ structure in Propositions~\ref{string_to_stringh} and~\ref{cpx_to_stringh}, and compute some low degree homotopy groups of $B\String^h$ in Lemma~\ref{lem:lowdeg}.

As a lead up to the \stringh\ definitions, we start off with several equivalent ways to define \spinc structures.
\begin{description}\itemsep=0pt
 \item[Trivialization of a class.] A \spinc structure on an oriented vector bundle $V\to X$ is a trivialization of $\square_\Z(w_2(V))$, where $\square_\Z\colon H^2(X;\Z/2)\to H^3(X;\Z)$ is the Bockstein.\footnote{When we refer to a trivialization of a cohomology class, we mean that it can be described as a~map to an Eilenberg--Mac Lane space, and a trivialization is a null-homotopy of that map. Thus a \spinc structure is a~null-homotopy of the map $\square_\Z(w_2(V))\colon X\to K(\Z, 3)$. Likewise, when we later refer to an identification of two cohomology classes $a$ and $b$, we mean a trivialization of $a-b$.}
 \item[Lift of a class.] A \spinc structure on $V$ is a class $c_1\in H^2(X;\Z)$ and an identification of $c_1\bmod 2 = w_2(V)$.
 \item[Twisted spin structure.] A \spinc structure on $V$ is data of a complex line bundle $L$ and a~spin structure on $V\oplus L$. $L$ is called the \textit{determinant line bundle} of the \spinc structure.
 \item[Structure group.] A \spinc structure on $V$, where $V$ has rank $n$, is a lift of the principal $\SO_n$-bundle of frames $\mathcal B_\SO(V)\to X$ of $V$ to a principal \spinc bundle $\mathcal B_{\Spin^c}(V)\to X$, i.e., a~$G$-structure for $G = \Spin^c_n$ with its usual map to $\O_n$.
\end{description}
We will give \stringh\ analogues of each of the first three definitions: trivializing a class in Definition~\ref{triv_stringh_defn}, lifting a class in Definition~\ref{lift_stringh_defn}, and in terms of a twisted string structure in Definition~\ref{ancillary}. We will also give a fourth definition, which does not have a direct analogue for \spinc structures, in Definition~\ref{4th_defn}.
These definitions are equivalent, which we prove in Theorem~\ref{defns_are_equiv}.
\begin{Definition}
For $n\ge 5$, $\Spin_n$ is a compact, simple, simply connected Lie group, so there is a canonical\footnote{As $\Aut(\Z)\cong\set{\pm 1}$, we need to disambiguate $1$ and $-1$. We choose the isomorphism $H^4(B\Spin_n;\Z)\to \Z$ to be the one such that the induced isomorphism $H^4(B\Spin_n;\R)\to\R$ sends the Chern--Weil class of the Killing form to a positive number.} isomorphism $\varphi\colon H^4(B\Spin_n;\Z)\overset\cong\to\Z$. We will let $\lambda\coloneqq\varphi^{-1}(1)$.
\end{Definition}

As usual, $\lambda$ defines a characteristic class of spin vector bundles by pulling back by the classifying map.
\begin{Remark}
For all $n\ge 5$, pulling back by the inclusion $\Spin_n\to\Spin_{n+1}$ sends $\lambda\mapsto\lambda$. Therefore, we may define $\lambda\in H^4(B\Spin_n;\Z)$ for $n<5$ by pulling back from $B\Spin_5$, and by passing to the colimit over all $B\Spin_n$, we obtain $\lambda\in H^4(B\Spin;\Z)$.
\end{Remark}

\begin{Lemma}[{\cite[Lemma 1.6]{Deb24}}]\label{lem:whitneysum}
 Let $V_1$ and $V_2$ be two vector bundles over a topological space $X$ each with a spin structure. Then $\lambda(V_1 \oplus V_2) = \lambda(V_1) + \lambda(V_2)$.
\end{Lemma}
Like for the usual Whitney sum formulas, one can upgrade Lemma~\ref{lem:whitneysum} to say that there is a~canonical null-homotopy of the map
\begin{equation*}
 \lambda(V_1 \oplus V_2) - \lambda(V_1) - \lambda(V_2)\colon \ X\longrightarrow K(\Z, 4).
\end{equation*}

\begin{Lemma}\label{pont}
If $p_1\in H^4(B\Spin_n;\Z)$ denotes the first Pontrjagin class, then $2\lambda = p_1$.
\end{Lemma}
That is, there is a canonical homotopy between the maps $2\lambda,p_1\colon B\Spin\to K(\Z, 4)$.
Thus if $V$ is a spin vector bundle then $2\lambda(V) = p_1(V)$. For this reason, $\lambda$ is often denoted~$\tfrac 12 p_1$.
\begin{Definition}
A \textit{string structure} on a spin vector bundle $V$ is a trivialization of $\lambda(V)$.
\end{Definition}

Sometimes string structures are referred to as \textit{$\O\ang 8$-structures} or \textit{$\ang 8$-structures} (e.g.,~\cite{Gia71}).
\begin{Definition}\label{defn:lambdac}
Let $V\to X$ be a vector bundle with \spinc structure $\mathfrak s$ and determinant line bundle $L$. We define
\begin{equation*}
 \lambda^c(V, \mathfrak s)\coloneqq \lambda(V \oplus L)\in H^4(X;\Z).
\end{equation*}
Often $\mathfrak s$ will be implicit, in which case we will write $\lambda^c(V)$ instead.
\end{Definition}
The class $\lambda^c$ is called $q_2$ in~\cite{Dua18} and $\widehat p$ in~\cite[Section~2.7]{CN19}. The classes $p_c$ from~\cite[equation~(3.10)]{CY20} and $\tfrac{c_1^2 - p_1}{2}$ from~\cite[Construction~2]{Dev} both equal $c_1(L)^2 - \lambda^c$.
\begin{Remark}\label{lambda_c_mod_2}
Since $\lambda\bmod 2 = w_4$, the Whitney sum formula shows $\lambda^c(V)\bmod 2 = w_4(V) + w_2(V)^2$. This can be phrased as the data of a homotopy equivalence of the two maps $\lambda^c\bmod 2$ and $w_4 + w_2^2$ from $B\Spin^c$ to $K(\Z/2, 4)$.
\end{Remark}
\begin{Remark}\label{cpx_lambdac}
If the \spinc structure on $V$ is induced from a complex structure, then~\cite[Lem\-ma~2.39]{CN19}
\begin{equation*}
 \lambda^c(V) = -c_2(V) - c_1(V)^2.
\end{equation*}
If the \spinc structure on $V$ is induced from a spin structure, so that $L$ is trivial, then $\lambda^c(V) = \lambda(V)$. Thus if $V$ is both complex and spin, $\lambda(V) = -c_2(V)$.

All of these identifications can be lifted to canonical homotopies between the corresponding maps $X\to K(\Z, 4)$.
\end{Remark}
\begin{Definition}
Recall the Bott map $\beta\colon\Sigma^2\ku\to\ku$ in connective $K$-theory. Its cofiber is the Postnikov $0$-truncation $\tau_0\colon \ku\to H\Z$, which is in particular a morphism of ring spectra. Thus, associated to the cofiber sequence
\begin{equation}\label{Bott_cofiber}
 \Sigma^2 \ku\overset\beta\longrightarrow
 \ku\overset{\tau_0}{\longrightarrow}
 H\Z
\end{equation}
there is a long exact sequence in cohomology; let $\square_\ku\colon H^n(\bl;\Z)\to \ku^{n+3}(\bl)$ denote the connecting morphism in this long exact sequence, which is called the \textit{$\ku$-theoretic Bockstein homomorphism}.
\end{Definition}
We will let $\square_\Z\colon H^n(\bl;\Z/2)\to H^{n+1}(\bl;\Z)$ denote the Bockstein homomorphism associated to the short exact sequence $0\to\Z\to\Z\to\Z/2\to 0$.
\begin{Definition}[{Devalapurkar~\cite{Dev}}]
\label{triv_stringh_defn}
A \textit{\stringh\ structure} on a \spinc vector bundle $V\to X$ is a~trivialization of $\square_\ku(\lambda^c(V))\in\ku^7(X)$.
\end{Definition}
As always, we say a manifold $M$ is \stringh\ if $TM$ is \stringh.
\begin{Definition}\label{lift_stringh_defn}
A \stringh\ structure on a \spinc vector bundle $V\to X$ is equivalent to a class $c_2^\ku(V)\in\ku^4(X)$ and data identifying $\tau_0(c_2^\ku(V)) = \lambda^c(V)$.
\end{Definition}
Equivalence of these definitions follows immediately from the long exact sequence induced from~\eqref{Bott_cofiber}. The third definition, which we give in Definition~\ref{ancillary}, is not as obviously equivalent.
\begin{Definition}\label{vbtwist}
Let $V\to X$ be a virtual vector bundle. An \textit{$(X, V)$-twisted string structure} on a vector bundle $E\to M$ is data of a map $f\colon M\to X$ and a string structure on $E\oplus f^*V$.

Given an $(X, V)$-twisted string structure on $E$, the bundle $f^*V\to M$ is called the \textit{ancillary bundle}.
\end{Definition}
\begin{Definition}
\label{preaux}
Let $S\to B\U$ be the tautological bundle. Then, a $(B\U, S)$-twisted string structure on a vector bundle $E\to M$ induces a \spinc structure on $E$, called the \textit{auxiliary \spinc structure}, as follows: the two-out-of-three data for \spinc structures implies \spinc structures on $E\oplus f^*(S)$ and on $f^*(S)$ induce one on~$E$. But $E\oplus f^*(S)$ is string, hence spin, hence \spinc, and~$S$ is complex, hence \spinc, so $E$ acquires a canonical \spinc structure.
\end{Definition}

\begin{Definition}\label{ancillary}
Let $V\to X$ be a \spinc vector bundle with \spinc structure $\mathfrak s$.
A \stringh\ structure on $V$ is a $(B\U, S)$-twisted string structure and an equivalence of \spinc structures from the auxiliary \spinc structure to~$\mathfrak s$.
\end{Definition}
By ``equivalence of \spinc structures from $\mathfrak s_1$ to $\mathfrak s_2$'', we mean that we have specified a homotopy class of paths from $\mathfrak s_1$ to $\mathfrak s_2$ within the homotopy type of \spinc structures on $V$,
\begin{Remark}

Definition~\ref{vbtwist} is not the standard way to define twisted string structures, though it appears implicitly in~\cite[Section~3]{BDDM24} and is inspired by a related definition of twisted spin structure due to Hason--Komargodski--Thorngren~\cite[Section~4.1]{HKT20}. A more conventional definition, which goes back to Wang~\cite[Definition~8.4]{Wan08}, chooses $d\in H^4(X;\Z)$ and defines an $(X, d)$-twisted string structure on a spin vector bundle $E\to M$ to be a map $f\colon M\to X$ and a~trivialization of $\lambda(E) - f^*(d)$; see also Sati--Schreiber--Stasheff~\cite[Section~2.2]{SSS12}. This definition cannot apply to our situation: by construction, any $(X, d)$-twisted string vector bundle is spin, but the tangent bundle to $\CP^2$ admits a \stringh\ structure, where the ancillary bundle is $-T\CP^2$, and $T\CP^2$ is not spin.

Definition~\ref{vbtwist} is one way to remedy this issue, though there are twists of string bordism according to Wang's definition that Definition~\ref{vbtwist} cannot describe, including the twists studied in~\cite{Deb24}. To include these twists and \stringh\ structures in a single framework, we need a more general notion of twisted string structure.
\end{Remark}

\begin{Definition}\label{rSH}
Let $\SH$ be the \textit{$($restricted$)$ supercohomology} spectrum\footnote{\label{extSH}Different authors mean different things by ``supercohomology;'' we use ``restricted'' to contrast with \textit{extended supercohomology} as introduced by Kapustin--Thorngren~\cite{KT17} and Wang--Gu~\cite{WG18, WG20}. See also~\cite[Sections~5.3 and~5.4]{GJF19}.} introduced by Freed \cite[Section~1]{Fre08} and Gu--Wen~\cite{GW14}, which is uniquely specified up to homotopy equivalence by $\pi_0(\SH)\cong\Z$, $\pi_{-2}(\SH) \cong\Z/2$, and the $k$-invariant $\square_\Z\circ\Sq^2\colon H^{-2}(\bl;\Z/2)\to H^1(\bl;\Z)$.
\end{Definition}
\begin{Lemma}[{Freed~\cite[Proposition 1.9\,(i)]{Fre08}}]\label{defntild}
There is a unique class $\widetilde \lambda\in\SH^4(B\SO)$ whose pullback to $B\Spin$ is the image of the usual $\lambda\in H^4(B\Spin;\Z)$ under the connective cover map $H^k(\bl;\Z)\to\SH^k(\bl)$.
\end{Lemma}
The Whitney sum formula for $\lambda$ (Lemma~\ref{lem:whitneysum}) also refines to supercohomology.
\begin{Lemma}[{Jenquin~\cite[Corollary 4.9]{Jen05}}]\label{super_whitney}
If $E,F\to X$ are oriented virtual vector bundles, $\widetilde\lambda(E\oplus F) = \widetilde\lambda(E) + \widetilde\lambda(F)\in\SH^4(X)$.
\end{Lemma}

We can make this identification canonical as follows: if $\mathit{SK}(4)\!\coloneqq\! \Omega^\infty \Sigma^4\SH$, then for a~space~$X$, $[X, \mathit{SK}(4)]$ is naturally isomorphic to $\SH^4(X)$. Thus $\widetilde\lambda(E\oplus F)$ and $\widetilde\lambda(E) + \widetilde\lambda(F)$ refine to maps $X\to\mathit{SK}(4)$, and Lemma~\ref{super_whitney} upgrades to the data of a canonical homotopy from the one to the other.

One can use $\widetilde\lambda$ to give another characterization of string structures.
\begin{Lemma}[{\cite[Section~1.4]{JFT20}}]\label{double_jump}
A string structure on an oriented vector bundle $E\to M$ is precisely a trivialization of $\widetilde \lambda(E)\in\SH^4(M)$.
\end{Lemma}
This motivates the following definition.
\begin{Definition}[{\cite[Definition 1.62]{Debray:2023tdd}}]
\label{SHtwist}
Let $X$ be a space and $\widetilde d\in\SH^4(X)$. An \textit{$\bigl(X, \widetilde d\bigr)$-twisted string structure} on an oriented vector bundle $E\to M$ is a map $f\colon M\to X$ and a trivialization of $\widetilde \lambda(E) - f^*(\widetilde d)\in \SH^4(M)$.
\end{Definition}
This or closely related definitions also appear in~\cite{BLM23, Freed:2007vy,JF20, TY23, TY:2021mby, TY25}. Since pulling $\widetilde\lambda$ back to $B\Spin$ recovers $\lambda$, Definition~\ref{SHtwist} encompasses Wang's definition, but it is more general.
\begin{Definition}\label{4th_defn_pre}
Let $r\colon B\U\to B\SO$ be the map defined by forgetting the complex structure. Then a $\bigl(B\U, -r^*\widetilde\lambda\bigr)$-twisted string structure on a vector bundle $V\to X$ induces a \spinc structure on $V$, which we call the \textit{auxiliary \spinc structure}, as follows. By definition, \smash{$\widetilde\lambda(E)+ r^*\widetilde\lambda$} has been trivialized, so its image under the map $\SH^4(X)\to H^2(X;\Z/2)$ induced by Postnikov truncation $\Sigma^4\SH\to\Sigma^2 H\Z/2$ is also trivialized. By Freed~\cite[Proposition~1.9\,(i)]{Fre08}, there is an identification $\widetilde\lambda\bmod 2 = w_2$, so we have trivialized $w_2(E) + r^*(c_1\bmod 2)$. Thus $r^*(c_1)$ is the $\Z$-valued cohomology class lifting $w_2$ that defines the \spinc structure.
\end{Definition}
\begin{Definition}\label{4th_defn}
Let $r$ and $\widetilde\lambda$ be as defined in Lemma~\ref{defntild} and Definition~\ref{4th_defn_pre}; then a \stringh\ structure on a \spinc vector bundle $V\to X$ is a~$\bigl(B\U, -r^*\widetilde \lambda\bigr)$-twisted string structure and an identification of the auxiliary \spinc structure with the original \spinc structure on~$V$.
\end{Definition}

\begin{Theorem}\label{defns_are_equiv}
Definitions~{\rm \ref{triv_stringh_defn}}, {\rm \ref{lift_stringh_defn}}, {\rm \ref{ancillary}}, and~{\rm \ref{4th_defn}} are equivalent.
\end{Theorem}
\begin{proof}
Exactness of the Bockstein long exact sequence associated to~\eqref{Bott_cofiber} implies Definitions~\ref{triv_stringh_defn} and~\ref{lift_stringh_defn} are equivalent. Lemmas~\ref{super_whitney} and~\ref{double_jump} jointly imply Definitions~\ref{ancillary} and~\ref{4th_defn} are equivalent.
Thus we will focus on relating Definitions~\ref{lift_stringh_defn} and~\ref{ancillary}.

For topological spaces $X$, there is an isomorphism $\rho\colon [X, B\SU]\overset\cong\to\ku^4(X)$,
and the composition
\begin{equation*}
 [X, B\SU] \overset\rho\longrightarrow \ku^4(X)\overset{\tau_0}{\longrightarrow} H^4(X;\Z)
\end{equation*}
sends a map $f\colon X\to B\SU$ to $f^*(c_2)$. Thus a class $x\in H^4(X;\Z)$ lifts to $\ku^4(X)$ if and only if $x$ is the second Chern class of an $\SU$-structured vector bundle, and in fact data of a lift is equivalent to the data of that vector bundle and $\SU$-structure.

Now we show a \stringh\ structure in the sense of Definition~\ref{lift_stringh_defn} induces one in the sense of Definition~\ref{ancillary}. By assumption, we have lifted $\lambda^c(V)$ to a class $c_2^\ku(V)\in\ku^4(X)$, which as above is equivalent data to a vector bundle $\widetilde E\to X$ with $\SU$-structure and an identification $c_2(\widetilde E) = \lambda^c(V)$. Let $L$ be the determinant bundle of $V$ and let $E\coloneqq \widetilde E \oplus L$; we will show that this data induces a string structure on $V\oplus E$, which means checking that we have data of trivializations
\begin{itemize}\itemsep=0pt
 \item $w_1(V\oplus E) = 0$,
 \item $w_2(V \oplus E)=0$, and
 \item $\lambda(V\oplus E) = 0$.
\end{itemize}
Because $\widetilde E$ and $L$ are oriented, $E$ is oriented, and because $V$ and $E$ are oriented, $V\oplus E$ is oriented, and therefore $w_1(V\oplus E)$ is trivialized.

Since $V$ is \spinc with determinant bundle $L$, we have data of a trivialization of $w_2(V)+w_2(L)$ coming from the spin structure on $V\oplus L$, and the $\SU$-structure on $\widetilde E$ induces a spin structure on $\widetilde E$ (see, e.g.,~\cite{Sto67}), hence also a trivialization of $w_2(\widetilde E)$. The Whitney sum formula provides for us an identification
\begin{equation*}
 w_2(V\oplus E) = w_2(V) + w_2(L) + 0 = 0
\end{equation*}
(again, we have not just an equality $w_2(V\oplus E) = 0$, but in fact data of a null-homotopy of the corresponding map $X\to K(\Z/2, 2)$).

On to $\lambda$. As described above, $V\oplus L$ and $\widetilde E$ are spin, so we have an identification
\begin{equation*}
 \lambda(V\oplus E) = \lambda\bigl(V\oplus L\oplus\widetilde E\bigr) = \lambda(V\oplus L) + \lambda\bigl(\widetilde E\bigr)
\end{equation*}
using the Whitney sum formula in Lemma~\ref{lem:whitneysum}. By Remark~\ref{cpx_lambdac}, there are identifications $\smash{\lambda\bigl(\widetilde E\bigr)} = -\smash{c_2\bigl(\widetilde E\bigr)} = -\lambda(V\oplus L)$, providing the required trivialization of $\lambda(V\oplus E)$ and therefore a string structure.

Finally, we will show that a \stringh\ structure in the sense of Definition~\ref{ancillary} induces a \stringh\ structure in the sense of Definition~\ref{lift_stringh_defn}. Let $E$ denote the ancillary bundle. Because $V\oplus E$ is string, it is in particular spin, so $w_2(V) + w_2(E)$ is trivialized. We therefore have a \spinc structure on $V$ with determinant bundle $L\coloneqq\Det(E)$, because $w_2(E) = w_2(\Det(E))$ canonically.

Let $\widetilde E\coloneqq E - \Det(E)$. Then we have canonical isomorphisms of complex line bundles
\begin{equation*}
 \Det\bigl(\widetilde E\bigr)\cong \Det(E)\otimes \Det(-\Det(E))\cong\Det(E)\otimes (\Det(E))^\vee\cong\underline\C,
\end{equation*}
giving us data of an $\SU$-structure on $\widetilde E$, and therefore a class $c_2^\ku\in\ku^4(X)$. We are done if we can produce data of an identification $\tau_0(c_2^\ku) = \lambda^c(V)$, i.e., data of an identification $\smash{c_2\bigl(\widetilde E\bigr)} = \lambda(V\oplus L)$. As in the previous part of this proof, the string structure on $V\oplus E$ furnishes an identification $\lambda(V\oplus L) +
\smash{\lambda\bigl(\widetilde E\bigr)} = \lambda(V\oplus E) = 0$, so we have an identification $\lambda(V\oplus L) = -\smash{\lambda\bigl(\widetilde E\bigr)}$; applying Remark~\ref{cpx_lambdac} induces an identification $\smash{c_2\bigl(\widetilde E\bigr)} = \lambda(V\oplus L) = \lambda^c(V)$.
\end{proof}

\begin{Lemma}[shearing]\label{shearing}
Let $V\to X$ be a virtual vector bundle of rank zero and $\MString_{X,V}$ be the Thom spectrum associated to the tangential structure ``$(X,V)$-twisted string structure''. Then there is a canonical homotopy equivalence of spectra
\begin{equation*}
 \MString_{X,V}\xrightarrow\simeq \MString\wedge X^V.
\end{equation*}
\end{Lemma}
To the best of our knowledge, Lemma~\ref{shearing} was originally a folklore lemma; see~\cite[Lem\-ma~10.18 and Corollary~10.19]{DDHM24} for a proof of a closely related, albeit slightly different, result. We list these differences and explain how to work around them.
\begin{enumerate}\itemsep=0pt
 \item The statement of Lemma~\ref{shearing} in \textit{loc.\ cit.} is at the level of bordism groups, not Thom spectra, but the proof constructs an equivalence of tangential structures then invokes the Pontrjagin--Thom theorem. Thus we can use the fact that a homotopy equivalence of tangential structures induces a homotopy equivalence of Thom spectra.
 \item The statement in loc.\ cit.\ is for twisted spin structures, instead of twisted string structures. The proof goes through \textit{mutatis mutandis}; see~\cite[proof of Theorem~3.48]{BDDM24} for discussion for twisted string structures.
\end{enumerate}

Combining Definition~\ref{ancillary} and Lemma~\ref{shearing}, we have the following corollary.

\begin{Corollary}\label{smashing}
There is a canonical equivalence of spectra $\MString^h\overset\simeq\to \MString\wedge\MU$.
\end{Corollary}

\begin{Remark}\label{ingredients_mult}
Recall that $\MString$ and $\MU$ both have canonical $E_\infty$-ring structures. One can think of these structures as witnessing and strengthening several standard facts about string, resp.\ complex vector bundles.
\begin{itemize}\itemsep=0pt
 \item The Whitney sum formulas (Lemmas~\ref{lem:whitneysum} and~\ref{super_whitney}) for the string obstructions $\lambda$ and $\widetilde\lambda$.
 \item The direct sum of string (resp.\ complex) vector bundles has a canonical string (resp.\ complex) structure, which refines to the next statement.
 \item The forgetful maps $B\String\to B\O$ and $B\U\to B\O$ are $E_\infty$, so one can invoke Lewis' theorem~\cite[Section~IX.7.4]{LMSM86} to prove the respective Thom spectra are $E_\infty$.
 \item The product of string (resp.\ complex) manifolds has a canonical string (resp.\ complex) structure, which corresponds to the multiplication on $\pi_*(\MString)$, resp.\ $\pi_*(\MU)$, under the Pontrjagin--Thom theorem.
\end{itemize}
\end{Remark}
\begin{Definition}\label{specprod}
We make $\MString^h$ into an $E_\infty$-ring spectrum using Corollary~\ref{smashing}: take the standard $E_\infty$-structures on $\MString$ and $\MU$ and smash them together.
\end{Definition}

The $E_\infty$-ring structure on $\MString^h$ we defined in Definition~\ref{specprod} induces a graded ring structure on $\pi_*(\MString^h)$. By the Pontrjagin--Thom theorem, this is a product on bordism classes of \stringh\ manifolds. Our next endeavor is to describe this product in terms of some of our definitions of \stringh\ structures.

\begin{Lemma}[Whitney sum formula for $\lambda^c$]\label{lambda_c_whitney}
Let $V,W\to X$ be \spinc vector bundles. Then in $H^4(X;\Z)$, there is an identification
\begin{equation*}
 \lambda^c(V\oplus W) = \lambda^c(V) + c_1(V)c_1(W) + \lambda^c(W).
\end{equation*}
\end{Lemma}
As always, this upgrades to data of a homotopy equivalence of the corresponding maps $X\to K(\Z, 4)$.

\begin{proof}By naturality, it suffices to prove this for $V$ and $W$, the tautological bundles over $B\Spin^c_{n_1}$, resp.\ $B\Spin^c_{n_2}$ for $n_1,n_2\gg 0$. Using Duan's calculation of $H^*(B\Spin^c_n;\Z)$~\cite[Theorem~D]{Dua18} and the K\"unneth formula, we learn that $H^4(B\Spin^c_{n_1}\times B\Spin^c_{n_2};\Z)$ lacks $2$-torsion, so if we can produce an identification $2\lambda^c(V\oplus W) = 2\lambda^c(V) + 2c_1(V)c_1(W) + 2\lambda^c(W)$, that would suffice to prove the lemma.\footnote{Specifically, if we have a homotopy $H\colon 2\lambda^c(V\oplus W) \Rightarrow 2\lambda^c(V) + 2c_1(V)c_1(W) + 2\lambda^c(W)$ of maps $X\to K(\Z, 4)$, we could use the fact that both the source and target of this homotopy are twice another class to obtain a canonical trivialization of both $H$ and their mod $2$ reductions, producing a preimage of $H$ under the multiplication-by-$2$ map.} That is, we want to produce an identification
\begin{equation}\label{rephrase_via_p1}
 p_1(V\oplus W\oplus (L_V \otimes L_W)) = p_1(V\oplus L_V) + 2c_1(V)c_1(W) + p_1(W\oplus L_W),
\end{equation}
where $L_V,L_W$ denote the determinant line bundles of $V$ and $W$, respectively. Here we used the fact that the determinant line bundle for a direct sum of \spinc vector bundles is the tensor product of their determinant line bundles.

The first Pontrjagin class satisfies a Whitney sum formula giving an identification $p_1(E\oplus F) = p_1(E) + p_1(F)$ if $E$ and $F$ are oriented~\cite[Theorem~1.6]{Bro82} (see also~\cite{Tho62}), and using that formula, we can reduce~\eqref{rephrase_via_p1}: to prove the lemma, it suffices to prove that for complex line bundles $L_1$ and $L_2$, there is an identification
\begin{equation}\label{p1_final}
 p_1(L_1\otimes L_2) = p_1(L_1) + 2c_1(L_1)c_1(L_2) + p_1(L_2).
\end{equation}
For any rank-$2$ oriented real vector bundle $E$, we have an identification $p_1(E) = e(E)^2$, and the Euler class of a complex line bundle is additive in tensor products, from which~\eqref{p1_final} follows, and then the lemma too.
\end{proof}

\begin{Lemma}[{Conner--Floyd~\cite[Section~8]{conner:1966}}]
For $n\ge 1$, the classes $c_1,\dotsc,c_n\in H^*(B\U_n;\Z)$ have canonical preimages $c_1^\MU,\dotsc,c_n^\MU\in\MU^*(B\U_n)$. Therefore, the same is true with $\MU$ replaced with any complex-oriented ring spectrum $E$.
\end{Lemma}
These classes are called the \textit{Conner--Floyd Chern classes}; see Adams~\cite[Section~I.4]{Ada74} for more information.

 Starting with Definition~\ref{triv_stringh_defn}, the maps $\bigl[X, \Sigma^7 \ku\bigr]$ represent classes in $\Sigma^{7}\ku$, which through suspension is equivalent to $\bigl[\Sigma X, \Sigma^8 \ku\bigr]$. According to the loops-suspension adjunction, this gives $[ X, \Omega B\U \ang 8]$, so $\Omega B\U\ang 8$ is the classifying space for~$\ku^7$.

\begin{Proposition}\label{direct_sum_stringh}
If $V,W\to X$ are \stringh\ vector bundles, there is a canonical \stringh\ structure on $V\oplus W$ extending the usual direct-sum \spinc structure, characterized in the following equivalent ways.
\begin{description}\itemsep=0pt
 \item[Lift of a class.] The Whitney sum formula $($Lemma~{\rm \ref{lambda_c_whitney})} implies that if $c_2^\ku(V)$, resp.\ $c_2^\ku(W)$ are lifts of $\lambda^c(V)$, resp.\ $\lambda^c(W)$ across $\tau_0$, then $c_2^\ku(V) +c^{\ku}_1(V)c^{\ku}_1(W)+ c_2^\ku(W)$ is a lift of $\lambda^c(V\oplus W)$ and thus defines a \stringh\ structure on $V\oplus W$.
 \item[Trivialization of a class.]
 We will exhibit a canonical null-homotopy of the map
 \begin{equation*}
\square_\ku(\lambda^c(V \oplus W)) - \square_\ku(\lambda^c(V)) - \square_\ku(\lambda^c(W))\colon \ X\to \Omega B\U\ang 8
 \end{equation*}
 which induces the equality
 \begin{equation}\label{box_whitney}
 \square_\ku(\lambda^c(V \oplus W)) = \square_\ku(\lambda^c(V)) + \square_\ku(\lambda^c(W)).
 \end{equation}
 Thus the trivializations of $\square_\ku(\lambda^c(V))$ and $\square_\ku(\lambda^c(W))$ induce a trivialization of \linebreak $\square_\ku(\lambda^c(V\oplus W))$, hence a \stringh\ structure on $V\oplus W$.
 \item[Twisted string structure.]
 Let $E$, resp.\ $F$ be the ancillary bundles to $V$, resp.\ $W$. Then $V\oplus E\oplus W\oplus F$ is a direct sum of two string vector bundles, hence acquires a string structure; switching $E$ and $W$, we have produced a \stringh\ structure on $V\oplus W$ with ancillary bundle $E\oplus F$.
\end{description}
\end{Proposition}

\begin{proof}
The Whitney sum formula and linearity of the Bockstein do not quite prove~\eqref{box_whitney}: they provide a canonical null-homotopy of the map
 \begin{equation*}%\label{box_whitney_not}
 \square_\ku(\lambda^c(V \oplus W)) - (\square_\ku(\lambda^c(V)) + \square_\ku(c_1(V)c_1(W)) + \square_\ku(\lambda^c(W)))\colon \ X\to \Omega B\U\ang 8.
 \end{equation*}
So it suffices to produce a null-homotopy of $\square_\ku(c_1(V)c_1(W))\colon X\to \Omega B\U\ang 8$, or equivalently a~lift to the fiber $B\SU$, the classifying space for $\ku^4$.
It suffices to do this for the universal case $B\U_1\times B\U_1$, where the canonical lift is $c_1^\ku(V)c_1^\ku(W)$.

Therefore, we are done if we can show that under the process we described in the proof of Theorem~\ref{defns_are_equiv}, the direct-sum string structure on $V\oplus E\oplus W\oplus F$ produces the ``obvious'' lift of $\lambda^c(V\oplus W)$, namely
\begin{equation*}
 c_2^\ku(V) + c_1^\ku(V)c_1^\ku(W) + c_2^\ku(W).
\end{equation*}
Since $c_1(V)c_1(W)$ has a canonical lift we just need to show that the direct sum string structure provides lifts to $c_2(V)$ and $c_2(W)$.
The lift we get from the direct-sum string structure is the $\ku^4$-class corresponding to the virtual $\SU$-structured vector bundle $E - \Det(E) \oplus F - \Det(F)$.
Taking~$c_2$ of this gives $c_2(E - \Det(E)) + c_2(F - \Det(F))$ but we have already proven in Theorem~\ref{defns_are_equiv} that $c_2(E - \Det(E)) = c_2(\widetilde E) = \lambda^c(V)$ and the lift of this class is $c^{\ku}_2(V)$; a similar statement holds for $c_2(F-\Det(F))$ where the lift is given by $c^{\ku}_2(W)$.
\end{proof}

\begin{Proposition}\label{Einf_comparison}
The following two maps \smash{$\Omega_m^{\String^h}\times\Omega_n^{\String^h}\to\Omega_{m+n}^{\String^h}$} are equal.
\begin{enumerate}\itemsep=0pt
 \item[$1.$] The multiplication map on $\pi_*\bigl(\MString^h\bigr)\!$ obtained from the $E_\infty$-structure we gave $\MString^h$ in Definition~{\rm \ref{specprod}}.
 \item[$2.$] The map sending the bordism classes of two closed \stringh\ manifolds $M$ and $N$ to their direct product, where the \stringh\ structure on $T(M\times N)\cong TM\boxplus TN$ is the \stringh\ structure on a direct sum we described in Proposition~{\rm \ref{direct_sum_stringh}}.
\end{enumerate}
\end{Proposition}
Here $\boxplus$ denotes the external direct sum: if $p_i$ is the $i^{\mathrm{th}}$ projection map in a product, $TM\boxplus TN$ is the bundle $p_1^*(TM)\oplus p_2^*(TN)\to M\times N$.
\begin{proof}[Sketch of the Proof]
The proof of Lemma~\ref{shearing} identifies the \stringh\ tangential structure on $V\to X$ with the tangential structure of a complex vector bundle $E\to X$ and a string structure on $V\oplus E$, matching Definition~\ref{ancillary}. We gave $\MString^h$ the $E_\infty$-structure which is the smash product of the $E_\infty$-structures on $\MString$ and $\MU$, which as we discussed in Remark~\ref{ingredients_mult} are induced from the string resp.\ complex structures on a direct sum of string resp.\ complex vector bundles. If one traces this direct sum $E_\infty$-structure through the proof of Lemma~\ref{shearing}, one arrives at precisely the third perspective on the direct-sum \stringh\ structure in Proposition~\ref{direct_sum_stringh}, namely the one in which one uses the direct-sum complex structure on ancillary bundles.
\end{proof}

We will next show that string or complex structures induce \stringh\ structures, analogously to how spin or complex structures induce \spinc structures. We will first state the result at the level of bordism spectra, then discuss the effects on structured vector bundles.
\begin{Lemma}\label{string_and_cpx_or}
There are canonical maps of $E_\infty$-ring spectra $\MString\to\MString^h$ and $\MU\to\MString^h$.
\end{Lemma}
The latter map is an example of a \textit{complex orientation}.
\begin{proof}
Smash the unit map $\Sph\to \MU$ with $\MString$ and use Corollary~\ref{smashing} to get the first map; for the second, switch the roles of $\MString$ and $\MU$.
\end{proof}

\begin{Proposition}[string implies \stringh]\label{string_to_stringh}
If $V$ is a vector bundle with a string structure, there is a canonical \stringh\ structure on $V$ characterized in the following equivalent ways:
\begin{description}\itemsep=0pt
 \item[Trivialization of a class.] Since $V$ is string, there is a canonical trivialization of $\lambda^c(V)$ (since this class equals $\lambda(V)$, which has been trivialized), and $\square_\ku(0)$ has a canonical trivialization.
 \item[Lift of a class.] There is a canonical lift of $0\in H^4(X;\Z)$ to $\ku$-cohomology, namely $0\in\ku^4(X)$.
 \item[Twisted string structure.] If $E = 0$, the string structure on $V$ induces a string structure on $V\oplus E$, so we obtain a \stringh\ structure with ancillary bundle $0$.
\end{description}
\end{Proposition}

\begin{proof}
This amounts to the assertion that if you take the $0$ characteristic class or vector bundle and pass it through the identifications we constructed in the proof of Theorem~\ref{defns_are_equiv}, you still end up with $0$, which is straightforward to verify.
\end{proof}

\begin{Proposition}[complex implies \stringh]\label{cpx_to_stringh}
If $V$ is a complex vector bundle, there is a canonical \stringh\ structure on $V$ characterized in the following equivalent ways:
\begin{description}\itemsep=0pt
 \item[Trivialization of a class.]
 The $\ku$-Bockstein of the universal class $\lambda^c\in H^4(B\U;\Z)$ lands in $\ku^7(B\U)$, which is the zero group~{\rm \cite[\emph{Theorem}~5.5.1]{BG10}}, and as a map $B\U\to\Omega B\U\ang 8$, this class has a canonical null-homotopy.
 \item[Lift of a class.] The map $\tau_0\colon \ku^4(B\U)\to H^4(B\U;\Z)$ is surjective, and in the notation of~{\rm \cite[\emph{Theorem} 5.5.1]{BG10}}, the class $-c_2 - c_1^2$ is a preimage of $\lambda^c$.
 \item[Twisted string structure] If $E = -V$, then $V\oplus E = 0$ has a canonical string structure, endowing $V$ with a \stringh\ structure with ancillary bundle $-V$.
\end{description}
\end{Proposition}
See~\cite[Remark~7]{Dev} for a fourth perspective on Proposition~\ref{cpx_to_stringh}. One can also add additional items to Proposition~\ref{string_to_stringh} and~\ref{cpx_to_stringh} describing these maps from the points of view of Definition~\ref{4th_defn} and Corollary~\ref{smashing}.

\begin{proof}
As usual, the equivalence of the first two perspectives follows from the long exact sequence coming from~\eqref{Bott_cofiber}. To bring in the third perspective, recall from the proof of Theorem~\ref{defns_are_equiv} that the preimage of $\lambda^c$ in the second perspective is $-c^\ku_2 - (c^\ku_1)^2$ of the ancillary bundle (the $c_2^\ku$ came from the $\SU$-bundle, and $(c_1^\ku)^2$ from the determinant line bundle), which matches the second perspective.
\end{proof}

Analogously to Proposition~\ref{Einf_comparison}, we should compare the two maps $\Omega_*^\String\to\Omega_*^{\String^h}$ induced by Lemma~\ref{string_and_cpx_or} and Proposition~\ref{string_to_stringh} (and likewise with $\MString$ replaced with $\MU$).
\begin{Proposition}\label{change_struct_same_maps}
The following two maps $\Omega_*^\String\to\Omega_*^{\String^h}$ are equal.
\begin{enumerate}\itemsep=0pt
 \item[$1.$] The map on homotopy groups induced by the map of spectra $\MString\to\MString^h$ from Lemma~{\rm \ref{string_and_cpx_or}}.
 \item[$2.$] The map sending the bordism class of a closed string manifold $M$ to the \stringh\ structure on $M$ induced by Proposition~{\rm \ref{string_to_stringh}}.
\end{enumerate}
The analogous fact is true with $\String$ replaced with $\U$.
\end{Proposition}
The proof is analogous to that of Proposition~\ref{Einf_comparison}: pass the construction in Lemma~\ref{string_and_cpx_or} through the Pontrjagin--Thom theorem and observe that it matches with Proposition~\ref{string_to_stringh}.
\begin{Proposition}\label{product_compatibility}
The construction in Proposition~{\rm \ref{string_to_stringh}} sends the canonical string structure on a direct sum of string vector bundles to the canonical \stringh\ structure from Proposition~{\rm \ref{direct_sum_stringh}}. The same is true for Proposition~{\rm \ref{cpx_to_stringh}} with ``string'' replaced with ``complex''.
\end{Proposition}

By Proposition~\ref{change_struct_same_maps}, we can interpret Proposition~\ref{product_compatibility} in terms of the maps of spectra from Lemma~\ref{string_and_cpx_or}; following Proposition~\ref{Einf_comparison}, Proposition~\ref{product_compatibility} expresses that the maps $\MString\to\MString^h$ and $\MU\to\MString^h$ are $E_\infty$.

\begin{proof}
Both parts of this proposition follow quickly using the twisted string structure/ancillary bundle perspective. For example, if $V$ and $W$ are string, the canonical identification $V\oplus W\overset\cong\to V\oplus 0\oplus W\oplus 0$ identifies the \stringh\ structure on $V\oplus W$ from Proposition~\ref{direct_sum_stringh} with the \stringh\ structure induced from the direct-sum string structure on $V\oplus W$. The proof for complex vector bundles is analogous.
 \end{proof}

We denote by $B\String^h$ the space which classifies \stringh\ bundles. We describe the properties of $B\String^h$ by first considering it in the context of Definitions~\ref{triv_stringh_defn} and~\ref{ancillary}.
\begin{Definition}[{Devalapurkar~\cite[Construction~2]{Dev}}]
 The space $B\String^h$ is the fiber of the map
\begin{equation*}
 \Box_{\ku}(\lambda^c)\colon \ B\Spin^c \longrightarrow \Omega B\U \ang 8.
\end{equation*}
\end{Definition}

This space arises from considering all those \spinc vector bundles together with the data of a~trivialization of the image of $\lambda^c(V)$ in $\ku^7(X)$.
\begin{Remark}\label{4_ways_to_BStringh}
Theorem~\ref{defns_are_equiv} provides several canonically equivalent characterizations of the space $B\String^h$~-- for example, using Definition~\ref{4th_defn}, if $\mathit{SK}(4)$ denotes the classifying space for degree-$4$ supercohomology (see~\cite[equation~(1.53)]{Debray:2023tdd}), then
$B\String^h$ is the fiber of either composition in
\begin{equation*}
\begin{tikzcd}
	{B\SO\times B\U} & {B\SO} \\
	{\mathit{SK}(4)\times\mathit{SK}(4)} & {\mathit{SK}(4).}
	\arrow["\oplus", from=1-1, to=1-2]
	\arrow["{(\widetilde\lambda, r)}"', from=1-1, to=2-1]
	\arrow["{\widetilde\lambda}", from=1-2, to=2-2]
	\arrow["{+}"', from=2-1, to=2-2]
\end{tikzcd}
\end{equation*}
(Lemma~\ref{super_whitney} implies this diagram commutes up to homotopy.)
\end{Remark}

The map $B\String^h \rightarrow B\Spin^c$ can also be deduced from Definition~\ref{ancillary}; for a manifold $X$, a~string structure on $TX \oplus \widetilde{E} \oplus L$ in particular gives a spin structure on $TX \oplus L$, since $\widetilde{E}$ is spin as it has an $\SU$-structure. Therefore, $w_2(TM)=w_2(L)$ and $X$ is spin$^c$. The string structure on $TX\oplus \widetilde E \oplus L$ implies $\lambda(TX\oplus L)= -\smash{\lambda\bigl(\widetilde E\bigr)} = \smash{c_2\bigl(\widetilde E\bigr)}$ where the last equality is due to Remark~\ref{cpx_lambdac}. Hence $B\String^h$ fits into the following pullback square:
\begin{equation}\label{eq:stringhpullback}
\begin{gathered}
\begin{tikzcd}
	{B\String^h} & {B\SU} \\
	{B\Spin^c} & {K(\Z,4)\,.}
	\arrow[,from=1-1, to=1-2]
	\arrow["\lambda^c", from=2-1, to=2-2]
	\arrow["{}", from=1-1, to=2-1]
	\arrow["{c_2}", from=1-2, to=2-2]
	\arrow["\lrcorner"{anchor=center, pos=0.125}, draw=none, from=1-1, to=2-2]
\end{tikzcd}
\end{gathered}
\end{equation}
The diagram in equation~\eqref{eq:stringhpullback} then implies that a \stringh\ structure on a \spinc vector bundle $V$ with determinant bundle $L$ is equivalent data to a $(B\SU, c_2)$-twisted string structure on $V\oplus L$.

By Proposition~\ref{string_to_stringh}, we can relate $B\String^h$ to $B\String$, giving the following diagram whose rows are fiber sequences~\cite[Lemma~3]{Dev}:
\begin{equation}\label{eq:stringhstring}
 \begin{tikzcd}
 B\String^h\arrow[r] & B\Spin^c \arrow[r] & \Omega B\U \ang 8 \\
 B\String \arrow[r]\arrow[u] & B\Spin \arrow[r, "\lambda"]\arrow[u] & K(\Z,4) \arrow[u, "\Box_{\ku}", swap]\,.
\end{tikzcd}
\end{equation}

\begin{Lemma}\label{lem:lowdeg}
 The low-degree homotopy groups of $B\String^h$ are
 \begin{alignat*}{3}
 &\pi_0(B\String^h)\cong 0, \qquad &&\pi_5(B\String^h)  \cong 0,&\\
 &\pi_1(B\String^h)\cong 0, \qquad &&\pi_6(B\String^h)  \cong \Z,&\\
 &\pi_2(B\String^h)\cong \Z, \qquad &&\pi_7(B\String^h)  \cong 0,&\\
 &\pi_3(B\String^h)\cong 0, \qquad &&\pi_8(B\String^h)  \cong \Z^2,&\\
 &\pi_4(B\String^h)\cong \Z, \qquad &&\pi_9(B\String^h)  \cong \Z/2,&
 \end{alignat*}
 and $\pi_{10}(B\String^h)$ is an extension of $\Z/2$ by $\Z$.
\end{Lemma}
\begin{proof}
 We apply the long exact sequence in homotopy groups for the top fiber sequence in~\eqref{eq:stringhstring}.
 \begin{figure}[!ht]
 \centering
\begin{tikzcd}
	{k} & {\pi_k(B\String^h)} & {\pi_k(B\Spin^c)} & {\pi_k(\Omega B\U\ang 8)} \\
	2 & \Z \arrow[r,"\cong"]& {\Z} & 0 \\
	3 & 0 & 0 & 0 \\
	4 & {\Z} \arrow[r,"\cong"] & \Z & {0} \\
	5 & {0} & 0 & 0 \\
	6 & {\Z} & {0} & {0} \\
	7 & {0} & {0} & \Z \arrow[from=7-4,to=6-2,in=-150, out=30,"\cong",swap]\\
	8 & {\Z\oplus \Z} \arrow[r] & {\Z} & 0\\
 9 & \Z/2 \arrow[r, "\cong"]&\Z/2 \arrow[r,"0"] &\Z \arrow[from=9-4,to=8-2,in=-150, out=30]\\
 10 & ? \arrow[r] & \Z/2 &0
\end{tikzcd}
 \caption{Homotopy long exact sequence for computing the homotopy groups of $B\String^h$ in degrees up to 10.}
 \label{fig:homotopyBStringh}
\end{figure}
There is a fibration $B\Spin\to B\Spin^c\to B\U_1$; the induced long exact sequence on homotopy groups tells us $\pi_*(B\Spin^c)$.
The homotopy groups of $\Omega B\U\ang 8$ come from Bott periodicity. Using these, we work out the long exact sequence on homotopy groups in Figure~\ref{fig:homotopyBStringh}, which proves the claim.
\end{proof}

By the Pontrjagin--Thom theorem, the homotopy groups of $\MString^h$ are isomorphic to the bordism groups of manifolds with \stringh\ structures on their stable normal bundles.
\begin{Remark}
Sometimes in this paper, we will consider manifolds with \stringh\ structures on their stable tangent bundles, rather than stable normal bundles. A priori this is a different tangential structure classified by the Madsen--Tillmann spectrum $\mathrm{MTString}^h$, the Thom spectrum of $-V\colon B\String^h\to B\O$, but one can show using Proposition~\ref{direct_sum_stringh} that a \stringh\ structure on $V\to X$ is equivalent data to a \stringh\ structure on $-V\to X$, like for orientations, spin structures, \spinc structures, etc. This furnishes a canonical equivalence $\mathrm{MTString}^h \simeq \mathrm{MString}^h$. We~will therefore pass between tangential and normal \stringh\ structures, and tangential and normal \stringh\ bordism and Thom spectra, without comment, and likewise for spin, \spinc, string, and stably almost complex structures.
\end{Remark}

\section[Relation between string\^{}h structures and spin\^{}c structures on loop spaces]{Relation between $\boldsymbol{{\rm string}^h}$ structures and $\boldsymbol{{\rm spin}^c}$ structures\\ on loop spaces}\label{subsection:stringc}

Thus far we have seen in multiple ways that \stringh\ structures are to string structures as \spinc structures are to spin structures. How far does the analogy go? In this subsection, we give an example where the analogy fails to hold: string structures are closely related to spin structures on loop spaces, but this is not true for \stringh\ structures and \spinc structures on loop spaces~-- instead, \spinc structures on loop spaces are governed by a different structure called a string\textsuperscript{$c$} structure (see~\cite{DHH:19}). We will review the definitions of string$^c$ structures and show that string$^c$ structures induce \stringh\ structures, but not vice versa.

\begin{Remark}\label{stringapp}
We studied this question with applications to string theory in mind. Witten~\cite{Witten:1987cg} showed at a physics level of rigor that the index of the supercharge in a (1+1)d nonlinear sigma model, i.e., the worldsheet of a fermionic (super) string, with target space a string manifold $M$ equals Ochanine's elliptic genus~\cite{Ochanine}. This elliptic genus can be recovered as the $S^1$-equivariant index of a Dirac operator on $LM$, using the string structure on $M$ to define a spin structure on $LM$. The results in this subsection suggest that string$^c$ structures, rather than \stringh\ structures, are the right way to generalize this for the \spinc Dirac operator.
\end{Remark}

\begin{Definition}[{Huang--Han--Duan~\cite{DHH:19}}]\label{strong_stringc}
Let $k\in\Z$. A \textit{strong string$^c$ structure of level $(2k+1)$} is a $\bigl(B\U_1, L^{\otimes (2k+1)}\bigr)$-twisted string structure, where $L\to B\U_1$ denotes the tautological bundle.
\end{Definition}
\begin{Remark}
The definition in~\cite{DHH:19} is phrased differently, as a \spinc structure and a trivialization of a characteristic class, but one can show the two are equivalent by an argument similar to that of Theorem~\ref{defns_are_equiv}.
\end{Remark}

\begin{Remark}
Definition~\ref{strong_stringc} for $k = 0$ was introduced earlier, by Chen--Han--Zhang~\cite[Definition~3.1]{CHZ}, and is sometimes just called a string$^c$ structure. On a \spinc vector bundle $V\to X$, this structure is obstructed by $\lambda^c$ from Definition~\ref{defn:lambdac}. Thus at least a priori this structure is stronger than a string$^h$ structure, which only requires $\Box_{ku}(\lambda^c)$ to vanish. In particular, a string$^c$ structure induces a string$^h$ structure.

See Sati~\cite{Sati:2010dc,Sati:2010ip} for some applications of this structure in physics.
\end{Remark}

\begin{Remark}
There exist \spinc vector bundles $V\to X$ such that $\lambda^c(-V)\ne -\lambda^c(V)$, which means that a string$^c$ structure on $V$ is not equivalent data to a string$^c$ structure on $-V$. Thus tangential and normal string$^c$ structures are not equivalent. This in particular implies the Thom spectra classifying tangential and normal string$^c$ structures do not have $E_\infty$-ring spectrum structures corresponding on bordism groups to direct product.
\end{Remark}

We will compare Definition~\ref{strong_stringc} with notions of \spinc structures on loop spaces. Before doing so, we recall from \cite{Mclaughlin} the analogous story for string manifolds and spin structures on their loop spaces.

Let $\mathcal G$ be a \textit{Fr\'echet Lie group} -- that is, a group that is a Fr\'echet manifold, such that multiplication and inversion are smooth. Brylinski~\cite[Proposition 1.6]{Bry00} showed that the group of Fr\'echet Lie group central extensions
\begin{equation*}
 \shortexact*{\U_1}{\widetilde{\mathcal G}}{\mathcal G},
\end{equation*}
such that $\widetilde{\mathcal G}\to\mathcal G$ is a principal $\U_1$-bundle,
is naturally isomorphic to the \textit{Segal-Mitchison cohomology group}~\cite{Seg70, Seg75} $H_{\mathrm{SM}}^2(\mathcal G;\U_1)$. If $G$ is a (finite-dimensional) Lie group, then $LG$ is a Fr\'echet Lie group, and if $G$ is compact, then there is a canonical homomorphism due to Brylinski--McLaughlin~\cite{BM94}
\begin{equation}\label{SM_H4}
 H^4(BG;\Z)\longrightarrow H_{\mathrm{SM}}^2(LG;\U_1).% \longrightarrow H^4(BG;\Z).
\end{equation}
See also~\cite[Chapter~23]{ADH21}.

In particular, if $G$ is connected, simple, and simply connected, there is a canonical isomorphism $H^4(BG;\Z)\overset\cong\to\Z$. The central extension of $LG$ classified by $1\in\Z$ is denoted $\widehat{LG}$, and is the \textit{universal central extension}: for any abelian Lie group $A$, any Fr\'echet Lie group central $A$-extension $\widetilde{\mathcal G}\to LG$ which is a principal $A$-bundle is isomorphic to an associated bundle
\begin{equation*}
 \widetilde{\mathcal G} \cong \widehat{LG} \times_{\U_1} A
\end{equation*}
for some Lie group homomorphism $\U_1\to A$~\cite[Chapter 4]{pressley1986loop}.

Now let $G = \Spin_n$, and assume $n\ge 5$ so that $G$ is simple and simply connected, and we have the universal central extension $\widehat{L\Spin}_n$ of $L\Spin_n$ by $\U_1$. For any spin manifold $M$, the frame bundle of $LM$ lifts canonically to an $L\Spin_n$-bundle $LP\to LM$.
\begin{Definition}[{McLaughlin~\cite[Section~1]{Mclaughlin}}]
A \textit{spin structure on $LM$} is a lift of $LP$ to a~principal $\widehat{L\Spin}_n$-bundle $\widehat{LP}\to LM$.
\end{Definition}
See also Killingback~\cite{Killingback:1986rd} and Witten~\cite[Section~3]{Witten:1987cg}.
\begin{Theorem}[{McLaughlin~\cite{Mclaughlin}}]\label{str_to_loop_spin}
If $M$ has a string structure, then $LM$ has a spin structure.
\end{Theorem}
\begin{Remark}
Pilch--Warner~\cite[Section~3]{PW88} showed the converse to Theorem~\ref{str_to_loop_spin} is not true, but versions of the converse with additional hypotheses do hold; see McLaughlin (ibid.), Kuribayashi~\cite{Kur96}, Kuribayashi--Yamaguchi~\cite{KY98},
Stolz--Teichner~\cite{ST05, ST04},
Waldorf~\cite{Wal10, Wal12a, Wal12b, Wal15, Wal16, Wal16a}, Kottke--Melrose~\cite{KM13},
Capotosti~\cite{Cap16},
and
Ludewig~\cite{Lud23}. See also Waldorf~\cite{Wal23} for an overview.
\end{Remark}

In a parallel manner, we would expect that for $M$ a \stringh\ manifold, there is a \spinc structure on $LM$. Since $\Spin_n^c$ is neither simple nor simply connected, the story is more complicated -- there is not a universal central extension by $\U_1$, and we will have to care about a $\Z$ worth of central extensions, corresponding to the level of the string$^c$ structure in Definition~\ref{strong_stringc}.
\begin{Definition}Let $k\in\Z$.
The Fr\'echet Lie group $\widehat{L_k\Spin_n^c}$ is the central extension of $L\Spin_n^c$ by $\U_1$ which, under the isomorphism $H_{\mathrm{SM}}^2(L\Spin_n^c;\U_1)\cong H^4(B\Spin_n^c;\Z)$ from~\eqref{SM_H4}, is identified with the class $\lambda^c - k c_1^2$.
\end{Definition}
Huang--Han--Duan~\cite{DHH:19} define the groups $\widehat{L_k\Spin_n^c}$ in a different but equivalent way.

Naturality of~\eqref{SM_H4} implies that the pullback of the central extension $\widehat{L_k\Spin}{}^c_n\to L\Spin_n^c$ along the inclusion $L\Spin_n\to L\Spin_n^c$ is the universal central extension of $L\Spin_n$. So even though we do not have a universal central extension in the \spinc setting, we favor these central extensions out of the group of all possible central extensions.

For any \spinc manifold $M$, the frame bundle on $M$ canonically lifts to an $L\Spin_n^c$-bundle $LQ\to LM$.
\begin{Definition}\label{def:spincLoopSpace}
 Let $M$ be a \spinc manifold. A \textit{level $(2k+1)$ \spinc structure on the loop space $LM$} is a lift of $LQ\to LM$ to a principal $\widehat{L_k\Spin}{}_n^c$-bundle $\widehat{LQ}\to LM$.
\end{Definition}
Again, this definition is different but equivalent to Huang--Han--Duan's notion of a \textit{weak string$^c$ structure of level $(2k+1)$}~\cite[Definition~4.1]{DHH:19}.

Now we can refine our earlier question: if \stringh\ is to string as \spinc is to spin, does the loop space of a \stringh\ manifold have a level $2k+1$ \spinc structure for some $k$? We were surprised to obtain a negative answer.
\begin{Theorem}\label{no_loop_spinc}
There are closed \stringh\ manifolds $M$ such that $LM$ is not \spinc for \emph{any} choice of level.
\end{Theorem}
To prove this we will use a characteristic-class criterion for a loop space having a \spinc structure of a given level.
\begin{Lemma}\label{abelian_loops_split}
Let $A$ be an $E_1$-space. Then there is a natural homotopy equivalence $LA\simeq A\times\Omega A$.
\end{Lemma}
Here $\Omega A$ is the space of loops in $A$ based at the identity. See~\cite{Agu81, Hai21, Zil77} for proofs and generalizations of Lemma~\ref{abelian_loops_split}.
\begin{Definition}\label{other_transgression}
Recall that the Serre spectral sequence for the fibration $G\to EG\to BG$ defines a transgression map $\tau\colon H^{n+1}(BG;\Z)\to H^n(G;\Z)$.
\begin{enumerate}\itemsep=0pt
 \item Let $c\coloneqq \tau(c_1)\in H^1(\Spin^c;\Z)$.
 \item Let $\mu^c\coloneqq \tau(\lambda^c)\in H^3(\Spin^c;\Z)$.
\end{enumerate}
\end{Definition}
Lemma~\ref{abelian_loops_split} implies a homotopy equivalence
\begin{equation*}
 B L\Spin^c \simeq B\Spin^c\times B\Omega\Spin^c \simeq B\Spin^c\times \Spin^c,
\end{equation*}
so the K\"unneth formula tells us that the classes $c$ and $\mu^c$, as well as their products with classes in $H^*(B\Spin^c;\Z)$, define integer-valued cohomology classes for $BL\Spin^c$, and therefore by pullback for $BL\Spin_n^c$ for all $n$. In particular, for any $n$, the class $cc_1$ has infinite order in $H^3(BL\Spin_n^c;\Z)$.
\begin{Proposition}[{Huang--Han--Duan~\cite[Remark~4.3]{DHH:19}}]\label{lspinc_char}
Let $M$ be a \spinc manifold. Then $LM$ has a \spinc structure of level $2k+1$ if and only if $\mu^c(LM) - 2k c(LM)c_1(LM) = 0$ in $H^3(LM;\Z)$.
\end{Proposition}

\begin{Definition}The \textit{loop transgression} map $\nu\colon H^*(M;\Z)\to H^{*-1}(LM;\Z)$ is the composition $\pi_!\circ \mathrm{ev}^*$, where $\mathrm{ev}\colon S^1\times LM\to M$ is the evaluation $(x,\gamma)\mapsto\gamma(x)$ and $\pi_!$ is integration over $S^1$.
\end{Definition}
Though $\nu$ and $\tau$ (from Definition~\ref{other_transgression}) are both called ``transgression'', they are not directly related.
\begin{Proposition}[{Huang--Han--Duan~\cite[Section~2.4]{DHH:19}}]
\label{transgr_char_class}
In $H^*(BL\Spin^c;\Z)$, $\nu(\lambda^c) = \mu^c$ and $\nu(c_1) = c$. Moreover, $\nu(xy) = \nu(x)y + (-1)^{\abs x}\nu(y)x$.
\end{Proposition}
\begin{Corollary}[{Huang--Han--Duan~\cite[Theorem~5.1]{DHH:19}}]
If $M$ is strong string$^c$ of level $2k+1$, then $LM$ has a spin$^c$ structure of level $2k+1$.
\end{Corollary}
\begin{proof}
Since $M$ is strong string$^c$ of level $2k+1$, $\lambda^c(M) - kc_1(L)^2 = 0$, where $L\to M$ is the determinant line bundle of the associated \spinc structure. Thus $\nu\bigl(\lambda^c(M) - kc_1(L)^2\bigr) = 0$ in $H^3(LM;\Z)$. By Proposition~\ref{transgr_char_class}, this means $\mu^c(LM) - 2c(LM)c_1(LM) = 0$, which by Proposition~\ref{lspinc_char} implies $LM$ has a \spinc structure of level $2k+1$.
\end{proof}

\begin{proof}[Proof of Theorem~\ref{no_loop_spinc}]
Let $M\coloneqq \CP^m\times\CP^n$ for $m,n \ge 3$; since $M$ is complex, then by Proposition~\ref{cpx_to_stringh} $M$ has a \stringh\ structure. We will show that there is no $k\in\Z$ such that $\mu^c(LM) - 2kc(LM)c_1(LM) = 0$, so that by Proposition~\ref{lspinc_char} $LM$ does not have a \spinc structure for any level. Let $x\in H^2(M;\Z)$ be the first Chern class of the tautological bundle over the first projective space factor and $y$ be the corresponding class for the second $\CP^n$ factor, so that
\begin{equation*}
 p_1(M) = (m+1)x^2 + (n+1)y^2
\end{equation*}
by the Whitney sum formula for $p_1$.\footnote{As we noted during the proof of Lemma~\ref{lambda_c_whitney}, the first Pontrjagin class satisfies the Whitney sum formula for oriented vector bundles.} The determinant line bundle $L$ for this complex structure satisfies $c_1(L) = c_1(M) = (m+1)x+(n+1)y$, so
\begin{equation}
\label{is_even}
 \textcolor{BrickRed}{p_1(M)} - (2k+1)\textcolor{MidnightBlue}{c_1(L)^2} = \textcolor{BrickRed}{(m+1)x^2 + (n+1)y^2} - (2k+1)\textcolor{MidnightBlue}{((m+1)x+(n+1)y)^2}.
\end{equation}
Since $2\lambda^c(V) = p_1(V)+ c_1(L)^2$ for any \spinc vector bundle $V$ with determinant bundle $L$, then $2\bigl(\lambda^c - kc_1^2\bigr) = p_1 - (2k+1)c_1^2$, and so
\begin{equation}\label{lamccpcp}
 \lambda^c(M) - kc_1(L)^2 = \frac{1}{2}\big((m+1)x^2 + (n+1)y^2 - (2k+1)((m+1)x+(n+1)y)^2\big).
\end{equation}
This equation takes place in $H^4(M;\Z)$ -- we are asserting that the right-hand side of~\eqref{is_even} is even.

By Proposition~\ref{transgr_char_class}, we want to show that the loop transgression of~\eqref{lamccpcp} does not vanish for any $k\in\Z$, as this will imply that $LM$ does not have a \spinc structure of any level. It suffices to pull back across the standard inclusion $\CP^m\hookrightarrow \CP^m\times\CP^n$, which on cohomology sends $x\mapsto x$ and $y\mapsto 0$.
That is, because the loop transgression map is natural, showing that $\nu\bigl(-(m+1)(2km+2k+m)x^2\bigr) \neq 0$ would imply that~\eqref{lamccpcp} also transgresses to a nonzero class. We will show $\nu\bigl(-(m+1)(2km+2k+m)x^2\bigr) \neq 0$ in two steps: first we will show $\nu\bigl(x^2\bigr)$ is infinite-order in $H^3(L\CP^m;\Z)$, and then we will show that $-(m+1)(2km+2k+m)\ne 0$ for the values of $m$ and $k$ of interest.

We will compare the loop transgression maps on $\CP^m$ and $\CP^\infty$.
Since $m\ge 3$, the inclusion $\CP^m\hookrightarrow \CP^\infty$ is at least $7$-connected. The natural isomorphism $\pi_k(X)\simeq\pi_{k-1}(\Omega X)$ thus tells us that $\Omega\CP^m\to\Omega B\U_1$ is at least $6$-connected. For any space $X$, there is a natural fibration $\Omega X\to LX\to X$; combining these two connectedness estimates with the long exact sequence of the fibration, we learn $L\CP^m\to L B\U_1$ is also at least $6$-connected. Naturality of the loop transgression map gives us a commutative diagram
\begin{equation*}\begin{tikzcd}
	{H^4(\CP^\infty; \Z)} & {H^4(\CP^m;\Z)} \\
	{H^3(L\CP^\infty;\Z)} & {H^3(L\CP^m;\Z),}
	\arrow["\cong", from=1-1, to=1-2]
	\arrow["\nu", from=1-1, to=2-1]
	\arrow["\nu", from=1-2, to=2-2]
	\arrow["\cong", from=2-1, to=2-2]
\end{tikzcd}\end{equation*}
and since the maps $\CP^m\to\CP^\infty$ and $L\CP^m\to L\CP^\infty$ are at least $6$-connected, the maps on~$H^4$ and~$H^3$ are isomorphisms. Since $k\ne 0$, then to show $\nu\bigl(k(m+1)x^2\bigr)\ne 0$ in $H^3(L\CP^m;\Z)$, it suffices to show that the transgression map $H^4(\CP^\infty; \Z)\to H^3(L\CP^\infty;\Z)$ is injective. This we know: by Proposition~\ref{transgr_char_class} (then pulling back along $\U_1\hookrightarrow\Spin^c$), $\nu\bigl(c_1^2\bigr) = 2cc_1$, which has infinite order as promised.

Finally, we show that, for $m\ge 3$ and $k\in\Z$ arbitrary, $-(m+1)(2km+2k+m)\ne 0$, so that
\begin{equation*}
 \nu(-(m+1)\bigl(2km+2k+m)x^2\bigr)= -(m+1)(2km+2k+m)\nu\bigl(x^2\bigr)\ne 0.
\end{equation*}
Suppose instead that $-(m+1)(2km+2k+m) = 0$; since $m\ne -1$, we can divide by $-(m+1)$ and deduce that in $\Q$,
\begin{equation}\label{nottrue}
 0 = 2k(m+1)+m.
\end{equation}
If $k = 0$, \eqref{nottrue} forces $m = 0$, which we know is false. If $k\ne 0$, \eqref{nottrue} implies that $k = -m/2(m+1)$, which cannot be satisfied if both $k$ and $m$ are integers and $m\ge 3$.
\end{proof}

We thank an anonymous referee for discussion regarding the following remark.
\begin{Remark}
When $G$ is connected but not simply connected, there is a body of evidence that one should not use $H^4(BG;\Z)$, but rather the supercohomology $\SH^4(BG)$ (Definition~\ref{rSH}). This evidence includes that $\SH^4(BG)$, rather than $H^4(BG;\Z)$, classifies anomalies of two-dimensional fermionic theories with $G$-symmetry, including the worldsheets of the string theories this section was motivated by (see Remark~\ref{stringapp}); in addition, levels for spin Chern--Simons theories are classified by supercohomology, not cohomology~\cite[Section~4.4]{Jen05}.\footnote{For disconnected $G$, one should in general replace $\SH$ with extended supercohomology (see footnote~\ref{extSH}) or the Anderson dual of spin bordism. See for example~\cite[Section~2.1.3]{CHS18}.} (When $G$ is simply connected, the canonical map $H^4(BG;\Z)\to\SH^4(BG)$ is an isomorphism.)

Since $\Spin_n^c$ is not simply connected, it is thus natural to speculate that there is a more general class of central extensions of $L\Spin_n^c$ parametrized by $\SH^4(B\Spin_n^c)$, rather than $H^4(B\Spin_n^c;\Z)$. To do this, one would need a transgression map generalizing~\eqref{SM_H4}, from $\SH^4(BG)$ to a group of central extensions of~$LG$, which is not in the literature. We attempted to construct a~map of this form and did not succeed; as such, we would be interested in knowing whether such a~generalization is possible, and whether it changes the conclusion of Theorem~\ref{no_loop_spinc}.
\end{Remark}

\section[Orienting tmf\_1(n)]{Orienting $\boldsymbol{\tmf_1(n)}$}\label{section:orientation}

In this section, we produce \stringh\ orientations of $\tmf_1(n)$ in Theorem~\ref{all_n_stringh_or}. We start by introducing $\TMF$ and $\Tmf$ with level structure, and from there introduce $\tmf_1(n)$.

The spectrum of \textit{(periodic) topological modular forms} $\TMF$ is the global sections of a sheaf of $E_\infty$-ring spectra $\mathcal{O}^{\mathrm{top}}$ on the \'etale site of the moduli stack of elliptic curves $\mathcal{M}_{\mathrm{ell}}$, that is $\TMF = \mathcal{O}^{\mathrm{top}}(\mathcal{M}_{\mathrm{ell}})$.
The homotopy ring $\pi_{2*}(\TMF)$ (i.e., the same ring with degrees doubled) is rationally\footnote{In fact, these two rings are isomorphic after inverting~$6$.} isomorphic to the ring
\begin{gather*}
 \widetilde{\MF}[\SL_2(\Z),\Z]\cong \Z[c_4,c_6,\Delta^{\pm}]/\bigl(c^3_4-c^2_6-1728\Delta\bigr), \qquad |c_4|=9, \quad |c_6|=12,\quad |\Delta|=12
\end{gather*}
of weakly holomorphic integral modular forms. The homotopy groups of $\TMF$ are periodic with period $576$.

The sheaf $\mathcal O^{\mathrm{top}}$ extends to define a sheaf on the \'etale site of the Deligne--Mumford compactification $\overline{\mathcal M}_{\mathrm{ell}}$ of $\mathcal M_{\mathrm{ell}}$, and the global sections of $\mathcal O^{\mathrm{top}}\to \overline{\mathcal M}_{\mathrm{ell}}$ are a spectrum $\Tmf$ which is neither periodic nor connective, called \textit{non-periodic nonconnective topological modular forms} or \textit{mixed $\Tmf$}.
The homotopy ring of $\Tmf$ is closely related to the ring of holomorphic integral modular forms
\begin{equation*}
 \MF(\SL_2(\Z), \Z)\cong\Z[c_4,c_6,\Delta]/\bigl(c^3_4-c^2_6-1728\Delta\bigr) .
\end{equation*}
There is a map $\pi_{2*}(\Tmf)\to\MF(\SL_2(\Z), \Z)$, and after inverting $6$, this is an isomorphism \emph{but only in nonnegative degrees}. Therefore, one defines the connective cover $\tmf\coloneqq\tau_{\ge 0}\Tmf$, so that there is an isomorphism $\pi_{2*}(\tmf)\cong \MF(\SL_2(\Z), \Z)\otimes\Z[1/6]$ in all degrees.

By considering moduli spaces with a little extra structure, one obtains interesting variants of $\TMF$ and $\Tmf$.
\begin{Definition}[Hill--Lawson~\cite{HL16}]\label{1n_log_defn}
Let $n\ge 1$, $\mathcal M_1(n)$ denote the moduli stack of elliptic curves with a chosen point of order $n$, and $\overline{\mathcal M}_1(n)$ be the Deligne--Mumford compactification of~$\mathcal M_1(n)$. The global sections of the pullback of~$\mathcal O^{\mathrm{top}}$ to $\mathcal M_1(n)[1/n]$, resp.\ to the log-\'etale site of~$\overline{\mathcal M}_1(n)[1/n]$ are denoted $\TMF_1(n)$, resp.\ $\Tmf_1(n)$.
\end{Definition}
Hill--Lawson also define analogous series of spectra $\TMF(n)$ and $\TMF_0(n)$, and $\Tmf(n)$ and $\Tmf_0(n)$. Prior to their work, various examples of these families of spectra were introduced by Behrens~\cite{Beh06, Beh07}, Mahowald--Rezk~\cite{MR09}, and Stojanoska~\cite{Sto12}.

Both $\TMF_1(n)$ and $\Tmf_1(n)$ are $E_\infty$-ring spectra by construction (in fact, $E_\infty$ $\TMF$-, resp.\ $\Tmf$-algebra spectra), and there is a ring spectrum map $\Tmf_1(n)\to\TMF_1(n)$. There is a~rational isomorphism from $\pi_{2*}(\TMF_1(n))$ to the ring $\widetilde{\MF}\bigl(\Gamma_1(n),\Z\bigl[\frac{1}{n}\bigr]\bigr)$ of weakly holomorphic modular forms for the congruence subgroup $\Gamma_1(n)\subset\SL_2(\Z)$, also called \textit{integral modular forms at level $n$}.

However, this analogy does not continue to mixed $\Tmf_1(n)$: the ring $\pi_*(\Tmf_1(n))\otimes\Q$ and the ring of holomorphic modular forms for $\Gamma_1(n)$ tensored with $\Q$ are not always isomorphic, even restricted to nonnegative degrees. This means that the connective cover of $\Tmf_1(n)$ is not always the right analogue of $\tmf$.

Fortunately, the discrepancy is not huge: the sole discrepancy is that $\pi_1(\Tmf_1(n))$ may be nonzero, and frequently it is $0$, including for all $n\le 22$ (see, for example, \cite[Remark~3.14]{Mei22}). Meier~\cite{Mei23}, following a general procedure of Lawson~\cite{lawson2015shimura} to remove $\pi_1$, constructs for all $n\ge 2$ an $E_\infty$-ring spectrum $\tmf_1(n)$ with $\pi_1(\tmf_1(n)) = 0$ and a map $\tmf_1(n)\to\Tmf_1(n)$ which is an isomorphism for $n = 0$ and $n\ge 2$, implying $\pi_{2*}(\tmf_1(n))$ is rationally isomorphic to the ring of holomorphic modular forms of level $n$ in all degrees. For this paper, $\tmf_1(n)$ always refers to Meier's construction, whether or not this is the connective cover of $\Tmf_1(n)$.
\begin{Remark}
Lawson--Naumann~\cite{LN14} constructed $\tmf_1(3)$ $2$-locally as an $E_\infty$-ring spectrum before Meier's work, and identified it with $\BP\ang 2$; in this case, $\pi_1(\Tmf_1(3))$ vanishes. See also Hill--Meier~\cite{HM17}.
\end{Remark}

Since topological modular forms with level structure were first systematically studied by Hill--Lawson~\cite{HL16}, it has been an open question to orient them by a Thom spectrum which is a better approximation than $\MU$ or $\MString$: see, for example~\cite[Section~1]{HL16}. Wilson~\cite{Wil15} provided some answers to this question, but does not answer it for $\tmf_1(n)$. Recently, Devalapurkar~\cite{Dev} answered this for $\tmf_1(3)$ using forthcoming work of Hahn--Senger.

\begin{Theorem}[{\cite[Theorem 5]{Dev}}]
There is a map of $E_\infty$-ring spectra $\sigma_D\colon\! \MString^h_{(2)}\!\to\!\tmf_1(3)_{(2)}$ such that the following diagram commutes:
\begin{equation*}\begin{tikzcd}
	\MString_{(2)} & \tmf_{(2)} \\
	{\MString^h_{(2)}} & {\tmf_1(3)_{(2)}.}
	\arrow["\sigma", from=1-1, to=1-2]
	\arrow[from=1-1, to=2-1]
	\arrow[from=1-2, to=2-2]
	\arrow["\sigma_D", from=2-1, to=2-2]
\end{tikzcd}\end{equation*}
\end{Theorem}
We will generalize this to arbitrary $n$.

\begin{Theorem}\label{all_n_stringh_or}
For all $n\ge 2$, there are maps of $E_\infty$-ring spectra
\begin{equation*}
	\sigma_1(n)\colon \ \MTString^h[1/n] \longrightarrow \tmf_1(n)
\end{equation*}
such that the composition of $\sigma_1(n)$ with the complex orientation on $\MTString^h$ constructed in Lemma~{\rm \ref{string_and_cpx_or}}, is the complex orientation of $\tmf_1(n)$ constructed in Senger~{\rm \cite[\emph{Theorem}~1.7]{Sen23}}, and there is a commutative
square
\begin{equation*}
\begin{tikzcd}
	\MTString[1/n] & \tmf[1/n] \\
	{\MTString^h[1/n]} & {\tmf_1(n)} .
	\arrow["\sigma", from=1-1, to=1-2]
	\arrow[from=1-1, to=2-1]
	\arrow[from=1-2, to=2-2]
	\arrow["{\sigma_1(n)}", from=2-1, to=2-2]
\end{tikzcd}
\end{equation*}
\end{Theorem}
\begin{Remark}
Our proof uses completely different methods than Devalapurkar's, and it would be interesting to know whether there is a $2$-local equivalence $\sigma_D\simeq\sigma_1(3)$.
\end{Remark}

The $E_\infty$-ring map $\MTString \rightarrow \MTString^h$ arising from Lemma~\ref{string_and_cpx_or}, and the map $\tmf[1/n]\rightarrow \tmf_1(n)$ is induced by the inclusion of the moduli stack of elliptic curves with a chosen point of order 3 into the moduli stack of all elliptic curves.
 We prove Theorem~\ref{all_n_stringh_or} by first showing it for neither-connective-nor-periodic $\Tmf_1(n)$, then lifting
to $\tmf_1(n)$.
\begin{Proposition}\label{not_connective_or}
The analogue of Theorem~{\rm \ref{all_n_stringh_or}}, but with $\Tmf_1(n)$ in place of $\tmf_1(n)$, is true.
\end{Proposition}

\begin{proof}
Throughout this proof, we invert $n$.

We will repeatedly use the fact that if $A$, $B$, $C$, and $D$ are $E_\infty$-ring spectra, and $f\colon A\to C$ and $g\colon B\to D$ are $E_\infty$-ring
homomorphisms, then $f\wedge g\colon A\wedge B\to C\wedge D$ has a canonical $E_\infty$-ring homomorphism structure.

Specifically, use this fact to daisy-chain together the following $E_\infty$-ring maps:
\begin{enumerate}\itemsep=0pt
	\item \label{stringh_MU} The $E_\infty$ equivalence $\MTString^h\simeq \MTString\wedge \MU $ we established in
	Corollary~\ref{smashing}. %, \matt{to be re-xrefed}
	\item\label{AHR_item} The $\sigma$-orientation $\sigma\colon \MTString \to \tmf$ constructed by
	Ando, Hopkins and Rezk~\cite[Theo\-rem~12.3]{AHR10}.
	\item\label{absm_item} The complex orientation $M(n)\colon \MU \to\tmf_1(n)\to \Tmf_1(n)$ due to Senger~\cite[Theorem~1.7]{Sen23}.\footnote{Absmeier~\cite[Theorem 1]{Abs21} uses different methods to construct $E_\infty$-orientations $\MU[\zeta_n, 1/n]\to\Tmf_1(n)$, where $\zeta_n$ is a primitive $n^{\mathrm{th}}$ root of unity; we use Senger's orientation to avoid $\zeta_n$.}
	\item\label{unit_item} The unit map $A(n)\colon \tmf\to\Tmf_1(n)$ of the $E_\infty$-$\tmf$-algebra structure on
	$\Tmf_1(n)$ obtained by Hill--Lawson~\cite[Theorem~6.1]{HL16}.
\end{enumerate}
Thus, the following composition is a homomorphism of $E_\infty$-ring spectra:
\begin{align*}
	\MTString^h & \underset{\eqref{stringh_MU}}{\overset{\simeq}{\longrightarrow}} \MTString\wedge \MU
	\underset{\text{(\ref{AHR_item}), (\ref{absm_item})}}{\overset{\sigma\wedge M(n)}{\longrightarrow}} \tmf\wedge
	\Tmf_1(n) \underset{\eqref{unit_item}}{\overset{A(n)\wedge\id}{\longrightarrow}} \Tmf_1(n)\wedge\Tmf_1(n)\\
& \overset{\mu}{\longrightarrow} \Tmf_1(n),
\end{align*}
where the final map is multiplication.
\end{proof}

\begin{Proposition}\label{connective_lifting}
Let $R$ be a connective $E_\infty$-ring spectrum with isomorphisms $\psi\colon \pi_0(R)\overset\cong\to\Z$ and $\pi_1(R) = 0$. Given a morphism $f\colon R\to \tau_{\ge 0}\Tmf_1(n)$ of $E_\infty$-ring spectra, there is a canonical lift of $f$ to a map $\widetilde f\colon R\to\tmf_1(n)$.
\end{Proposition}

\begin{proof}
Meier~\cite[Proposition 2.9, Lemmas 2.10 and 2.11]{Mei23} constructs a pullback square of $E_\infty$-ring spectra
\begin{equation}\label{meier_pullback}
\begin{tikzcd}
	{\tmf_1(n)} & {H\pi_0(\Tmf_1(n))} \\
	{\tau_{\ge 0}\Tmf_1(n)} & {\tau_{0:1}\Tmf_1(n),}
	\arrow["{\tau_{\le 1}}", from=2-1, to=2-2]
	\arrow["\varphi", from=1-2, to=2-2]
	\arrow[from=1-1, to=2-1]
	\arrow[from=1-1, to=1-2]
	\arrow["\lrcorner"{anchor=center, pos=0.125}, draw=none, from=1-1, to=2-2]
\end{tikzcd}\end{equation}
and shows $\varphi$ is the unique $E_\infty$-ring map $H\pi_0(\Tmf_1(n))\to\tau_{0:1}\Tmf_1(n)$ inducing a ring isomorphism on $\pi_0$. Therefore, it suffices to produce $E_\infty$-ring maps $a\colon R\to\tau_{\ge 0}\Tmf_1(n)$ and $b\colon R\to H\pi_0(\Tmf_1(n))$ and an identification of their compositions with the maps $\tau_{\le 1}$, resp.\ $\varphi$ in~\eqref{meier_pullback}:
\begin{equation*}
\begin{tikzcd}
	\textcolor{BrickRed}{R} & {H\pi_0(\Tmf_1(n))} \\
	{\tau_{\ge 0}\Tmf_1(n)} & {\tau_{0:1}\Tmf_1(n).}
	\arrow["{\tau_{\le 1}}", from=2-1, to=2-2]
	\arrow["\varphi", from=1-2, to=2-2]
	\arrow["a"', color={BrickRed}, from=1-1, to=2-1]
	\arrow["b", color={BrickRed}, from=1-1, to=1-2]
\end{tikzcd}
\end{equation*}
Choose $a = f$ and let $b$ be the composition
\begin{equation*}
 R\overset{\tau_{\ge 0}}{\longrightarrow} H\pi_0(R)\overset{\psi}{\longrightarrow} H\Z \overset{\boldsymbol 1}{\longrightarrow} H\pi_0\Tmf_1(n),
\end{equation*}
where $\boldsymbol 1$ is the unit. To provide an identification $\tau_{\le 1}\circ a \simeq \varphi\circ b$, first use that the target is $1$-truncated, so that both compositions canonically factor through $\tau_{\le 1} R$. Since $R$ is connective and $\pi_1(R) = 0$, the $0$-truncation map $\tau_{\le 1}R\to\tau_{\le 0}R \simeq H\pi_0(R)$ is an equivalence of $E_\infty$-ring spectra. Thus we without loss of generality replace $R$ with $H\pi_0(R)$.

For both $\tau_{\le 1}\circ a$ and $\varphi\circ b$, the induced map on $\pi_0$ is the localization $\Z\to\Z[1/n]$, using the specified isomorphism \smash{$\psi\colon \pi_0(R)\overset\cong\to \Z$} and Meier's identification~\cite[Lemma 2.11]{Mei23} $\pi_0(\Tmf_1(n))\cong\Z[1/n]$. As the ring homomorphism $\Z\to\Z[1/n]$ is \'etale,\footnote{Meier~\cite{Mei23} works in a more general setting where a $E_\infty$-ring spectrum $R$ has $\pi_0 R$ an \'etale extension of a~localization of $\Z$. Thus he imposes \'etaleness as a~hypothesis, while in our setting the map is always \'etale.} a theorem of Lurie~\cite[Theorem~7.5.0.6]{HA} shows that this map on $\pi_0$ lifts uniquely to an $E_\infty$-ring map $H\pi_0(R)\to\tau_{0:1}\Tmf_1(n)$ with a contractible space of automorphisms. Thus $\tau_{\le 1}\circ a$ and $\varphi\circ b$ are canonically equivalent up to contractible data and we may conclude.
\end{proof}

Now proving Theorem~\ref{all_n_stringh_or} amounts to showing $\MTString^h$ satisfies the hypotheses of Proposition~\ref{connective_lifting}.

\begin{proof}[Proof of Theorem~\ref{all_n_stringh_or}]
Thom spectra of rank-zero virtual vector bundles, such as $\MTString^h$, are connective. Connectivity provides a canonical lift of the $E_\infty$-ring map $\MTString^h\to\Tmf_1(n)$ constructed in Proposition~\ref{not_connective_or} to an $E_\infty$-ring map $\MTString^h\to\tau_{\ge0}\Tmf_1(n)$. Therefore, to lift to $\tmf_1(n)$, it suffices to show \smash{$\Omega_0^{\String^h}\cong\Z$} and \smash{$\Omega_1^{\String^h}\cong 0$}, then invoke Proposition~\ref{connective_lifting}.

If $M$ is a spin$^c$ manifold of dimension $3$ or below, $\lambda^c(M) = 0$, because it is an element of~$H^4(M;\Z)\cong 0$. Therefore, $M$ admits a canonical string$^h$ structure: lift $\lambda^c(M)$ to $0\in\ku^4(M)$. This implies that for $k\le 2$, $\Omega_k^{\String^h}\to\Omega_k^{\Spin^c}$ is an isomorphism, and $\Omega_0^{\Spin^c}\cong\Z$, ${\Omega_1^{\Spin^c} \cong 0}$.
\end{proof}

\subsection{Real-equivariance}
We briefly discuss a real-equivariant generalization of Theorem~\ref{all_n_stringh_or}, and as with everything, we begin with the \spinc story.

Complex conjugation defines a $\Z/2$-action on complex $K$-theory; the resulting $\Z/2$-equiv\-ari\-ant spectrum is called \textit{real(-equivariant) $K$-theory} and denoted $\KR$~\cite{Ati66}.
The underlying spectrum of $\KR$ is $\KU$, and the $\Z/2$-(homotopy) fixed point spectrum is $\KO$. $\KR$ is \textit{cofree} (see, e.g.,~\cite{HZ20}), meaning that its structure as a $\Z/2$-$E_\infty$-ring spectrum is induced from a $\Z/2$-action on the spectrum $\KU$ by $E_\infty$-ring maps.

Landweber~\cite{Lan67, Lan68}, Fujii~\cite{Fuj75}, Araki~\cite{Ara79, Ara79a}, and Araki--Murayama~\cite{AM78} constructed a $\Z/2$-equivariant-ring spectrum $\MR$ whose underlying spectrum is $\MU$ with $\Z/2$-action by complex conjugation. $\MR$ also has an $E_\infty$-structure~\cite[Proposition~3.1]{HK01}.
\begin{Definition}[\cite{Ara79, AM78}]
A \textit{real orientation} of a $\Z/2$-ring spectrum $E$ is a homomorphism of $\Z/2$-ring spectra $\MR\to E$.
\end{Definition}

\begin{Theorem}\label{cfmap}
$\KR$ is real-oriented, and the real orientation may be chosen to be a $\Z/2$-$E_\infty$-ring map $\mathit{cf}_\R\colon\MR\to\KR$.
\end{Theorem}
Araki--Murayama~\cite[Section~12]{AM78} gave a real orientation, in the sense of a map of homotopy commutative ring spectra $\MR\to\KR$. The lift to a $\Z/2$-$E_\infty$-ring map appears to be a folklore result: see~\cite{HK24}.

Nonequivariantly, the complex orientation $\MU\to\KU$ constructed by Conner--Floyd~\cite[Section~5]{conner:1966} factors through $E_\infty$-ring maps $u\colon \MU\to\MSpin^c$ and $\widehat A\colon \MSpin^c\to\KU$; the former can be constructed similar to the methods we used in Corollary~\ref{smashing} and the latter is due to Joachim~\cite{Joa04}. Halladay--Kamel~\cite{HK24} recently generalized this to the real-equivariant setting.
\begin{Theorem}[Halladay--Kamel~\cite{HK24}]\label{hkthm}
There is a $\Z/2$-$E_\infty$-ring spectrum $\MSpin_\R^c$ and $\Z/2$-$E_\infty$-ring maps $u_\R\colon\MR\to\MSpin_\R^c$ and $\widehat A_\R\colon \MSpin_\R^c\to\KR$ such that
\begin{enumerate}\itemsep=0pt
 \item[$(1)$] the underlying spectrum of $\MSpin_\R^c$ is $\MSpin^c$,
 \item[$(2)$] forgetting to underlying spectra, $\widehat A_\R$, resp.\ $u_\R$ restrict to $\widehat A$, resp.\ $u$, and
 \item[$(3)$] the real orientations $\widehat A_\R\circ u_\R$ and $\mathit{cf}_\R$ $($see Theorem~{\rm \ref{cfmap})} are equivalent.
\end{enumerate}
\end{Theorem}
That is, Halladay--Kamel answer the question, ``what is to $\KR$ as $\MSpin^c$ is to $\KU$?''

We prove an analogue of Theorem~\ref{hkthm} for topological forms with level structure in Theorem~\ref{real_lift}. However, the version of the theorem stated there is slightly weaker than the na\"{\i}ve generalization of Theorem~\ref{all_n_stringh_or}: the orientation lands in $\Tmf_1(n)_\R$, rather than~$\tmf_1(n)_\R$, as aspects of the lifting argument from $\tau_{\ge 0}\Tmf_1(n)$ to $\tmf_1(n)$ are tricky to make equivariant. Moreover, we did not construct an $E_\infty$ map, only a map of $\Z/2$-ring spectra. Ultimately this is because there is not yet a construction of a $\Z/2$-$E_\infty$-ring map $\MR\to\tmf_1(n)_\R$ (see~\cite{Mei23, Sen23}). We predict that such an~$E_\infty$ refinement exists.\footnote{Quinn--Zhu~\cite[Corollary~7.2.3\,(ii)]{QZ25} show this orientation admits an $E_\rho$ refinement, where $\rho$ is the regular representation of $\Z/2$.}

The following theorem is a combination of work of Hill--Meier~\cite{HM17} and Meier~\cite{Mei23}.
\begin{Theorem}\label{the_real_TMF_ors}
For all $n\ge 2$, there are $\Z/2$-$E_\infty$-ring spectra $\tmf_1(n)_\R$, $\Tmf_1(n)_\R$, and $\TMF_1(n)_\R$ whose underlying spectra are $\tmf_1(n)$, $\Tmf_1(n)$, and $\TMF_1(n)$ respectively. For $n = 3$, their $\Z/2$-fixed point spectra are $\tmf_0(3)$, $\Tmf_0(3)$, and $\TMF_0(3)$, respectively. For all $n\ge 2$, the $E_\infty$-ring spectrum maps
\begin{equation*}
 \MU[1/n]\overset{M(n)}{\longrightarrow} \tmf_1(n) \overset{}{\longrightarrow} \Tmf_1(n) \longrightarrow \TMF_1(n)
\end{equation*}
lift to $\Z/2$-ring maps,
\begin{equation*}
 \MR[1/n]\overset {M_\R(n)}{\longrightarrow} \tmf_1(n)_\R \overset{}{\longrightarrow} \Tmf_1(n)_\R \longrightarrow \TMF_1(n)_\R,
\end{equation*}
the last two of which are $E_\infty$.
\end{Theorem}

\begin{Definition}
Let $\MString_\R^h\coloneqq \MR\wedge\MString$, where $\MString$ is given the cofree $\Z/2$-$E_\infty$-ring structure arising from the trivial action.
\end{Definition}
Thus $\MString_\R^h$ is a $\Z/2$-$E_\infty$-ring spectrum whose underlying spectrum is $\MString^h$.
\begin{Theorem}\label{real_lift}
For all $n\ge 2$, there is a map of $\Z/2$-ring spectra
\begin{equation*}
 \sigma_1(n)_\R\colon \ \MString_\R^h[1/n]\longrightarrow \Tmf_1(n)_\R
\end{equation*}
which on underlying spectra is $\sigma_1(n)$ composed with the usual map $\tmf_1(n)\to\Tmf_1(n)$, and such that the real orientation $M_\R(n)\colon \MR[1/n]\to\Tmf_1(n)$ factors as a real orientation $v_\R\colon \MR\to\MString_\R^h$ followed by the usual orientation $\MString^h\to\Tmf_1(n)$. On underlying spectra, $v_\R$~is~$v$.
\end{Theorem}

\begin{proof}The proof strategy is the same as for Proposition~\ref{not_connective_or}. To adapt that proof, we need the following data:
\begin{enumerate}\itemsep=0pt
 \item A refinement of the complex orientation $\MU\to\Tmf_1(n)$ to a real orientation $\MR[1/n]\to\Tmf_1(n)_\R$, which is provided by Meier~\cite[Theorem~3.6]{Mei22} (here Theorem~\ref{the_real_TMF_ors}).
 \item An extension of the $E_\infty$-ring map $\sigma\colon \MString\to \tmf$ to a map between the respective cofree $\Z/2$-$E_\infty$-ring spectra associated to the trivial $\Z/2$-actions on $\MString$ and $\tmf$. By~\cite[Section~6.2.2]{BH15}, it suffices to show that $\sigma$ is equivariant for the trivial $\Z/2$-actions on its domain and codomain, which is trivially true.
 \item Lastly we need to refine the $E_\infty$-ring map $\tmf[1/n]\to\Tmf_1(n)$ to a $\Z/2$-$E_\infty$-ring map $\tmf[1/n]\to \Tmf_1(n)_\R$, where $\tmf[1/n]$ is cofree, induced from the trivial $\Z/2$-action. Without loss of generality, we may replace $\tmf[1/n]$ by $\Tmf[1/n]$, then precompose with the map $\tmf[1/n]\to\Tmf[1/n]$ (which refines to $\Z/2$-spectra in the same way as in the previous bullet point). Again by~\cite[Section~6.2.2]{BH15}, it suffices to show that $\Tmf[1/n]\to\Tmf_1(n)$ is $\Z/2$-equivariant for the trivial $\Z/2$-action on $\Tmf[1/n]$ and the $\Z/2$-action on $\Tmf_1(n)$ defined in~\cite[Section~4.1]{HM17}.

 For this, we return to the moduli of elliptic curves. The $\Z/2$-action on $\Tmf_1(n)$ is the map induced on global sections of $\mathcal O^{\mathrm{top}}$ from a $\Z/2$-action on (the log-\'etale site of) $\overline{\mathcal M}_1(n)[1/n]$, the Deligne--Mumford compactification of the moduli stack of elliptic curves $C$ with a chosen point $x$ of order $n$ (see Definition~\ref{1n_log_defn}). This $\Z/2$-action sends $(C, x)\mapsto (C, -x)$; therefore the map $\overline{\mathcal M}_1(n)[1/n]\to \overline{\mathcal M}[1/n]$ forgetting $x$ is $\Z/2$-equivariant with respect to the trivial action on $\overline{\mathcal M}[1/n]$. Taking sections of $\mathcal O^{\mathrm{top}}$, we obtain the usual map $\Tmf[1/n]\to\Tmf_1(n)$, together with the fact that it is $\Z/2$-equivariant.
\end{enumerate}
With these lifts in place, the construction of the map $\MString_\R^h\to \Tmf_1(n)_\R$ proceeds just as before. At the time of writing, the real orientation of $\Tmf_1(n)$ has not been refined to a $\Z/2$-$E_\infty$-map (see~\cite[Question~1.10]{Sen23}), so this construction is just a $\Z/2$-ring spectrum map.
\end{proof}

We would like to compare Theorem~\ref{real_lift} with Halladay--Kamel's real-equivariant lift of the Atiyah--Bott--Shapiro orientation. However, the constructions of $\MString_\R^h$ and $\MSpin_\R^c$ are difficult to relate, so we leave the comparison as a conjecture.
\begin{Lemma}[Hill--Meier~\cite{HM17}]\label{lambda_exists}
There is a $\Z/2$-$E_\infty$-ring map $\Lambda\colon (\Tmf_1(3)_\R)_{(2)}\to\KR_{(2)}$.
\end{Lemma}
Hill--Meier do not explicitly state Lemma~\ref{lambda_exists} in this form, but they provide all the pieces, so we show how to assemble those pieces into a proof.
\begin{proof}
As noted above, $\KR$ is cofree, and $\Tmf_1(3)_\R$ is also cofree~\cite[Section~4.1]{HM17}. Hill--Meier (ibid., Section~4.2), using a theorem of Hill--Lawson~\cite[Theorem~6.2]{HL16}, show that there is an $E_\infty$-map of nonequivariant spectra $\widetilde\Lambda\colon \Tmf_1(3)_{(2)}\to\KU_{(2)}$ which is equivariant for the $\Z/2$-actions on $\Tmf_1(3)_{(2)}$ and $\KU_{(2)}$. Blumberg--Hill~\cite[Section~6.2.2]{BH15} (see also~\cite[Theorem~2.4]{HM17}) show that if $\phi\colon R\to S$ is a $\Z/2$-equivariant map of $E_\infty$-ring spectra with respect to $\Z/2$-actions by $E_\infty$-ring maps on $R$ and $S$, then $\phi$ upgrades to a $\Z/2$-$E_\infty$-ring map on the cofree $\Z/2$-$E_\infty$-ring spectra built from~$R$ and~$S$. Applying this to $\widetilde\Lambda$, we conclude.
\end{proof}

\begin{ques}
Throughout this question, implicitly $2$-localize.

Does there exist a $\Z/2$-$E_\infty$-ring map $\MString_\R^h\to\MSpin_\R^c$ which on underlying spectra is the map $\MString^h\to\MSpin^c$ and such that the following diagram commutes
\begin{equation*}
\begin{tikzcd}
	\MR & {\MString_\R^h} && {\Tmf_1(n)_\R} \\
	& {\MSpin_\R^c} && \KR?
	\arrow["{v_\R}", from=1-1, to=1-2]
	\arrow["{u_\R}"', from=1-1, to=2-2]
	\arrow["{\sigma_1(n)_\R}", from=1-2, to=1-4]
	\arrow["{\exists?}", dashed, from=1-2, to=2-2]
	\arrow["\Lambda", from=1-4, to=2-4]
	\arrow["{\widehat A_\R}", from=2-2, to=2-4]
\end{tikzcd}\end{equation*}
\end{ques}
If such a map exists, it would also be nice to describe it geometrically, e.g., in terms of characteristic classes of $\Z/2$-equivariant vector bundles.

Another potential benefit of real-equivariance would arise by taking fixed points. Abdallah, Halladay, and Kamel~\cite{AHK26, AK25, HK24} study the fixed-point spectrum $(\MSpin_\R^c)^{\Z/2}$; it appears to be an unwieldy object, but Halladay--Kamel~\cite[Proposition~4.5]{HK24} show there is an $E_\infty$-ring map $\widetilde u\colon \MSpin\to (\MSpin_\R^c)^{\Z/2}$, so that one can form the composition
\begin{equation}\label{HK_KO}
 \MSpin\overset{\widetilde u}{\longrightarrow} (\MSpin_\R^c)^{\Z/2} \overset{(\widehat A_\R)^{\Z/2}}{\longrightarrow} (\KR)^{\Z/2}\simeq \KO,
\end{equation}
and (ibid., Corollary~6.14) this recovers the usual Atiyah--Bott--Shapiro orientation.

Because $(\Tmf_1(3)_\R)^{\Z/2}\simeq\Tmf_0(3)$~\cite[Section~4]{HM17}, one could try to generalize Halladay--Kamel's approach to orient $\Tmf_0(3)$. Ultimately because $\MR^{\Z/2}$ is complicated (though understood: see~\cite{GM17, HK01}), we expect $(\MString_\R^h)^{\Z/2}$ to not be easy to work with.
\begin{Proposition}\label{maybe_0n}
Let $\xi\colon B\to B\O$ be a tangential structure with two-out-of-three data such that there is an $E_\infty$-ring map $\psi\colon\mathit{M\xi}\to (\MString_\R^h)^{\Z/2}$. Then by forming a composition analogous to~\eqref{HK_KO}, there is a canonical orientation $\mathit{M\xi}[1/3]\to \Tmf_0(3)$.
\end{Proposition}
More generally, we would obtain orientations $M\xi[1/n]\to\Tmf_1(n)^{\Z/2}$.

The most na\"{\i}ve generalization of Halladay--Kamel's construction leads to $\xi = \String$, and the $\String$-orientation of $\Tmf_0(3)$ is not new information. It would be interesting to understand whether a more careful use of Proposition~\ref{maybe_0n} could be used to construct an orientation of $\Tmf_0(3)$ by some nontrivial $\MString$-algebra Thom spectrum, analogously to the $\String^h$-orientation of $\tmf_1(n)$. We note that Wilson~\cite[Corollary~4.16]{Wil15} has produced an orientation $\MSpin\ang{w_4}[1/3]\to\tmf_0(3)$ (hence also to $\Tmf_0(3)$), where $\Spin\ang{w_4}$ is the tangential structure which is a spin structure and a trivialization of the Stiefel--Whitney class $w_4$; we do not know whether Wilson's orientation factors through $(\MString_\R^h)^{\Z/2}[1/3]$ in Proposition~\ref{maybe_0n}.

\subsection[Low-degree homotopy groups of MString\^{}h]{Low-degree homotopy groups of $\boldsymbol{\MTString^h}$}
In this subsection, we compute low-dimensional \stringh\ bordism groups, and also calculate the effects of some of the orientations $\sigma_1(n)$ of the previous section on homotopy groups.

We will recall a few facts about the Brown--Peterson spectrum $\BP$ since many of the results in the remainder of this section utilize $\BP$ and its siblings $\BP\ang n$ in their proofs. $\BP$ is obtained by localizing $\MU$ at a prime $p$, and then $\MU_{(p)}$ is equivalent to a wedge sum of suspensions of~$\BP$~\cite{BP66}. This equivalence is not compatible with the ring structure, except on the lowest-degree $\BP$ summand.
\begin{Theorem}[{Basterra--Mandell~\cite[Theorem~1.1]{BM13}}]
The wedge-sum decomposition of $\MU_{(p)}$ into a sum of shifts of $\BP$ may be chosen so that the maps $\MU_{(p)}\leftrightarrows\BP$ splitting off the lowest-degree summand are $E_4$-algebra maps.
\end{Theorem}
\begin{Corollary}\label{E4_or}
Given an $E_\infty$-ring homomorphism $f\colon \MU\to E$, where $E$ is $p$-local, we can precompose with the $E_4$-ring homomorphism $\BP\to\MU_{(p)}$ to obtain an $E_4$-ring homomorphism $\widetilde f\colon \BP\to E$.
\end{Corollary}
It is not known whether one can strengthen this result to an $E_n$-splitting for some $n > 4$; Lawson~\cite[Remark~4.4.7]{Law18} and Senger~\cite[Theorem~1.3]{Sen17} have shown that it is not possible to do so for $n \ge 2(p+3)$.

The homotopy rings of $\BP$ and $\BP\ang n$ are
\begin{equation*}
\begin{aligned}
 &\BP_*\cong \Z_{(p)}[v_1, v_2,\dotsc],\qquad
 &\BP\ang n_*\cong \Z_{(p)}[v_1,v_2,\dotsc, v_n].
\end{aligned}
\end{equation*}
In both cases, $\abs{v_i} = 2(p^i-1)$. The map $\BP\to\BP\ang n$ sends $v_i\mapsto v_i$ for $1\le i\le n$ and sends $v_i\mapsto 0$ for $i > n$.

Now we introduce the main results of this subsection. Our first result is an analogue of a theorem of Hopkins--Mahowald (unpublished) and Devalapurkar~\cite{devalapurkar2019ando}, who showed that the Ando--Hopkins--Rezk orientation $\sigma\colon\MString\to\tmf$ is surjective on homotopy groups.
\begin{Proposition}\label{prop:surjhomotopy}
The following maps are surjective on homotopy groups:
\begin{equation*}
 \sigma_1(3)\colon \ \MString^h[1/3] \longrightarrow \tmf_1(3),\qquad
 \sigma_1(2)\colon \ \MString^h[1/2] \longrightarrow \tmf_1(2).
\end{equation*}
\end{Proposition}
We will prove this as a consequence of Propositions~\ref{cor:elevenconnected} and~\ref{prime2_7connected}.

\begin{proof}
We begin with $\sigma_1(3)$. In Proposition~\ref{prime2_7connected}, we will establish that the map $\widetilde\sigma_1(3)\colon \BP\wedge\MString\to\tmf_1(3)_{(2)}$ obtained from Corollary~\ref{E4_or} is $7$-connected. Lawson--Naumann~\cite[Theorem~1.1]{LN12a} show that the generators of the homotopy ring $(\tmf_1(3)_{(2)})_*$ are in degrees less than $7$, so since $\widetilde\sigma_1(3)$ is a map of ring spectra, it is surjective on homotopy groups in all degrees.

Since $\widetilde\sigma_1(3)$ factors through $\sigma_1(3)$, we conclude that after localizing at $2$, $\sigma_1(3)$ is surjective on homotopy groups. To finish, we argue we lift from $\Z_{(2)}$ to $\Z[1/3]$ by observing that $\tmf_1(3)_*$ is a polynomial ring over $\Z[1/3]$ on two generators in degrees $2$ and $6$, and we already observed that these generators are in the image of $\sigma_1(3)$ up to multiplication by a number prime to $6$. Thus it suffices to show that $\sigma_1(3)[1/6]$ hits the images of these generators in $\tmf_1(3)[1/6]_*$, which can easily be checked with the Atiyah--Hirzebruch spectral sequence because $\MString^h\simeq \MU\wedge\MString$ and $\tmf_1(3)$ lack torsion once $6$ is inverted.

The argument for $\tmf_1(2)$ is essentially the same, except using Proposition~\ref{cor:elevenconnected}, which says that $\widetilde\sigma_1(2)$ is $11$-connected, and the fact that $\tmf_1(2)_*$ is a polynomial ring over $\Z[1/2]$ with generators in degrees below $9$.\footnote{This fact appears in Hill~\cite{Hil07}, where it is attributed to Hopkins--Mahowald (unpublished) and Behrens~\cite{Beh06}.}
\end{proof}

It would be interesting to generalize this to $\tmf_1(n)$ for $n > 3$.

\begin{Theorem}\label{MU_tmf}There is a ring isomorphism
\begin{equation*}
 \pi_*(\tmf\wedge\MU)\overset\cong\longrightarrow
 \Z[a_1, a_2, a_3, a_4, a_6, e_n\mid n\ge 4],
\end{equation*}
where $\abs{a_n} = 2n$ and $\abs{e_n} = 2n$.
\end{Theorem}
\begin{Remark}
The history of Theorem~\ref{MU_tmf} is a little complicated. Rezk~\cite[Proposition 20.4]{Rez07} states Theorem~\ref{MU_tmf} without proof, and Hopkins--Mahowald~\cite{HopkinsMiller2014} and Bauer~\cite{bauer2008computation} use it implicitly in their discussion of the Adams--Novikov spectral sequence for $\tmf$. Mathew~\cite[Corollary~5.2]{mathew2013homology} proves Theorem~\ref{MU_tmf} conditional on the ``gap theorem'', whose first complete proof was given by Carrick--Davies--van Nigtevecht~\cite[Theorem~A]{CDvN24}.
\end{Remark}
\begin{Corollary}\label{bordism_ring}
There is a ring isomorphism
\begin{equation*}%\label{bordisom}
 \Omega_*^{\String^h} \overset\cong\longrightarrow \Z[x_2, x_4, x_6, x_8, y_8, x_{10}, x_{12}, y_{12}, x_{14}, \dotsc]/(\dotsb),
\end{equation*}
where $\abs{x_i} = \abs{y_i} = i$ and all generators and relations not listed are in degrees~$16$ and above.
\end{Corollary}
\begin{proof}
The Ando--Hopkins--Rezk map $\MString\to\tmf$ is a $15$-connected ring homomorphism \cite[Theorem~2.1]{Hil09}, so when evaluated on a connective ring spectrum such as $\MU$, it is a ring homomorphism which is an isomorphism in degrees~$15$ and below.
\end{proof}

\begin{Remark}\label{5tors}
Localized at a prime $p\ge 5$, we can do better: $\bigl(\Omega_*^{\String^h}\bigr)_{(p)}$ is torsion-free in all degrees. This is because $p$-locally, both $\MTString$ and $\MU$ split as sums of shifts of~$\BP$.\footnote{For $\MU$, this is a theorem of Brown--Peterson~\cite[Theorem~1.3]{BP66}; for $\MTString$ one combines Brown--Peterson's theorem with a calculation due to Giambalvo~\cite[Corollary~1]{Gia69}.} Therefore
\begin{equation*}
 \MTString^h_{(p)} \simeq \bigvee_i\Sigma^{n_i}\BP\wedge \BP.
\end{equation*}
Quillen~\cite{Qui69} showed that $\pi_*(\BP\wedge\BP)$ is torsion-free (see also \cite[Theorem~3.11]{Wil82}). Therefore, $\pi_*(\MTString^h)_{(p)}$ is also torsion-free.
\end{Remark}
\begin{Remark}
One can construct manifold representatives for some of the generators in Corollary~\ref{bordism_ring} by studying the images of the maps from $\Omega_*^\String$ and $\Omega_*^\U$ to $\Omega_*^{\String^h}$. Specifically,
\begin{itemize}\itemsep=0pt
 \item The map from $\MU_*$ hits $x_2$, $x_4$, $x_6$, $x_{12}$ and $x_{14}$ for which the first three generators have descriptions as $\CP^1$, $\CP^2$, and a Milnor hypersurface $H_{22}\amalg\CP^2$.
 \item The map from $\MTString_*$ rationally hits $y_8$ and $y_{12}$.
\end{itemize}
\end{Remark}

\begin{Proposition}\label{cor:elevenconnected}
The $E_4$-orientation $\widetilde\sigma_1(2)\colon \MTString\wedge \BP\to(\tmf_1(2))_{(3)}$\footnote{There is a homotopy equivalence $\tmf_0(2)\simeq\tmf_1(2)$; we mostly refer to this object as $\tmf_1(2)$ to streamline our notation, but it is often called $\tmf_0(2)$ in the literature.} obtained by smashing the $E_4$-orientation $\BP\to\tmf_1(2)_{(3)}$ from Corollary~{\rm \ref{E4_or}} with the string orientation of $\tmf_1(2)$ is $11$-connected.
\end{Proposition}
We will prove this using the Baker--Lazarev Adams spectral sequence~\cite{BL01}; for $3$-local $\tmf$-homology specifically, this spectral sequence was developed by Henriques and Hill (see~\cite{douglas2014topological, Hil07}), building on work of Behrens~\cite{Beh06} and Hopkins--Mahowald (unpublished).

For a spectrum $X$, this spectral sequence has the signature
\begin{equation*}
 E_2^{s,t} = \Ext_{\cA^{\tmf}}(H^*(X;\Z/3), \Z/3) \Longrightarrow \tmf_{t-s}(X)_3^\wedge,
\end{equation*}
where
\begin{equation*}%\label{Atmf}
 \cA^\tmf\coloneqq \Z/3\bigl\langle \beta, \cP^1\bigr\rangle/\bigl(\beta^2, \bigl(\cP^1\bigr)^3, \beta\bigl(\cP^1\bigr)^2\beta - \bigl(\beta\cP^1\bigr)^2 - \bigl(\cP^1\beta\bigr)^2\bigr),
\end{equation*}
with a $\Z$-grading specified on $\cA^\tmf$ by $\abs\beta = 1$ and $\abs{\cP^1} = 4$. The $\cA^\tmf$-action on $H^*(X;\Z/3)$ is specified by having $\beta$ act as the Bockstein for the short exact sequence $0\to\Z/3\to\Z/9\to\Z/3\to 0$ and $\cP^1$ act as the first Steenrod power.\footnote{In other words, we have specified the $\cA^\tmf$-action by defining an algebra homomorphism from $\cA^\tmf$ to the mod $3$ Steenrod algebra. This homomorphism is \textit{not} injective! See Henriques~\cite[Section~13.3]{douglas2014topological}.} See~\cite{BDDM24, BR21, Debray:2023tdd,Hil07, Hil09, TZ24} for additional computations with this spectral sequence.
\begin{Lemma}\label{BP_homology_tmf}
Define the $\cA^\tmf$-module $N_3\coloneqq \cA^\tmf/(\beta)$, so that $N_3\cong \Z/3\bigl[\cP^1\bigr]/\bigl(\bigl(\cP^1\bigr)^3\bigr)$. Then there is an $\cA^\tmf$-module isomorphism
\begin{equation}\label{BP_Atmf}
H^*(\BP;\Z/3)\cong \textcolor{BrickRed}{N_3} \oplus
 \textcolor{MidnightBlue}{\Sigma^{12}N_3} \oplus P,
\end{equation}
where $P$ is concentrated in degrees $16$ and above.
\end{Lemma}
Thus we can (and do) ignore $P$.
\begin{proof}
Recall that $H^*(\BP;\Z/3)\cong\Z/3\bigl[\cP^1, \cP^2, \dotsc\bigr]$. Therefore, $\beta\in\cA^\tmf$ acts trivially on $H^*(\BP;\Z/3)$, and we can determine the $\cP^1$-action using the Adem relations, which in this case simplify to
\begin{equation*}
 \cP^1\cP^n = (n+1)\cP^{n+1}
\end{equation*}
for $n\ge 0$.
\end{proof}

In~\eqref{BP_Atmf}, $\textcolor{BrickRed}{N_3}$ is spanned by $\bigl\{1, \cP^1, \cP^2\bigr\}$ and $\textcolor{MidnightBlue}{\Sigma^{12}N_3}$ is spanned by $\bigl\{\cP^3, \cP^4, \cP^5\bigr\}$.
\begin{proof}[Proof of Proposition~\ref{cor:elevenconnected}]
By construction $\widetilde\sigma_1(2)$ factors as a composition of $\MTString\wedge \BP\to\tmf\wedge\BP$, which is $15$-connected, followed by $\overline\sigma_1(2)\colon \BP\wedge\tmf\to\tmf_1(2)$, and both are $E_4$-ring spectrum maps. Therefore, it suffices to show $\overline\sigma_1(2)$ is $11$-connected.

On homotopy groups, $\overline\sigma_1(2)$ is a ring homomorphism, so it sends $1\mapsto 1$ and therefore is an isomorphism on $\pi_0$. Because $\overline\sigma_1(2)$ is an isomorphism on $\pi_0$, then $\overline\sigma_1(2)^*\colon H_\tmf^0(\tmf_1(2))\to H_\tmf^0(\tmf\wedge\BP)$ is also an isomorphism. By~\eqref{BP_Atmf}, $H_\tmf^0(\tmf\wedge\BP)$ splits as $\cA^\tmf$-modules as the sum of $N_3$ and an $11$-connected summand, and by~\cite[Theorem~4.13]{mathew2013homology}, there is a $\tmf_{(3)}$-module equivalence $(\tmf_1(2))_{(3)}\simeq\tmf_{(3)}\wedge Y$ for a spectrum $Y$ with $H^*(Y;\Z/3)\cong N_3$, so $H_\tmf^*(\tmf_1(2))\cong N_3$. Thus since $\overline\sigma_1(2)^*$ is an $\cA^\tmf$-module map which is an isomorphism on~$H^0$, it must map the $N_3$ summand in $H_\tmf^*(\tmf_1(2))$ isomorphically onto the~$N_3$ summand from $\tmf\wedge\BP$. Then by the Baker--Lazarev Adams spectral sequence, the map is 11-connected.
\end{proof}

\begin{Remark}
Because $\pi_2(\MTString^h)\cong\Z$ but $\pi_2(\tmf_1(2))\cong\Z[1/2]$, $\sigma_1(2)$ is not $11$-connected. But Proposition~\ref{cor:elevenconnected} implies that localized at $3$, $\sigma_1(2)$ is surjective on homotopy, because $\pi_*(\MTString^h)_{(3)}$ surjects onto $\pi_*(\BP\wedge\tmf)$ and $\overline\sigma_1(n)$ hits the generators of the homotopy ring of $\tmf_1(2)$.
\end{Remark}
\begin{Proposition}\label{prime2_7connected}
The map $\overline\sigma_1(3)\colon \BP\wedge\MTString\to \tmf_1(3)_{(2)}$ is $7$-connected.
\end{Proposition}
\begin{proof} The proof is almost exactly the same as that of Proposition~\ref{cor:elevenconnected}. What makes that proof work is that the quotients of $H_\tmf^*(\tmf\wedge\BP)$ and $H_\tmf^*(\tmf_1(2))$ by all classes in degrees~$12$ and above are isomorphic, cyclic $\cA^\tmf$-modules on a generator in degree $0$, so that we could lift an evidently $0$-connected map to an isomorphism on $H_\tmf^*$ in degrees~$11$ and below.

For the rest of the proof, we work at $p = 2$. The analogue of $\cA^\tmf$ is $\cA(2)$, the subalgebra of the $2$-primary Steenrod algebra generated by $\Sq^1$, $\Sq^2$, and~$\Sq^4$. Henriques showed that~$\cA(2)$ plays the analogous role in the $2$-primary Baker--Lazarev Adams spectral sequence as $\cA^\tmf$ does at $p = 3$~\cite{douglas2014topological}.
Therefore, it suffices to show that $H_\tmf^*(\BP\wedge\tmf)$ and $H_\tmf^*(\tmf_1(3))$ (this time with $\Z/2$ coefficients, not~$\Z/3$ coefficients like in the previous paragraph), quotiented by all classes in degrees $8$ and above, are isomorphic cyclic $\cA(2)$-modules on a generator in degree~$0$. For $\tmf_1(3)$, there is a $\tmf$-module equivalence $\tmf_1(3)_{(2)}\simeq \tmf\wedge DA(1)$ for a~spectrum~$DA(1)$ with $H^*(DA(1);\Z/2)\cong\cA(2)/\bigl(\Sq^1\bigr)$~\cite[Theorem~1.2]{mathew2013homology}. For $\BP\wedge\tmf$, we need to compute $H^*(\BP;\Z/2)$ as an $\cA(2)$-module. Brown--Peterson~\cite[Corollary~1.2]{BP66} show that $H^*(\BP;\Z/2)\cong\cA/\bigl(\Sq^1\bigr)$, so we are done by the observation that the inclusion $\cA(2)\to\cA$ is $7$-connected, so that $\cA(2)/\bigl(\Sq^1\bigr)\to\cA/\bigl(\Sq^1\bigr)$ is $7$-connected (since $\Sq^8$ is the lowest-degree Steenrod square not contained in~$\cA(2)$).
\end{proof}

\begin{Lemma}[{Mathew~\cite[Section~5.2]{mathew2013homology}}]\label{thm_tmf_BP}
There is a $\tmf$-module $M$ and a $2$-local $\tmf$-module equivalence
\begin{equation}\label{eq:2equiv}
 \MU\wedge\tmf\overset\simeq\longrightarrow \tmf_1(3)\vee M.
\end{equation}
\end{Lemma}

Mathew does not state this explicitly; instead, he shows (ibid., Corollary 5.7) that there is a map $DA(1)\to\MU$ and an ideal $I$ in $\MU_*(\tmf)$ such that the composition
\begin{equation*}
 \tmf\wedge DA(1) \longrightarrow\tmf\wedge\MU\longrightarrow (\tmf\wedge\MU)/I
\end{equation*}
is a $\tmf$-module equivalence. Moreover, there is a $2$-local $\tmf$-module equivalence $\tmf\wedge DA(1)\simeq\tmf_1(3)$~\cite[Theorem~1.2]{HM14} (see also~\cite[Theorem~5.8]{mathew2013homology}).

\subsection[Does sigma\_1(n) split?]{Does $\boldsymbol{\sigma_1(n)}$ split?}\label{subsection:splitting}

Anderson--Brown--Peterson~\cite{ABP67} showed that the Atiyah--Bott--Shapiro maps~\cite{ABS64} $\MSpin\to\ko$ and $\MSpin^c\to\ku$ admit $2$-local sections $\ko\to\MSpin$, resp.\ $\ku\to\MSpin^c$, and used these sections, along with higher-degree analogues, to effectively determine spin and \spinc bordism. The analogous question for the Ando--Hopkins--Rezk orientation~\cite{AHR10} $\sigma\colon\MString\to\tmf$ is a~longstanding open question in homotopy theory, discussed for example in~\cite{ Abs21,devalapurkar2019ando, Dev24,Lau04, Lau16, LO16, LO18, LS19,MG95, MH02,MR09, Tok24}. It therefore seems reasonable to ask the following.
\begin{ques}\label{splitting_ques}
Let $p = 2$ or $3$ and $p\nmid n$. Does the map $\sigma_1(n)\colon \MTString^h_{(p)}\to \tmf_1(n)_{(p)}$ have a section? What about Devalapurkar's orientation $\sigma_D$?
\end{ques}
One could also ask this question localized at a large prime (i.e., $p\ge 5$), where it is much easier, as both $\MString^h$ and $\tmf_1(n)$ for many $n$ are known to decompose into sums of shifts of $\BP\wedge\BP$, resp.\ $\BP$ (see the proof of Remark~\ref{5tors}, resp.\ \cite[Section~5]{Mei23}). Thus we focus on the harder primes. We think for $p = 2$, $n = 3$, and for $p = 3$, $n = 2$, Question~\ref{splitting_ques} has an affirmative answer.

Question~\ref{splitting_ques} passes a few basic checks.
\begin{Proposition}[Devalapurkar]
If there is a section $s\colon \tmf_{(2)}\to\MString_{(2)}$ of $\sigma$, then there is a section $s'\colon\tmf_1(3)_{(2)}\to\MString^h_{(2)}$ of $\sigma_1(3)$.
\end{Proposition}
\begin{proof}
This is immediate since $\MU \wedge \tmf$ would split off from $\MTString^h$, and $\tmf_1(3)$ itself splits off of $\MU \wedge \tmf$ by Lemma~\ref{thm_tmf_BP}.
\end{proof}

In addition, an affirmative answer to Question~\ref{splitting_ques} would imply that $\sigma_1(n)$ is surjective on homotopy after $p$-completion. For $(p,n)\in\set{(2,3), (3,2)}$, we proved homotopy surjectivity unconditionally in Proposition~\ref{prop:surjhomotopy}.

\begin{Proposition}\label{notBPmod}
There is no $2$-local section of $\sigma_1(3)$ that is a map of $\BP$-module spectra, where $\MString^h_{(2)}$ acquires its $\BP$-module structure from the $E_4$-map $\BP\to\MU_{(2)}\to\MString^h_{(2)}$ and $\tmf_1(3)_{(2)}$ acquires its $\BP$-module structure from the equivalence $\tmf_1(3)_{(2)}\simeq \BP\ang 2$.
\end{Proposition}

\begin{proof}
Suppose such a section existed and rationalize. That would imply the existence of a~section of
\begin{equation*}
 \sigma_1(3)_*\colon \ \pi_*(\MString^h)\otimes\Q \longrightarrow \tmf_1(3)_*\otimes\Q
\end{equation*}
which is linear with respect to $\BP_*\otimes\Q\cong\Q[v_1,v_2,\dotsc]$. Since
\begin{equation*}
 \pi_*(\MString^h)\otimes\Q\cong\pi_*(\MString)\otimes\Q\otimes \pi_*(\MU)
\end{equation*}
and $\pi_*(\MU)\otimes\Q$ is a free $\pi_*(\BP)\otimes_{\Z_{(2)}}\Q$-module and $\pi_*(\MString)\otimes\Q$ is a polynomial algebra, $v_3$ acts injectively on $\pi_*(\MString^h\otimes\Q)$. However, since $\tmf_1(3)_{(2)}$ acquired its $\BP$-module structure by being a form of $\BP\ang 2$, $v_3$ acts as zero on $\tmf_1(3)_*\otimes\Q$. A section equivariant for the action of $v_3$ cannot carry a zero action to an injective action.
\end{proof}

\section[String\^{}h and the Diaconescu--Moore--Witten anomaly]{String$\boldsymbol{{}^h}$ and the Diaconescu--Moore--Witten anomaly}\label{section:stringhDMW}

We will now explain an application of \stringh\ structures to type IIA string theory by understanding their relationship with the Diaconescu--Moore--Witten anomaly. Let $X$ be a 10-dimensional manifold which serves as the target space for type IIA string theory. The intimate way in which type IIA and M-theory are related means that the same anomaly also manifests in M-theory on $Y=X\times S^1$, and is in a sense where it originates. In particular, the partition function for the RR-fluxes in type IIA string theory can be matched with the corresponding partition function computed in M-theory, and we review how an anomaly arises in M-theory by looking at a certain part of its partition function. A~priori the way in which the anomalies arise in M-theory and type IIA are different, but it was shown in~\cite{DMW:E8} that the anomaly cancellation information is equivalent. Therefore, whatever criterion on the structure of spacetime is imposed by anomaly cancellation is shared by both the M-theory and the type IIA target space.

M-theory has two types of branes: the M2 and M5 branes. On the M2 brane, there is an associated 3-form field $C$ with field strength $G = dC$. The topological quantization of $G$ is given by choosing any element $\alpha \in H^4(X;\Z)$ and letting $G_{\alpha}$ be the ``mode'' of $G$ contributing to the topological sector labeled by $\alpha$. From this, we can form the partition function of M-theory by considering the contributions from all $\alpha$.

\begin{Definition}[{Diaconescu--Moore--Witten~\cite[Section~5]{DMW:2000}}]
\label{f_defn}
Let $X$ be a closed spin $10$-manifold and $\alpha\in H^4(X;\Z)$. Assume that there is a class $x\in\ku^4(X)$ such that $\tau_0(x) = \alpha$; as we observed in the proof of Theorem~\ref{defns_are_equiv}, $x$ may be represented by an $\SU$-bundle, and because $B\SU_5\to B\SU$ is $11$-connected, $x$ may be represented by a rank-$5$ complex vector bundle $E\to X$ with $\SU$-structure. Then define
\begin{equation}\label{three_indices}
 f(\alpha) \coloneqq q(E\otimes\overline E) + \bigl((\mathrm{Ind}\bigl(\Lambda^2(E)\bigr) + \mathrm{Ind}(E))\bmod 2\bigr)\in\Z/2,
\end{equation}
where $\mathrm{Ind}(V)$ denotes the index of the Dirac operator coupled to $V$ and $q$ denotes the mod $2$ index of this Dirac operator.
\end{Definition}
\begin{Definition}
Let $X$ be a closed spin $10$-manifold. The partition function of the topological sector of M-theory on $X$ is given by the following sum over $\alpha \in H^4(X;\Z)$ for $G$:\footnote{We write $\sim$ rather than $=$ because of some prefactors that are gauge-invariant and thus not relevant for the present discussion. See~\cite[Section~3]{DMW:2000} for more on these terms.}
\begin{equation}\label{eq:ZMtheory}
 \mathcal{Z}_{M}\sim \sum_{\alpha \in H^4(X;\Z)} (-1)^{f(\alpha)}\exp\bigl(-|G_\alpha|^2\bigr) ,
\end{equation}
where $|G_\alpha|^2 = \int G_\alpha \wedge {\star}\, G_\alpha$, and we are assuming the $\ku^4$ lifts chosen in Definition~\ref{f_defn} in the definition of $f$ exist.
\end{Definition}

Even though the topological sector is not the full partition function of M-theory, it can already give hints at anomalies, in particular by studying the ambiguity of~\eqref{eq:ZMtheory} with respect to the existence of lifts across $\tau_0$.
\begin{Definition}\label{defn:inteSW}
The \textit{$k^{\mathrm{th}}$ integral Stiefel--Whitney class} $W_k\in H^k(B\O;\Z)$ is $W_k\coloneqq\square_\Z(w_{k-1})$.
\end{Definition}
Thus, for example, a \spinc structure on an oriented vector bundle $V$ is equivalent data to a~trivialization of $W_3(V)$.
\begin{Proposition}[Diaconescu--Moore--Witten~\cite{DMW:E8}]
\label{f_well_defined}
With $X$ as above, in particular spin, given data of a trivialization of $W_7(X)$ the quantity $f(\alpha)$ from Definition~{\rm \ref{f_defn}} is well-defined for all $\alpha\in H^4(X;\Z)$: each~$\alpha$ has a $\ku$-cohomology lift, and the quantity~\eqref{three_indices} does not depend on the choice of lift.
\end{Proposition}
Proposition~\ref{f_well_defined} has the physics consequence that the partition function of the topological sector of type IIA string theory on $X$, which in general suffers a sign ambiguity, is well-defined when $W_7(X)$ is trivialized. More heuristically speaking, in order for $X$ to be a valid background of type IIA string theory, either we must choose a trivialization of $W_7(X)$ or wrap branes within submanifolds of $X$. In this paper we will consider the first option.

Mathematically, Proposition~\ref{f_well_defined} tells us that a $10$-manifold with spin structure and a trivialization of $W_7$ admits a \stringh\ structure.

\begin{Remark}\label{IIA_spinc}
It is possible to generalize this story to the case when $X$ is merely \spinc. In this case, the relation to M-theory is slightly changed: if $L\to X$ is the determinant line bundle of the \spinc structure, then the total space of the unit sphere bundle~$S(L)$ is a closed $11$-manifold with a canonical spin structure induced from the \spinc structure on~$X$, and one thinks of type~IIA string theory on $X$ as a ``twisted compactification'' of M-theory on~$S(L)$. In this setting, there is a generalization of Proposition~\ref{f_well_defined} implying that on a closed \spinc $10$-manifold $X$, the data of a trivialization of $W_7(X)$ suffices to resolve the sign ambiguities in the partition function of the topological sector of type IIA string theory on~$X$.
\end{Remark}

\begin{Remark}
Let us suppose that the target space of type IIA string theory has a string$^h$ structure. As a first level consistency check, we recall that a string$^h$ structure induces a spin$^c$ structure. This is consistent with the fact that the target space of type IIA string theory has a spin$^c$ structure, as we just discussed in Remark~\ref{IIA_spinc}. We exhibit the spin$^c$ structure on type IIA by observing the transformation of the gravitino field $\Psi$ in the low energy supergravity. This is a fermion that is charged under a $\U_1$-gauge symmetry, where the $\U_1$-bundle arises from dimensionally reducing away the M-theory circle. A gauge transformation $\Psi \rightarrow e^{2 \pi i q} \Psi$ reflects the spin$^c$ structure if $q$ is half integral, and it was shown in \cite{BEM:2003vb, Duff:1998us} that this is the case.
\end{Remark}

\begin{Definition}\label{DMW_str}
 The \textit{Diaconescu--Moore--Witten anomaly cancellation condition} is the requirement $W_7(X) = 0$. A \textit{Diaconescu--Moore--Witten $($DMW$)$ structure} on a vector bundle $V$ is a~\spinc structure and a trivialization of $W_7(V)$.
\end{Definition}

See \cite{Fiorenza:2019usl,Sati:2021uhj} for more on how generalized cohomology theories can be applied to M-theory from the DMW anomaly, and~\cite{Freed:2019sco} for more on the tangential structure of M-theory when time-reversal is taken into account.

The following serves as a sketch of the argument given in \cite{DMW:2000} for their anomaly cancellation condition. We start by unpacking the conditions on the function $f(\alpha)$. In particular, $f(\alpha)$ satisfies the property that
\begin{equation*}
 f(\alpha+\alpha') = f(\alpha)+f(\alpha')+\int_{X} \alpha \Sq^2(\alpha'\bmod 2) ,
\end{equation*}
and $(X,\alpha)\mapsto f(\alpha)$ is a bordism invariant in $\Hom\bigl(\Omega^{\Spin}_{10}(K(\Z,4)),\Z/2\bigr)$. Suppose $\gamma\in H^4(X;\Z)$ is torsion. Then, while $|G_\alpha|^2$ is invariant under $\alpha \to \alpha +\gamma$, $f(\alpha)$ is often not invariant. Consider a specific transformation $\alpha \to \alpha +2 \gamma$; then,
\begin{equation}\label{2g_shift}
 f(\alpha+2\gamma) = f(\alpha)+f(2\gamma)+\int_X \alpha \Sq^2(2\gamma\bmod 2),
\end{equation}
but $2\gamma\bmod 2 = 0$ so we only need to consider the new term $f(2\gamma)$. Expanding again, we see
\begin{equation*}
 f(2\gamma) = f(\gamma)+f(\gamma)+\int_{X} \gamma \Sq^2(\gamma\bmod 2) .
\end{equation*}
Stong~\cite{stong} shows that $\int_X \gamma\Sq^2(\gamma) = \Sq^4\bigl(\Sq^2(\gamma)\bigr)$; combining this with the Wu formula, Diaco\-nescu--Moore--Witten~\cite[Section~6]{DMW:2000} show that
\begin{equation*}
 \int_X \gamma\Sq^2(\gamma\bmod 2) = \int_X \gamma\Sq^2(\lambda(X)\bmod 2).
\end{equation*}
The effect of~\eqref{2g_shift} in the partition function is given by $(-1)^{f(\alpha) + f(2\gamma)}$. As $2f(\gamma) = 0$, we only have to worry about $\int_X \gamma \Sq^2(\gamma\bmod 2)$.
\begin{Lemma}[{\cite[Section~6]{DMW:2000}}]
Let $\ang{\bl,\bl}\colon\mathrm{Tors}\bigl(H^4(X;\Z)\bigr)\otimes \mathrm{Tors}\bigl(H^7(X;\Z)\bigr)\to\Q/\Z$ denote the torsion pairing on a closed spin $10$-manifold $X$. Then for all $\gamma\in\mathrm{Tors}\big(H^4(X;\Z)\bigr)$,
\begin{equation}\label{tors_not}
 \frac 12 \int_X \gamma \cdot \Sq^2(\lambda(X)\bmod 2)) = \bigl\langle \gamma, \square_\Z\bigl(\Sq^2(\lambda(X))\bigr)\bigr\rangle.
\end{equation}
\end{Lemma}
The equality~\eqref{tors_not} takes place in $\Q/\Z$: for the left-hand side, the integral is an element of~$\Z/2$, so dividing by $2$ we obtain an element of $\bigl(\tfrac 12 \Z\bigr)/\Z$, which is a subgroup of $\Q/\Z$.

Finally, a direct calculation with the Wu formula and the relation $\lambda(X)\bmod 2 = w_4(X)$ shows
\begin{equation*}
 \square_\Z\bigl(\Sq^2(\lambda(X)\bmod 2)\bigr) = W_7(X),
\end{equation*}
so the DMW condition $W_7(X) = 0$ fixes the sign of the partition function unambiguously.\footnote{In conversation with Moore, it became known to the authors that the theory could be consistent even if the~$W_7$ anomaly is non-trivial. If the partition function vanishes, that does not inherently mean that the theory itself is invalid. We plan to return to theories with nontrivial values of this anomaly in future work.}

\subsection[Relating Diaconescu--Moore--Witten anomaly cancellation with string\^{}h]{Relating Diaconescu--Moore--Witten anomaly cancellation with $\boldsymbol{{\rm string}^h}$} \label{subsection:DMWStringh}

 We will show how a \stringh\ structure induces a Diaconescu--Moore--Witten structure.

\begin{Theorem}\label{prop:strinhtoW7}
Let $V\to X$ be a \stringh\ vector bundle. Then $W_7(V)$ admits a canonical trivialization.
\end{Theorem}
Thus the Diaconescu--Moore--Witten anomaly cancellation condition (Definition~\ref{DMW_str}) is automatically satisfied on \stringh\ $10$-manifolds.

\begin{proof}
We use the characterization of \stringh\ structure from Definition~\ref{triv_stringh_defn}: that we have trivialized $\square_\ku(\lambda^c(V))$. That is, we think of the class $\square_\ku(\lambda^c(V))$ as a map
\begin{equation}\label{oops_not_spectra}
 X \xrightarrow V B\Spin^c \xrightarrow{\lambda^c} K(\Z, 4) \xrightarrow{\square_\ku} \Omega B\U\ang 8,
\end{equation}
and we have supplied a null-homotopy of this map.
To relate Definition~\ref{triv_stringh_defn} to the $W_7(X)=0$ anomaly cancellation condition of Diaconescu--Moore--Witten, we observe the following commutative diagram, whose rows are cofiber sequences:
\begin{subequations}
 \begin{equation} \label{cofib_seq}
 \begin{tikzcd}
 \Sigma^2 \ku \arrow[r,"\beta"] \arrow[d, "\tau_{\le 2}"] & \ku \arrow[r,"\tau_0"] \arrow[d,"\tau_{\leq 2}"]& H\Z \arrow[r,"\Box_{\ku}"] \arrow[d, Rightarrow, no head]& \Sigma^3 \ku \arrow[d,"\tau_0"]\\
\Sigma^2 H\Z \arrow[r] &\tau_{\leq 2}\ku \arrow[r] & H\Z \arrow[r,"\Box_{\Z} \Sq^2\rho_2"]& \Sigma^3 H \Z \,.
 \end{tikzcd}
 \end{equation}
Here $\rho_2$ is reduction modulo $2$.
The top map builds $\ku$ as an extension with $\Box_{\ku}$ as the $k$-invariant, $\beta$ the Bott map, and the map down is the truncation map.
This builds $\tau_{\leq 2}\ku$ also as an extension with $k$-invariant $\Box_{\Z} \Sq^2$.

Functoriality of Postnikov truncation means that~\eqref{cofib_seq} provides data of a contractible space of homotopies in each square that make each square commute. Choose one of them for the rightmost square, and call it $H$. Then apply $\Omega^\infty \Sigma^4$ to the rightmost square of~\eqref{cofib_seq}:
\begin{equation}\label{cofib_seq_2}
\begin{tikzcd}
	{K(\Z, 4)} & {\Omega B\U\ang 8} \\
	& {K(\Z, 7).}
	\arrow["{\square_\ku}", from=1-1, to=1-2]
	\arrow[""{name=0, anchor=center, inner sep=0}, "{\square_\Z\Sq^2\rho_2}"', from=1-1, to=2-2]
	\arrow["{\tau_{\le 7}}", from=1-2, to=2-2]
	\arrow["H", between={0}{0.8}, Rightarrow, from=1-2, to=0]
\end{tikzcd}
\end{equation}
\end{subequations}
Now paste together~\eqref{oops_not_spectra} and~\eqref{cofib_seq_2} to obtain a commutative diagram as follows:
\begin{equation*}\begin{tikzcd}
	{X} & {B\Spin^c} & {K(\Z, 4)} & {\Omega B\U\ang 8} \\
	&&& {K(\Z, 7)}
	\arrow["V", from=1-1, to=1-2]
	\arrow["{\lambda^c}", from=1-2, to=1-3]
	\arrow["{\square_\ku}", from=1-3, to=1-4]
	\arrow[""{name=0, anchor=center, inner sep=0}, "{\square_\Z \Sq^2\rho_2}"', from=1-3, to=2-4]
	\arrow["{\tau_{\le 7}}", from=1-4, to=2-4]
 \arrow["H", between={0}{0.8}, Rightarrow, from=1-4, to=0]
\end{tikzcd}\end{equation*}
The null-homotopy of $\square_\ku\circ\lambda^c\circ V$ induces a null-homotopy of $\tau_{\le 7}\circ\square_\ku\circ\lambda^c\circ V$. Composing with the homotopy $H$ specified above, we have a null-homotopy of $\square\circ\Sq^2\circ\rho_2\circ\lambda^c\circ V$. It therefore suffices to identify this composition with $W_7(V)\colon X\to K(\Z, 7)$. By Definition~\ref{defn:inteSW}, it suffices to identify $\Sq^2\circ\rho_2\circ\lambda^c$ with $w_6\colon B\Spin^c\to K(\Z/2, 6)$. By Remark~\ref{lambda_c_mod_2}, we have an identification $\rho_2\circ\lambda^c\simeq w_4 + w_2^2$ as maps $B\Spin^c\to K(\Z/2, 4)$; the Wu formula then furnishes the last identification.
\end{proof}
\begin{Remark}
Fiorenza--Sati--Schreiber~\cite[Section~3.2]{Fiorenza:2019usl} derive the Diaconescu--Moore--Witten condition $W_7 = 0$ in another way, starting from their Hypothesis~H: that the $C$-field in M-theory is quantized in twisted cohomotopy, rather than twisted cohomology (see~\cite{Fiorenza:2019usl} for more on Hypothesis~H and~\cite[Section~4.2]{SS25} for an introduction). There are interesting parallels to the story presented here: rather than arising from a \stringh\ structure, Fiorenza--Sati--Schreiber obtain the condition $W_7 = 0$ from a twisted $\mathrm{Sp}$-structure giving rise to the needed twist of cohomotopy. Intriguingly, the appearances of cohomotopy in their work and $\tmf$ in ours may be related: because the Hurewicz map $\mathbb S\to\tmf$ is $6$-connected, these two spectra are not easy to distinguish in appearances in string theory. We thank Urs Schreiber for a helpful discussion regarding this comparison.
\end{Remark}

While a string$^h$ structure always trivializes $W_7$, one can also ask the reverse question: given a manifold with DMW structure, must this lift to a \stringh\ structure? Is such a lift unique? This would allow us to choose a \stringh\ structure when studying type IIA backgrounds, and/or to argue that such a choice is uniquely defined.

\begin{Proposition}\label{prop:obstructlift}
Let $V\to X$ be a vector bundle with DMW structure. Then there is a class $\rho(V)\in H^9(X;\Z)$ which vanishes if $V$ is \stringh. If $X$ is a manifold of dimension $10$ or below with DMW structure, $\rho(X)$ is the complete obstruction to lifting the DMW structure to a \stringh\ structure.
\end{Proposition}
\begin{Lemma}
 Let $B\Spin^c\langle W_7 \rangle$ be the fiber of the map $B\Spin^c \xrightarrow{W_7}K(\Z,7)$. The homotopy groups up to degree~$10$ of $B\Spin^c\langle W_7 \rangle$ are given by
 \begin{equation*}
 \pi_*(B\Spin^c\langle W_7 \rangle)= \{0,\,0,\,\Z,\,0,\,\Z,\,0,\Z,\,0, \Z,\,\Z/2,\,\Z/2,\ldots \} .
 \end{equation*}
\end{Lemma}
\begin{proof} The lemma follows immediately from studying the homotopy long exact sequence for $B\Spin^c\langle W_7 \rangle \rightarrow B\Spin^c \rightarrow K(\Z,7)$, and using the homotopy groups in Figure~\ref{fig:homotopyBStringh}.
\end{proof}

\begin{proof}[Proof of Proposition~\ref{prop:obstructlift}]
 Let the space $F$ be the fiber of the map
\[
f\colon \ B\String^h \to B\Spin^c\langle W_7\rangle.
\]
For a map $X\to B\Spin^c\langle W_7\rangle$, we want to quantify the first obstruction to lifting against the map $f$. This will be given by a cohomology class $H^{n+1}(X;\pi_n(F))$. The long exact sequence in homotopy groups for the fiber sequence $F \rightarrow B\String^h\rightarrow B\Spin^c\langle W_7\rangle$ is given in Figure~\ref{fig:homotopyofF}. We see that the only homotopy group that contributes to the obstruction for manifolds $X$ in the degrees we are considering is $\pi_8(F) = \mathbb{Z}$. Therefore, the obstruction class is in $H^9(X;\Z)$.
 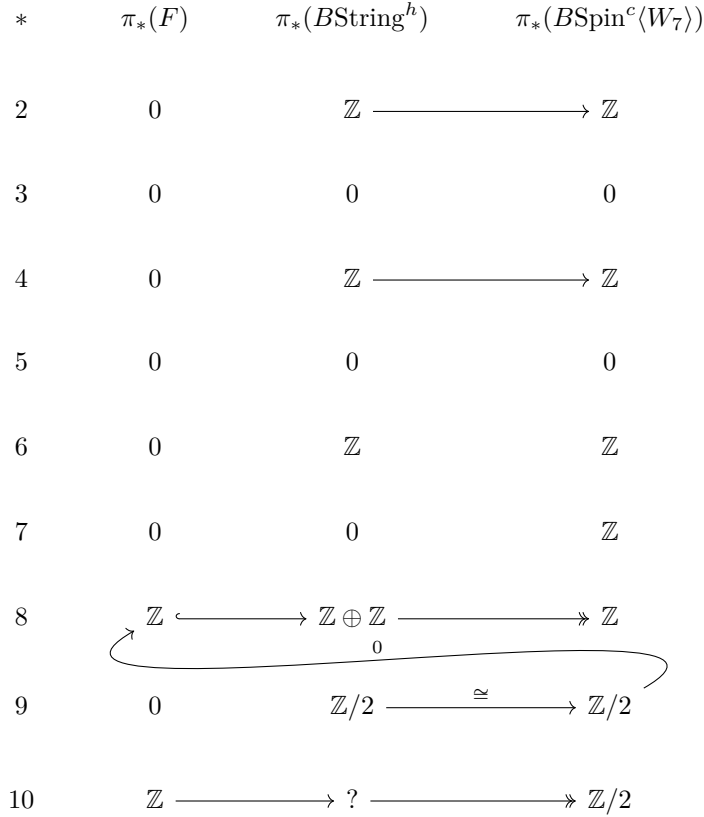
\begin{figure}[!ht]
 \centering
\begin{tikzcd}
	{*} & {\pi_*(F)} & {\pi_*(B\String^h)} & {\pi_*(B\Spin^c\langle W_7\rangle)} \\
	2 & 0 & {\Z}\arrow[r] & \Z \\
	3 & 0 & 0 & 0 \\
	4 & 0 & \Z \arrow[r] & \Z \\
	5 & {0} & 0 & 0 \\
	6 & 0 & \Z & \Z \\
	7 & {0} & {0} & \Z \\
	8 & \Z \arrow[r,hook] & {\Z\oplus \Z} \arrow[r, two heads] & \Z\\
 9 & 0 &\Z/2 \arrow[r,"\cong"] &\Z/2 \arrow[from=9-4,to=8-2,in=-150, out=30,"0",swap] \\
 10 & \Z \arrow[r] & ? \arrow[r,two heads] &\Z/2
\end{tikzcd}
 \caption{Homotopy long exact sequence for computing the homotopy groups of $B\String^h$ in degrees up to 10.}
 \label{fig:homotopyofF}
\end{figure}
\end{proof}

If an obstruction class in $H^n$ trivializes, the choices of trivializations live in a torsor for $H^{n-1}$. We summarize the implications below:
\begin{itemize}\itemsep=0pt
 \item For manifolds of dimension $5$ and below, a \spinc structure lifts uniquely to a DMW structure. On $6$-manifolds, such a lift exists but may not be unique, as $H^6(X,\Z)$ is not necessarily trivial.
 \item For spin$^c$ manifolds $X$ that are dimension $7$ and below, the obstruction for a DMW structure to lift to a string$^h$ structure vanishes, and there is a unique lift of a DMW structure to a \stringh\ structure. In dimension 8, the obstruction vanishes but the choice of lift may be nonunique.
 \item For spin$^c$ manifolds in dimensions 9 and 10, it is thus far unclear whether a DMW structure always lifts to a string$^h$ structure (by contrast to the spin case, where we saw in Proposition~\ref{f_well_defined} that a $10$-manifold with both spin and DMW structures admits a \stringh\ structure).
\end{itemize}

\begin{Theorem}\label{prop:dim9}
Let $M$ be a closed $9$-dimensional \spinc manifold. Every DMW structure on $M$ lifts to a \stringh\ structure.
\end{Theorem}
\begin{Lemma}\label{lem:splitoff}
 If $M$ is a closed, $R$-oriented $n$-manifold, for $R$ a ring spectrum, then the map $\Sigma^n R\to R \wedge M_{+}$ induced by the top cell of $M$ splits off as a direct summand.
\end{Lemma}
\begin{proof}
The bottom cell of any space always splits off stably, by the inclusion of the basepoint followed by the crush map. If $M$ is $R$-oriented, Atiyah--Poincar\'e duality identifies $R\wedge M_+$ with its $(\dim(M))$-shifted ($R$-module) Spanier--Whitehead dual $R\wedge \Sigma^{-\dim(M)} D(M)_+$; the top cell of $D(M)$ corresponds to the bottom cell of $M$, hence splits off, and therefore the top cell of $R\wedge M_+$ does as well.
\end{proof}

As a result, the $p=9$ column of the Atiyah--Hirzebruch spectral sequence computing $\ku^*(M)$ for a closed \spinc $9$-manifold $M$ splits off as a direct sum, which prohibits any differentials or extensions involving this column.

\begin{proof}[Proof of Theorem~\ref{prop:dim9}]
Consider the $\ku$-cohomology Atiyah--Hirzebruch spectral se\-quen\-ce for $M$, and suppose that $\lambda^c \in E_2^{4,0} \cong H^4(M;\Z)$ is such that $\Box_{\Z} \Sq^2(\lambda^c)=0$. If $\lambda^c$ survives to the $E_\infty$ page of the $\ku^*(M)$ spectral sequence, then $\lambda^c$ has a $K$-theory lift to $\ku^4(M)$. The homotopy groups of $\ku$ start in degree 0 given by $\Z$, and by Bott periodicity are $\Z$ in each positive even degree.
 The only differentials that $\lambda^c$ can admit, and within the range of degrees that we are considering, are $d_3$ and $d_5$:
 \begin{itemize}\itemsep=0pt
 \item The $d_3$ differential is given by $\Box_{\Z} \circ \Sq^2\circ \text{mod}\, 2$, which maps $\lambda^c$ to $W_7(M)$. But since this class is trivial by assumption, $d_3$ vanishes.
\item The $d_5$ differential maps $\lambda^c$ to $E_5^{9,-4}$, but by Lemma~\ref{lem:splitoff} $E_2^{9,-4}$ must split off as a direct sum and therefore cannot be killed by a differential, hence $d_5$ vanishes.
 \end{itemize}

Thus, $\lambda^c$ survives to $E^{4,0}_{\infty}$ which means it has a $\ku$-cohomology lift, and $M$ has a string$^h$ structure.
\end{proof}

\subsection[Applications of string\^{}h for type IIA compactifications]{Applications of $\boldsymbol{{\rm string}^h}$ for type IIA compactifications}\label{subsection:applications}

Consider a compactification of type IIA string theory down to dimension $d<10$. We know that imposing Diaconescu--Moore--Witten's anomaly cancellation condition $W_7 = 0$ resolves a~sign ambiguity in the partition function, but it is a priori possible that the compactified theory has an anomaly $\alpha$ of some other provenance. This anomaly is a unitary $(d+1)$-dimensional invertible field theory of manifolds equipped with a DMW structure and possibly a map to a space $X$ (e.g., $X = BG$ if we have a background gauge field for the group $G$), so by work of Freed--Hopkins~\cite{Freed:2016rqq} and Grady~\cite{Grady:2023sav}, $\alpha$ is classified in terms of the bordism groups $\Omega_{k}^{\Spin^c\ang{W_7}}(X)$ for $k= d+1,d+2$.

There is a standard procedure to compute bordism groups of manifolds with a trivialized characteristic class such as DMW-structures (see, for example,~\cite[Section~3.3.2]{BDDM24}): first, use the Serre spectral sequence to study $H^*(B\Spin^c\ang{W_7})$, then use that cohomology as input to the Adams or Atiyah--Hirzebruch spectral sequence. This is thus quite a bit more complicated than just computing \spinc bordism.

In this subsection, we will use \stringh\ bordism to simplify the bordism computations underlying anomaly cancellation of these IIA compactifications. Specifically, we will lift from DMW-structures to \stringh\ structures, and show that in dimensions $d\le 8$, this loses no information about the anomaly. We will also see that the Atiyah--Hirzebruch and Adams spectral sequences for \stringh\ bordism are relatively straightforward after our work in the previous section. See \cite{Debray:2023tdd,Debray:2022wcd,Tachikawa:2021mvw, TY:2021mby} for more examples of anomaly cancellations in compactifications of supergravity and heterotic string theory.
We highlight here how using string$^h$ affects the computations in different dimensions:
\begin{itemize}\itemsep=0pt
 \item If $M$ is a manifold of dimension $5$ or below, every \spinc structure on $M$ lifts uniquely to a~\stringh\ structure. Therefore, in dimensions $5$ and below, lifting to a \stringh\ structure does not buy us anything new over computing with \spinc.
 \item For manifolds in dimensions $6\le d\le 9$, \spinc and DMW structures are not equivalent, and every DMW structure lifts to a \stringh\ structure. Therefore, given a $d$-dimensional field theory on manifolds with a DMW structure (and perhaps also some background fields), the anomaly is trivializable as an IFT of DMW manifolds if and only if it is trivializable as a \stringh\ theory. Because \stringh\ bordism is easier to compute than DMW bordism, as we will see
 in a few examples below, this can assist in anomaly cancellation computations.
 \item In dimension 10, because we do not know whether every DMW structure lifts to a \stringh\ structure, we do not know whether restricting to \stringh\ manifolds loses information with regards to anomaly cancellation.
\end{itemize}

The real highlight of when \stringh\ leads to simplifications is when we are concerned with anomalies of a compactified theory that has some Lie group global symmetry. Using the change of rings from $\mathcal{A}$ to $\mathcal{E}(2)$ that is afforded to us by the orientation $\sigma_1(3)$, the Adams spectral sequence can be used to compute these bordism groups in low degrees.

\begin{Remark}If we suppose a na\"{\i}veness to string$^h$ as
 well as the $W_7=0$ condition
 and only considered spin$^c$ structures for target space manifolds, then we will start to see the difference after degree~6. After this degree is when string$^h$ bordism begins to have more free summands than in spin$^c$ and thus there could potentially be more perturbative anomalies to check.
\end{Remark}

\begin{Example}\label{ex:even}
Consider any theory that arises as a compactification of type IIA in dimension~9 or below, with a~$G$ symmetry where $G$ is a Lie group of the type $\U_n$, $\SU_n$, or $\mathrm{Sp}_n$ for $n>1$. The bordism groups relevant for anomaly cancellation will be \smash{$\Omega^{\Spin^c\langle W_7\rangle}_{*}(BG)$}, and by the above discussion \smash{$\Omega_*^{\String^h}$} maps surjectively onto the DMW bordism groups in dimensions $*\le 9$. Since the homology and \smash{$\Omega_*^{\String^h}$} are both concentrated in even degrees in the range relevant to this example, the Atiyah--Hirzebruch spectral sequence that computes \smash{$\Omega^{\String^h}_{*}(BG)$} therefore collapses on the $E_2$-page in the degrees relevant to this example. Thus the Anderson dual \smash{$\bigl(I_{\Z}\Omega^{\String^h}\bigr)^*(BG)$}, which classifies the unitary invertible field theories with this structure, is free and concentrated in odd degrees. This implies that any theory with a \stringh\ structure and a~$G$ global symmetry has no global anomalies to cancel, and once the perturbative anomalies are cancelled then the theory is anomaly-free.
\end{Example}

\begin{Example}\label{ex:spin}
 Consider any theory that arises as a compactification of type IIA in dimension~9 or below, with a $\U_1$ global symmetry. If one is interested in computing \smash{$\Omega^{\Spin^c\langle W_7\rangle}_*(B\U_1)$}, the situation is much more complicated from the point of view of the Atiyah--Hirzebruch spectral sequence because the low degree homology classes for $B\U_1$ are more nontrivial. However, this is where being able to lift to string$^h$ pays off. Since the homology of $B\U_1$ is in even degrees, and by Corollary~\ref{bordism_ring} the homotopy groups of $\MTString^h$ are also torsion-free and concentrated in even degrees in this range. Therefore, the Atiyah--Hirzebruch spectral sequence for \smash{$\Omega^{\String^h}_*(B\U_1)$} collapses on the $E_2$ page and the anomalies share the same properties as in Example~\ref{ex:even}.
 \end{Example}

\begin{Example}\label{ex:simplyconnectedLie}
 Let $G$ be a connected, simple, simply connected Lie group. Since the reduced homology of $BG$ begins in degree $4$, and \smash{$\Omega^{\String^h}_*\to \Omega^{\Spin^c}_*$} is an isomorphism in degrees $< 6$ we see by the Atiyah--Hirzebruch spectral sequence that the first place where the two groups \smash{$\Omega^{\String^h}_*(BG)$} and \smash{$\Omega^{\Spin^c}_*(BG)$} can potentially begin to differ is in bi-degree $(p,q)=(4,6)$. Therefore,
 \smash{$\Omega^{\String^h}_*(BG)\to \Omega^{\Spin^c}_*(BG)$} is an isomorphism in degrees $\leq 9$. In Theorem~\ref{appthm}, we prove that the groups \smash{$\Omega^{\Spin^c}_*(BG)$} are torsion-free in degrees $9$ and below, so the same is true for \stringh\ bordism. Similarly to the previous two examples, this suffices to detect anomalies on manifolds with DMW structures and principal $G$-bundles in dimensions $9$ and below, so we learn that anomalies for these symmetries therefore share the same features as in Example~\ref{ex:even}.

In higher degrees we predict it will be easier to use the Adams spectral sequence, where one can take advantage of a change-of-rings result to work over $\cE(2)$, since the complications with using the Atiyah--Hirzebruch spectral sequence build up very quickly in higher degrees. We leave the details to future work.
\end{Example}

\appendix

\section{A bordism computation}\label{appendix:computations}
The purpose of this appendix is to prove the following theorem.
\begin{Theorem}\label{appthm}
Let $G$ be a connected, simple, simply connected Lie group. Then $\ku_*(BG)$ is torsion-free in degrees $10$ and below.
\end{Theorem}
This is an ingredient in our anomaly cancellation result in Examples~\ref{ex:simplyconnectedLie} and~\ref{ex:spin}; however, it requires different techniques than we used in that section, so we have siloed it off here.
\begin{Lemma}\label{easy_appendix}\hfill
\begin{enumerate}\itemsep=0pt
 \item[$1.$] Theorem~{\rm \ref{appthm}} is true for $G = \mathrm{Sp}_n$, $\SU_n$, and $($for $n\le 6)$ $\Spin_n$.
 \item[$2.$] If we localize at a prime $p > 5$, the theorem is true for all $G$ in the statement of Theorem~{\rm \ref{appthm}}. Localized at $p = 5$, the theorem is true for all such~$G$ except perhaps~$E_8$, and localized at $p = 3$, the theorem is true for all such $G$ except perhaps $F_4$, $E_6$, $E_7$, and~$E_8$.
\end{enumerate}
\end{Lemma}

\begin{proof}
For part~(1), let $G$ be $\mathrm{Sp}_n$, $\SU_n$, or (for $n\le 6$) $\Spin_n$ and set up the Atiyah--Hirzebruch spectral sequence
\begin{equation*}
 E^2_{p,q} = H_p(BG; \ku_q) \Longrightarrow \ku_{p+q}(BG).
\end{equation*}
For these choices of $G$, $H_*(BG;\Z)$ is torsion-free and concentrated in even degrees. Since $\ku_*$ is also torsion-free and concentrated in even degrees, the spectral sequence collapses to imply the first part of the lemma statement.

The proof of part~(2) is similar except for using the $\ku_{(p)}$-homology Atiyah--Hirzebruch spectral sequence, whose input is the $\Z_{(p)}$-homology of $BG$. Assume $p\ge 7$, or $p = 5$ and $G\ne E_8$, or $p = 3$ and $G\not\in\set{F_4, E_6, E_7, E_8}$. Borel~\cite[Th\'eorèmes~B et~2.5]{Bor61} shows that for these choices of $G$ and $p$, $H^*(BG;\Z)$ lacks $p$-torsion and is concentrated in even degrees, so the Atiyah--Hirzebruch spectral sequence collapses as in the previous paragraph.
\end{proof}

The Atiyah--Hirzebruch-style proof of Lemma~\ref{easy_appendix} does not generalize nicely to the remaining cases of Theorem~\ref{appthm}, so we use the Adams spectral sequence. Choose a prime $p$ and let $\cA$ denote the $p$-primary Steenrod algebra, the $\Z$-graded noncommutative $\Z/p$-algebra consisting of natural transformations $H^*(\bl;\Z/p)\to H^{*+n}(\bl;\Z/p)$ that commute with the suspension functor. Then the Adams spectral sequence has signature
\begin{equation}\label{Adams_E2}
 E_2^{s,t} = \Ext_\cA(H^*(X;\Z/p), \Z/p) \Longrightarrow \pi_{t-s}^s(X)_p^\wedge,
\end{equation}
where $\pi_*^s$ denotes stable homotopy groups and $(\bl)_p^\wedge$ denotes $p$-completion.
\begin{Definition}
Let $Q_i\in\cA$ denote the $i^{\mathrm{th}}$ Milnor primitive; thus $Q_0$ is the Bockstein operator for $0\to\Z/p\to\Z/p^2\to\Z/p\to 0$ and $Q_1$ is the commutator of $Q_0$ and $\Sq^2$ (if $p = 2$) or $\cP^1$ (if $p > 2$).

Let $\cE(1)\coloneqq\ang{Q_0, Q_1}\subset\cA$; the Adem relations imply $\cE(1)$ is an exterior algebra on $Q_0$ and $Q_1$.
\end{Definition}
\begin{Theorem}[{Adams~\cite[Section~16]{Ada74}}]\label{adams_summand}
For each prime $p$, there is a spectrum $\ell$ with the following properties:
\begin{enumerate}\itemsep=0pt
 \item[$1.$] $\ku_{(p)} \simeq \ell\vee \Sigma^2\ell\vee\dotsb\vee\Sigma^{2(p-2)}\ell$, or if $p = 2$, $\ku_{(2)}\simeq\ell$.
 \item[$2.$] There is an $\cA$-module isomorphism $H^*(\ell;\Z/p)\cong \cA\otimes_{\cE(1)}\Z/p$.
\end{enumerate}
\end{Theorem}

The first part of Theorem~\ref{adams_summand} implies that, if we can prove $\ell_*(BG)$ is torsion-free in degrees~$10$ and below, then we have proven Theorem~\ref{appthm}. The second part of Theorem~\ref{adams_summand} allows us to simplify the Adams spectral sequence calculating $\ell$-homology: if one plugs in $X = \ell\wedge Y$ to~\eqref{Adams_E2}, the spectral sequence simplifies to
\begin{equation}\label{ell_Adams}
 E_2^{s,t} = \Ext_{\cE(1)}^{s,t}(H^*(Y;\Z/p), \Z/p) \Longrightarrow \ell_{t-s}(Y)_p^\wedge.
\end{equation}
The $\ell_*$-module structure on $\ell_*(Y)$ manifests in this spectral sequence through the action of the algebra $\Ext_{\cE(1)}(\Z/p, \Z/p)$ on the $E_2$-page of~\eqref{ell_Adams}. Explicitly, this algebra is~\cite[Section~2.1]{BG03}
\begin{equation}\label{extE1}
 \Ext_{\cE(1)}^{*,*}(\Z/p, \Z/p) \cong \Z/p[h_0, v_1]
\end{equation}
with $h_0\in\Ext^{1,1}$ and $v_1\in\Ext^{1,2p-2}$. The action of $h_0$ on the $E_\infty$-page of~\eqref{ell_Adams} lifts to detect multiplication by $p$ on $\ell_*(Y)$.

On the $E_2$-page of~\eqref{ell_Adams}, an \textit{$h_0$-tower} is a free $\Z/p[h_0]$-module of rank $1$.
\begin{Lemma}\label{even_good}
Suppose $Y$ is a CW complex with finitely many cells in each dimension and the $E_2$-page of the spectral sequence~\eqref{ell_Adams} for $Y$ consists solely of $h_0$-towers in even $(t-s)$-degrees as long as $t-s\le N$. Then in degrees $k\le N-1$, $\ell_k(Y)$ is torsion-free.
\end{Lemma}
For any connected Lie group $G$, there is a choice of $BG$ which is a CW complex with finitely many cells in each dimension, so this hypothesis does not worry us.
\begin{proof}
By assumption, $t-s$ is even for all nonzero classes on the $E_2$-page with $t-s\le N$; since Adams differentials change the parity of $t-s$, this forces all differentials in that range to vanish. Then, all extension questions for multiplication by $p$ in that range are resolved by the $h_0$-action: since the $E_\infty$-page for $t-s\le N-1$ is a direct sum of $h_0$-towers, there can be no hidden extensions by $p$ in this range. Since $Y$ has finitely many cells in each dimension, $\ell_k(Y)$ is a finitely generated $\Z_{(p)}$-module for each $k$, so we conclude that, in the range claimed, $\ell_k(Y)$ is a free $\Z_{(p)}$-module of finite rank.
\end{proof}
\begin{Remark}\label{Zp_even}
Because $\abs{v_1}$ is even, \eqref{extE1} implies $\Ext_{\cE(1)}(\Z/p, \Z/p)$ consists of $h_0$-towers in even degrees.
\end{Remark}
\begin{Lemma}[{Adams-Priddy~\cite[\S 3]{AP76}}]\label{uQ_even}
Up to multiplication by a unit in $(\Z/p)^\times$, there is a unique nonzero $\cE(1)$-module map $f\colon \Sigma^{-1}\cE(1)\to \Sigma^{-1}\Z/p$. Given any such map, let $\uQ\coloneqq \ker(f)$; the isomorphism type of $\uQ$ does not depend on $f$. Moreover, there is an $\Ext_{\cE(1)}(\Z/p, \Z/p)$-equivariant isomorphism $\Ext^{s,t}_{\cE(1)}(\uQ, \Z/p)\cong \Ext_{\cE(1)}^{s+1, t+1}(\Z/p, \Z/p)$, so $\Ext_{\cE(1)}(\uQ, \Z/p)$ consists of $h_0$-towers in even degrees.
\end{Lemma}

For any $\cE(1)$-module $M$, let $M_{\ge d}$ denote the $\cE(1)$-submodule generated by homogeneous classes in degrees $d$ and greater.
\begin{Definition}
We say that two $\cE(1)$-modules $M$ and $N$ are \textit{isomorphic up to degree $d$}, denoted $M\cong_{<d}N$, if there is an $\cE(1)$-module isomorphism $M/M_{\ge d}\overset\cong\to N/N_{\ge d}$.
\end{Definition}
\begin{Proposition}
Let $Y$ be a CW complex with finitely many cells in each dimension. Suppose that $H^*(Y;\Z/p)$ is isomorphic up to degree $d$ to a direct sum of copies of $\Sigma^{2m_i}\Z/p$ and $\Sigma^{2n_j}\uQ$ for various $m_i,n_j$. Then for $k\le d-2$, $\ell_*(Y)$ has no $p$-torsion.
\end{Proposition}
\begin{proof}
Use the long exact sequence in Ext associated to the short exact sequence $0\to M_{\ge n}\to M\to M/M_{\ge n}\to 0$ to show the map $M\to M/M_{\ge d}$ induces an isomorphism in Ext in degrees $d-2$ and below. Therefore, for the purpose of calculating $\ell_*(Y)$ in degrees $d-2$ and below, we may replace $H^*(Y;\Z/p)$ with a sum of shifts of $\Z/p$ and $\uQ$ by even degrees. The result then follows from Lemma~\ref{even_good} and the observations in Remark~\ref{Zp_even} and Lemma~\ref{uQ_even} that $\Ext_{\cE(1)}(\Z/p, \Z/p)$ and $\Ext_{\cE(1)}(\uQ, \Z/p)$ consist of $h_0$-towers in even degrees.
\end{proof}
Thus to prove Theorem~\ref{appthm}, it would suffice to prove the following assertions.
\begin{Proposition}
The following are isomorphisms of $\cE(1)$-modules up to degree $12$.
\begin{itemize}\itemsep=0pt
 \item For $G = G_2$, $F_4$, $E_6$, $E_7$, $E_8$, and $\Spin_7$, $\widetilde H^*(BG;\Z/2)\cong_{<12} \Sigma^4\uQ\oplus \Sigma^8 \Z/2$.
 \item For $G = \Spin_8$ and $\Spin_9$, $\widetilde H^*(BG;\Z/2)\cong_{<12} \Sigma^4\uQ\oplus \Sigma^8\Z/2\oplus\Sigma^8 \Z/2$.
 \item For $n\ge 10$, $\widetilde H^*(B\Spin_n;\Z/2)\cong_{<12}\Sigma^4\uQ\oplus \Sigma^8\Z/2\oplus\Sigma^8 \uQ$.
 \item For $G=F_4,E_6,E_7$, and $E_8$, $\widetilde H^*(BG;\Z/3)\cong_{<12} \Sigma^4\uQ\oplus \Sigma^8\Z/3$.
 \item $\widetilde H^*(BE_8;\Z/5)\cong_{<12} \Sigma^4 \Z/5$.
\end{itemize}
\end{Proposition}

\begin{proof}In \cite{Lee:2022spd}, the authors compute the $\cA(1)$-module structure on $H^*(BG;\Z/2)$ in the degrees we need for $G = \Spin_n$, $G_2$, $F_4$, $E_6$, $E_7$, and $E_8$, from which the $\cE(1)$-module structures in the theorem statement follow.

The assertion for $H^*(BF_4;\Z/3)$ follows from the products and Steenrod operations given by Toda~\cite[Theorems I, II, III]{Tod73}. The cohomology $H^*(BE_6;\Z/3)$, can be computed from \cite[Th\'eorème 2.3]{Bor61} where the author computes $H^*(E_6;\Z/3)$. The cohomology $H^*(BE_7;\Z/3)$ can be computed from \cite[Theorem 8]{araki1961differential} and $H^*(BE_8;\Z/3)$ can be computed from \cite[Theorem 9]{araki1961differential}, where the author computes $H^*(E_7;\Z/3)$ and $H^*(E_8;\Z/3)$ respectively together with Steenrod powers on the generators; the Kudo transgression theorem~\cite{Kud56} determines the Steenrod powers in the cohomology of the classifying spaces.
The assertion for $BE_8$ at $p = 5$ follows from Borel~\cite[Th\'eorème 2.3]{Bor61} calculating $H^*(E_8;\Z/5)$ together with a quick transgression argument.
\end{proof}

%\subsection*{Data availability statement}
%The authors declare that there is no additional associated data to this paper.

\subsection*{Acknowledgements}
It is a pleasure to thank Sanath Devalapurkar for helpful discussions on the content of this paper, including bringing to our attention his definition of \stringh\ and its application to $\tmf_1(3)$. The authors also thank Ivano Basile, Markus Dierigl, Theo Johnson-Freyd, Greg Moore, and Urs Schreiber for helpful conversations. The author MY is supported by the EPSRC Open Fellowship EP/X01276X/1.
%The authors declare that they have no competing interests relevant to the content of this article.

\pdfbookmark[1]{References}{ref}
\LastPageEnding


\begin{thebibliography}{99}
\footnotesize\itemsep=-0.3pt

\bibitem{AHK26}
Abdallah H.H., Halladay Z., Kamel Y., On the nonexistence of a {G}reen functor
  with values {${\rm spin}^c$} bordism and {${\rm spin}$} bordism,
  \href{https://doi.org/10.1007/s00013-026-02229-y}{\textit{Arch. Math.
  (Basel)}} \textbf{126} (2026), 377--381,
  \href{https://arxiv.org/abs/2502.12228}{arXiv:2502.12228}.

\bibitem{AK25}
Abdallah H.H., Kamel Y., The homotopy fixed points of real spin bordism,
  \href{https://arxiv.org/abs/2511.10624}{arXiv:2511.10624}.

\bibitem{Abs21}
Absmeier D., Strictly commutative complex orientations of {${\rm Tmf}_1(N)$},
  \href{https://arxiv.org/abs/2108.03079}{arXiv:2108.03079}.

\bibitem{Ada74}
Adams J.F., Stable homotopy and generalised homology, \textit{Chicago Lectures in
  Math.}, University of Chicago Press, Chicago, Ill., 1974.

\bibitem{AP76}
Adams J.F., Priddy S.B., Uniqueness of~{$B{\rm SO}$},
  \href{https://doi.org/10.1017/S0305004100053111}{\textit{Math. Proc.
  Cambridge Philos. Soc.}} \textbf{80} (1976), 475--509.

\bibitem{Agu81}
Aguad\'e J., On the space of free loops of an odd sphere, \textit{Publ. Sec.
  Mat. Univ. Aut\`onoma Barcelona}  \textbf{25} (1981), 87--90.

\bibitem{ADH21}
Amabel A., Debray A., Haine P.J. (Editors),
Differential cohomology: {C}ategories, characteristic classes, and connections,
  \href{https://arxiv.org/abs/2109.12250}{arXiv:2109.12250}.

\bibitem{ABP67}
Anderson D.W., Brown Jr. E.H., Peterson F.P., The structure of the {S}pin
  cobordism ring, \href{https://doi.org/10.2307/1970690}{\textit{Ann. of
  Math.}} \textbf{86} (1967), 271--298.

\bibitem{AHR10}
Ando M., Hopkins M.J., Rezk C., Multiplicative orientations of
  {$\mathrm{KO}$}-theory and of the spectrum of topological modular forms,
  2010, \url{https://faculty.math.illinois.edu/~mando/papers/koandtmf.pdf}.

\bibitem{araki1961differential}
Araki S., Differential {H}opf algebras and the cohomology {${\rm mod}\ 3$} of
  the compact exceptional groups {$E_{7}$} and {$E_{8}$},
  \href{https://doi.org/10.2307/1970339}{\textit{Ann. of Math.}} \textbf{73}
  (1961), 404--436.

\bibitem{Ara79}
Araki S., Forgetful spectral sequences, \textit{Osaka Math.~J.} \textbf{16}
  (1979), 173--199.

\bibitem{Ara79a}
Araki S., Orientations in {$\tau $}-cohomology theories,
  \href{https://doi.org/10.4099/math1924.5.403}{\textit{Japan.~J.
  Math.~(N.S.)}} \textbf{5} (1979), 403--430.

\bibitem{AM78}
Araki S., Murayama M., {$\tau $}-cohomology theories,
  \href{https://doi.org/10.4099/math1924.4.363}{\textit{Japan.~J.
  Math.~(N.S.)}} \textbf{4} (1978), 363--416.

\bibitem{Ati66}
Atiyah M.F., {$K$}-theory and reality,
  \href{https://doi.org/10.1093/qmath/17.1.367}{\textit{Quart.~J. Math. Oxford
  Ser.~(2)}} \textbf{17} (1966), 367--386.

\bibitem{ABS64}
Atiyah M.F., Bott R., Shapiro A., Clifford modules,
  \href{https://doi.org/10.1016/0040-9383(64)90003-5}{\textit{Topology}}
  \textbf{3} (1964), suppl.~1, 3--38.

\bibitem{BL01}
Baker A., Lazarev A., On the {A}dams spectral sequence for {$R$}-modules,
  \href{https://doi.org/10.2140/agt.2001.1.173}{\textit{Algebr. Geom. Topol.}}
  \textbf{1} (2001), 173--199,
  \href{https://arxiv.org/abs/math/0105079}{arXiv:math/0105079}.

\bibitem{BDDM24}
Basile I., Debray A., Delgado M., Montero M., Global anomalies \& bordism of
  non-supersymmetric strings,
  \href{https://doi.org/10.1007/jhep02(2024)092}{\textit{J.~High Energy Phys.}}
  \textbf{2024} (2024), no.~2, 092, 71~pages,
  \href{https://arxiv.org/abs/2310.06895}{arXiv:2310.06895}.

\bibitem{BM13}
Basterra M., Mandell M.A., The multiplication on {BP},
  \href{https://doi.org/10.1112/jtopol/jts032}{\textit{J.~Topol.}} \textbf{6}
  (2013), 285--310, \href{https://arxiv.org/abs/1101.0023}{arXiv:1101.0023}.

\bibitem{bauer2008computation}
Bauer T., Computation of the homotopy of the spectrum {\tt tmf}, in Groups,
  Homotopy and Configuration Spaces, \textit{Geom. Topol. Monogr.}, Vol.~13,
  \href{https://doi.org/10.2140/gtm.2008.13.11}{Geometry \& Topology
  Publications}, Coventry, 2008, 11--40,
  \href{https://arxiv.org/abs/math/0311328}{arXiv:math/0311328}.

\bibitem{BLM23}
Beardsley J., Luecke K., Morava J., Brauer--{W}all groups and truncated
  {P}icard spectra of {$K$}-theory,
  \href{https://doi.org/10.2140/agt.2026.26.1749}{\textit{Algebr. Geom.
  Topol.}} \textbf{26} (2026), 1749--1780,
  \href{https://arxiv.org/abs/2306.10112}{arXiv:2306.10112}.

\bibitem{Beh06}
Behrens M., A modular description of the {$K(2)$}-local sphere at the prime~3,
  \href{https://doi.org/10.1016/j.top.2005.08.005}{\textit{Topology}}
  \textbf{45} (2006), 343--402,
  \href{https://arxiv.org/abs/math/0507184}{arXiv:math/0507184}.

\bibitem{Beh07}
Behrens M., Buildings, elliptic curves, and the {$K(2)$}-local sphere,
  \href{https://doi.org/10.1353/ajm.2007.0037}{\textit{Amer.~J. Math.}}
  \textbf{129} (2007), 1513--1563,
  \href{https://arxiv.org/abs/math/0510026}{arXiv:math/0510026}.

\bibitem{BH15}
Blumberg A.J., Hill M.A., Operadic multiplications in equivariant spectra,
  norms, and transfers,
  \href{https://doi.org/10.1016/j.aim.2015.07.013}{\textit{Adv. Math.}}
  \textbf{285} (2015), 658--708,
  \href{https://arxiv.org/abs/1309.1750}{arXiv:1309.1750}.

\bibitem{Bor61}
Borel A., Sous-groupes commutatifs et torsion des groupes de {L}ie compacts
  connexes, \href{https://doi.org/10.2748/tmj/1178244298}{\textit{Tohoku
  Math.~J.}} \textbf{13} (1961), 216--240.

\bibitem{BEM:2003vb}
Bouwknegt P., Evslin J., Mathai V., {$T$}-duality: topology change from
  {$H$}-flux, \href{https://doi.org/10.1007/s00220-004-1115-6}{\textit{Comm.
  Math. Phys.}} \textbf{249} (2004), 383--415,
  \href{https://arxiv.org/abs/hep-th/0306062}{arXiv:hep-th/0306062}.

\bibitem{Bro82}
Brown Jr. E.H., The cohomology of {${\rm BSO}_{n}$} and {${\rm BO}_{n}$} with
  integer coefficients, \href{https://doi.org/10.2307/2044298}{\textit{Proc.
  Amer. Math. Soc.}} \textbf{85} (1982), 283--288.

\bibitem{BP66}
Brown Jr. E.H., Peterson F.P., A spectrum whose {$Z_{p}$} cohomology is the
  algebra of reduced {$p^{\rm th}$} powers,
  \href{https://doi.org/10.1016/0040-9383(66)90015-2}{\textit{Topology}}
  \textbf{5} (1966), 149--154.

\bibitem{BG03}
Bruner R.R., Greenlees J.P.C., The connective {$K$}-theory of finite groups,
  \href{https://doi.org/10.1090/memo/0785}{\textit{Mem. Amer. Math. Soc.}}
  \textbf{165} (2003), viii+127~pages.

\bibitem{BG10}
Bruner R.R., Greenlees J.P.C., Connective real {$K$}-theory of finite groups,
  \textit{Math. Surveys Monogr.}, Vol.~169,
  \href{https://doi.org/10.1090/surv/169}{American Mathematical Society},
  Providence, RI, 2010.

\bibitem{BR21}
Bruner R.R., Rognes J., The {A}dams spectral sequence for topological modular
  forms, \textit{Math. Surveys Monogr.}, Vol.~253,
  \href{https://doi.org/10.1090/surv/253}{American Mathematical Society},
  Providence, RI, 2021.

\bibitem{Bry00}
Brylinski J.-L., Differentiable cohomology of gauge groups,
  \href{https://arxiv.org/abs/math/0011069}{arXiv:math/0011069}.

\bibitem{BM94}
Brylinski J.-L., McLaughlin D.A., The geometry of degree-four characteristic
  classes and of line bundles on loop spaces.~{I},
  \href{https://doi.org/10.1215/S0012-7094-94-07518-2}{\textit{Duke Math.~J.}}
  \textbf{75} (1994), 603--638.

\bibitem{Cap16}
Capotosti A., From String structures to Spin structures on loop spaces, Ph.D.~Thesis, {R}oma Tre University, 2016,
  \url{https://arcadia.sba.uniroma3.it/bitstream/2307/5999/1/Thesis_A.Capotosti.pdf}.

\bibitem{CDvN24}
Carrick C., Davies J.M., van Nigtevecht S., The descent spectral sequence for
  topological modular forms,
  \href{https://arxiv.org/abs/2412.01640}{arXiv:2412.01640}.

\bibitem{CHZ}
Chen Q., Han F., Zhang W., Generalized {W}itten genus and vanishing theorems,
  \href{https://doi.org/10.4310/jdg/1317758867}{\textit{J.~Differential Geom.}}
  \textbf{88} (2011), 1--39,
  \href{https://arxiv.org/abs/1003.2325}{arXiv:1003.2325}.

\bibitem{conner:1966}
Conner P.E., Floyd E.E., The relation of cobordism to {$K$}-theories,
  \textit{Lecture Notes in Math.}, Vol.~28,
  \href{https://doi.org/10.1007/BFb0071091}{Springer}, Berlin, 1966.

\bibitem{CHS18}
C\'ordova C., Hsin P.S., Seiberg N., Global symmetries, counterterms, and
  duality in {C}hern--{S}imons matter theories with orthogonal gauge groups,
  \href{https://doi.org/10.21468/SciPostPhys.4.4.021}{\textit{SciPost Phys.}}
  \textbf{4} (2018), 021, 54~pages,
  \href{https://arxiv.org/abs/1711.10008}{arXiv:1711.10008}.

\bibitem{CN19}
Crowley D., Nordstr\"om J., The classification of 2-connected 7-manifolds,
  \href{https://doi.org/10.1112/plms.12222}{\textit{Proc. Lond. Math. Soc.}}
  \textbf{119} (2019), 1--54,
  \href{https://arxiv.org/abs/1406.2226}{arXiv:1406.2226}.

\bibitem{CY20}
Crowley D., Yang H., The existence of contact structures on 9-manifolds,
  \href{https://arxiv.org/abs/2011.09809}{arXiv:2011.09809}.

\bibitem{Deb24}
Debray A., Bordism for the 2-group symmetries of the heterotic and {CHL}
  strings, in Higher Structures in Topology, Geometry, and Physics,
  \textit{Contemp. Math.}, Vol. 802,
  \href{https://doi.org/10.1090/conm/802/16079}{American Mathematical Society,
  Providence, RI}, 2024, 227--297,
  \href{https://arxiv.org/abs/2304.14764}{arXiv:2304.14764}.

\bibitem{DDHM24}
Debray A., Dierigl M., Heckman J.J., Montero M., The chronicles of {IIB}ordia:
  dualities, bordisms, and the swampland,
  \href{https://doi.org/10.4310/atmp.241028224804}{\textit{Adv. Theor. Math.
  Phys.}} \textbf{28} (2024), 805--1025,
  \href{https://arxiv.org/abs/2302.00007}{arXiv:2302.00007}.

\bibitem{Debray:2023tdd}
Debray A., Yu M., Adams spectral sequences for non-vector-bundle {T}hom
  spectra, \href{https://arxiv.org/abs/2305.01678}{arXiv:2305.01678}.

\bibitem{Debray:2022wcd}
Debray A., Yu M., What bordism-theoretic anomaly cancellation can do for {U},
  \href{https://doi.org/10.1007/s00220-024-04937-4}{\textit{Comm. Math. Phys.}}
  \textbf{405} (2024), 154, 28~pages,
  \href{https://arxiv.org/abs/2210.04911}{arXiv:2210.04911}.

\bibitem{devalapurkar2019ando}
Devalapurkar S.K., The {A}ndo--{H}opkins--{R}ezk orientation is surjective,
  \href{https://arxiv.org/abs/1911.10534}{arXiv:1911.10534}.

\bibitem{Dev}
Devalapurkar S.K., A {${\rm String}$}-analogue of {${\rm Spin}^{\mathbf C}$},
  2022,
  \url{https://sanathdevalapurkar.github.io/files/string-analogue-spinc.pdf}.

\bibitem{Dev24}
Devalapurkar S.K., Higher chromatic {T}hom spectra via unstable homotopy
  theory, \href{https://doi.org/10.2140/agt.2024.24.49}{\textit{Algebr. Geom.
  Topol.}} \textbf{24} (2024), 49--108,
  \href{https://arxiv.org/abs/2004.08951}{arXiv:2004.08951}.

\bibitem{DMW:2000}
Diaconescu D.-E., Moore G., Witten E., A derivation of {$K$}-theory from
  {$M$}-theory,
  \href{https://arxiv.org/abs/hep-th/0005091}{arXiv:hep-th/0005091}.

\bibitem{DMW:E8}
Diaconescu D.-E., Moore G., Witten E., {$E_8$} gauge theory, and a derivation of
  {$K$}-theory from {M}-theory,
  \href{https://doi.org/10.4310/ATMP.2002.v6.n6.a2}{\textit{Adv. Theor. Math.
  Phys.}} \textbf{6} (2002), 1031--1134 (2003),
  \href{https://arxiv.org/abs/hep-th/0005090}{arXiv:hep-th/0005090}.

\bibitem{douglas2014topological}
Douglas C.L., Francis J., Henriques A.G., Hill M.A. (Editors), Topological
  modular forms, \textit{Math. Surveys Monogr.}, Vol. 201,
  \href{https://doi.org/10.1090/surv/201}{American Mathematical Society},
  Providence, RI, 2014.

\bibitem{Dua18}
Duan H., The characteristic classes and {W}eyl invariants of spinor groups,
  \href{https://arxiv.org/abs/1810.03799}{arXiv:1810.03799}.

\bibitem{Duff:1998us}
Duff M.J., L\"{u} H., Pope C.N., {${\rm AdS}_5\times S^5$} untwisted,
  \href{https://doi.org/10.1016/S0550-3213(98)00464-7}{\textit{Nuclear
  Phys.~B}} \textbf{532} (1998), 181--209,
  \href{https://arxiv.org/abs/hep-th/9803061}{arXiv:hep-th/9803061}.

\bibitem{Fiorenza:2019usl}
Fiorenza D., Sati H., Schreiber U., Twisted cohomotopy implies {M}-theory
  anomaly cancellation on 8-manifolds,
  \href{https://doi.org/10.1007/s00220-020-03707-2}{\textit{Comm. Math. Phys.}}
  \textbf{377} (2020), 1961--2025,
  \href{https://arxiv.org/abs/1904.10207}{arXiv:1904.10207}.

\bibitem{FSS12}
Fiorenza D., Schreiber U., Stasheff J., \v{C}ech cocycles for differential
  characteristic classes: an {$\infty$}-{L}ie theoretic construction,
  \href{https://doi.org/10.4310/atmp.2012.v16.n1.a5}{\textit{Adv. Theor. Math.
  Phys.}} \textbf{16} (2012), 149--250,
  \href{https://arxiv.org/abs/1011.4735}{arXiv:1011.4735}.

\bibitem{Fre86}
Freed D.S., Determinants, torsion, and strings,
  \href{https://doi.org/10.1007/BF01221001}{\textit{Comm. Math. Phys.}}
  \textbf{107} (1986), 483--513.

\bibitem{Fre00}
Freed D.S., Dirac charge quantization and generalized differential cohomology,
  in Surveys in Differential Geometry, \textit{Surv. Differ. Geom.}, Vol.~7,
  \href{https://doi.org/10.4310/SDG.2002.v7.n1.a6}{International Press},
  Somerville, MA, 2000, 129--194,
  \href{https://arxiv.org/abs/hep-th/0011220}{arXiv:hep-th/0011220}.

\bibitem{Fre08}
Freed D.S., Pions and generalized cohomology,
  \href{https://doi.org/10.4310/jdg/1217361066}{\textit{J.~Differential Geom.}}
  \textbf{80} (2008), 45--77,
  \href{https://arxiv.org/abs/hep-th/0607134}{arXiv:hep-th/0607134}.

\bibitem{Freed:2019sco}
Freed D.S., Hopkins M.J., Consistency of {M}-theory on non-orientable
  manifolds, \href{https://doi.org/10.1093/qmath/haab007}{\textit{Q.~J.~Math.}}
  \textbf{72} (2021), 603--671,
  \href{https://arxiv.org/abs/1908.09916}{arXiv:1908.09916}.

\bibitem{Freed:2016rqq}
Freed D.S., Hopkins M.J., Reflection positivity and invertible topological
  phases, \href{https://doi.org/10.2140/gt.2021.25.1165}{\textit{Geom. Topol.}}
  \textbf{25} (2021), 1165--1330,
  \href{https://arxiv.org/abs/1604.06527}{arXiv:1604.06527}.

\bibitem{Freed:2007vy}
Freed D.S., Hopkins M.J., Teleman C., Consistent orientation of moduli spaces,
  in The Many Facets of Geometry,
  \href{https://doi.org/10.1093/acprof:oso/9780199534920.003.0019}{Oxford
  University Press}, Oxford, 2010, 395--419,
  \href{https://arxiv.org/abs/0711.1909}{arXiv:0711.1909}.

\bibitem{FV87}
Freed D.S., Vafa C., Global anomalies on orbifolds,
  \href{https://doi.org/10.1007/BF01212418}{\textit{Comm. Math. Phys.}}
  \textbf{110} (1987), 349--389.

\bibitem{Fuj75}
Fujii M., Cobordism theory with reality, \textit{Math.~J. Okayama Univ.}
  \textbf{18} (1976), 171--188.

\bibitem{GJF19}
Gaiotto D., Johnson-Freyd T., Symmetry protected topological phases and
  generalized cohomology,
  \href{https://doi.org/10.1007/jhep05(2019)007}{\textit{J.~High Energy Phys.}}
  \textbf{2019} (2019), no.~5, 007, 34~pages,
  \href{https://arxiv.org/abs/1712.07950}{arXiv:1712.07950}.

\bibitem{Gia69}
Giambalvo V., The {${\rm mod}$} {$p$} cohomology of {${\rm BO}\langle 4k\rangle
  $}, \href{https://doi.org/10.2307/2035706}{\textit{Proc. Amer. Math. Soc.}}
  \textbf{20} (1969), 593--597.

\bibitem{Gia71}
Giambalvo V., On {$\langle 8\rangle$}-cobordism,
  \href{https://doi.org/10.1215/ijm/1256052508}{\textit{Illinois~J. Math.}}
  \textbf{15} (1971), 533--541.

\bibitem{Grady:2023sav}
Grady D., Deformation classes of invertible field theories and the
  {F}reed--{H}opkins conjecture,
  \href{https://arxiv.org/abs/2310.15866}{arXiv:2310.15866}.

\bibitem{Green:1984sg}
Green M.B., Schwarz J.H., Anomaly cancellations in supersymmetric {$D=10$}
  gauge theory and superstring theory,
  \href{https://doi.org/10.1016/0370-2693(84)91565-X}{\textit{Phys. Lett.~B}}
  \textbf{149} (1984), 117--122.

\bibitem{GM17}
Greenlees J.P.C., Meier L., Gorenstein duality for real spectra,
  \href{https://doi.org/10.2140/agt.2017.17.3547}{\textit{Algebr. Geom.
  Topol.}} \textbf{17} (2017), 3547--3619,
  \href{https://arxiv.org/abs/1607.02332}{arXiv:1607.02332}.

\bibitem{GW14}
Gu Z.C., Wen X.G., Symmetry-protected topological orders for interacting
  fermions: {F}ermionic topological nonlinear $\sigma$ models and a~special
  group supercohomology theory,
  \href{https://doi.org/10.1103/PhysRevB.90.115141}{\textit{Phys. Rev.~B}}
  \textbf{90} (2014), 115141, 59~pages,
  \href{https://arxiv.org/abs/1201.2648}{arXiv:1201.2648}.

\bibitem{Hai21}
Haine P.J., Splitting free loop spaces, 2021,
  \url{https://math.berkeley.edu/~phaine/files/Splitting_free_loops.pdf}.

\bibitem{HK24}
Halladay Z., Kamel Y., Real spin bordism and orientations of topological
  {K}-theory, \href{https://doi.org/10.1090/tran/9593}{\textit{Trans. Amer.
  Math. Soc.}} \textbf{379} (2026), 4869--4894,
  \href{https://arxiv.org/abs/2405.00963}{arXiv:2405.00963}.

\bibitem{HKT20}
Hason I., Komargodski Z., Thorngren R., Anomaly matching in the symmetry broken
  phase: domain walls, {CPT}, and the {S}mith isomorphism,
  \href{https://doi.org/10.21468/scipostphys.8.4.062}{\textit{SciPost Phys.}}
  \textbf{8} (2020), 062, 43~pages,
  \href{https://arxiv.org/abs/1910.14039}{arXiv:1910.14039}.

\bibitem{Hil07}
Hill M.A., The 3-local {${\rm tmf}$}-homology of {$B\Sigma_3$},
  \href{https://doi.org/10.1090/S0002-9939-07-08937-X}{\textit{Proc. Amer.
  Math. Soc.}} \textbf{135} (2007), 4075--4086,
  \href{https://arxiv.org/abs/math/0511649}{arXiv:math/0511649}.

\bibitem{Hil09}
Hill M.A., The {S}tring bordism of {$BE_8$} and {$BE_8\times BE_8$} through
  dimension~14,
  \href{https://doi.org/10.1215/ijm/1264170845}{\textit{Illinois~J. Math.}}
  \textbf{53} (2009), 183--196,
  \href{https://arxiv.org/abs/0807.2095}{arXiv:0807.2095}.

\bibitem{HL16}
Hill M.A., Lawson T., Topological modular forms with level structure,
  \href{https://doi.org/10.1007/s00222-015-0589-5}{\textit{Invent. Math.}}
  \textbf{203} (2016), 359--416,
  \href{https://arxiv.org/abs/1312.7394}{arXiv:1312.7394}.

\bibitem{HM17}
Hill M.A., Meier L., The {$C_2$}-spectrum {${\rm Tmf}_1(3)$} and its invertible
  modules, \href{https://doi.org/10.2140/agt.2017.17.1953}{\textit{Algebr.
  Geom. Topol.}} \textbf{17} (2017), 1953--2011,
  \href{https://arxiv.org/abs/1507.08115}{arXiv:1507.08115}.

\bibitem{HZ20}
Hill M.A., Zeng M., Generalized {$\mathbb Z$}-homotopy fixed points of {$C_n$}
  spectra with applications to norms of {$MU_{\mathbb R}$}, \textit{New York~J.
  Math.} \textbf{26} (2020), 92--115,
  \href{https://arxiv.org/abs/1808.10412}{arXiv:1808.10412}.

\bibitem{Hopkins1995}
Hopkins M.J., Topological modular forms, the {W}itten genus, and the theorem of
  the cube, in Proceedings of the {I}nternational {C}ongress of
  {M}athematicians, {V}ol.~1,~2 ({Z}\"{u}rich, 1994), Birkh\"auser, Basel,
  1995, 554--565.

\bibitem{hopkins2002algebraic}
Hopkins M.J., Algebraic topology and modular forms, in Proceedings of the
  {I}nternational {C}ongress of {M}athematicians, {V}ol.~{I} ({B}eijing, 2002),
  Higher Education Press, Beijing, 2002, 291--317,
  \href{https://arxiv.org/abs/math/0212397}{arXiv:math/0212397}.

\bibitem{HM14}
Hopkins M.J., Mahowald M., From elliptic curves to homotopy theory, in
  Topological Modular Forms, \textit{Math. Surveys Monogr.}, Vol.~201,
  \href{https://doi.org/10.1090/surv/201/15}{American Mathematical Society},
  Providence, RI, 2014, 261--285.

\bibitem{HopkinsMiller2014}
Hopkins M.J., Miller H.R., Elliptic curves and stable homotopy~{I}, in
  Topological Modular Forms, \textit{Math. Surveys Monogr.}, Vol.~201,
  \href{https://doi.org/10.1090/surv/201/14}{American Mathematical Society},
  Providence, RI, 2014, 209--260.

\bibitem{HK01}
Hu P., Kriz I., Real-oriented homotopy theory and an analogue of the
  {A}dams--{N}ovikov spectral sequence,
  \href{https://doi.org/10.1016/S0040-9383(99)00065-8}{\textit{Topology}}
  \textbf{40} (2001), 317--399.

\bibitem{DHH:19}
Huang R., Han F., Duan H., String{$^c$} structures and modular invariants,
  \href{https://doi.org/10.1090/tran/8311}{\textit{Trans. Amer. Math. Soc.}}
  \textbf{374} (2021), 3491--3533,
  \href{https://arxiv.org/abs/1905.02093}{arXiv:1905.02093}.

\bibitem{Jen05}
Jenquin J.A., Classical {C}hern--{S}imons on manifolds with spin structure,
  \href{https://arxiv.org/abs/math/0504524}{arXiv:math/0504524}.

\bibitem{Joa04}
Joachim M., Higher coherences for equivariant {$K$}-theory, in Structured Ring
  Spectra, \textit{London Math. Soc. Lecture Note Ser.}, Vol.~315,
  \href{https://doi.org/10.1017/CBO9780511529955.006}{Cambridge University
  Press}, Cambridge, 2004, 87--114.

\bibitem{JF20}
Johnson-Freyd T., Topological {M}athieu moonshine,
  \href{https://arxiv.org/abs/2006.02922}{arXiv:2006.02922}.

\bibitem{JFT20}
Johnson-Freyd T., Treumann D., {${\rm H}^4({\rm Co}_0; {\bf Z}) = {\bf Z}/24$},
  \href{https://doi.org/10.1093/imrn/rny219}{\textit{Int. Math. Res. Not.}}
  \textbf{2020} (2020), 7873--7907,
  \href{https://arxiv.org/abs/1707.07587}{arXiv:1707.07587}.

\bibitem{KT17}
Kapustin A., Thorngren R., Fermionic {SPT} phases in higher dimensions and
  bosonization, \href{https://doi.org/10.1007/jhep10(2017)080}{\textit{J.~High
  Energy Phys.}} \textbf{2017} (2017), no.~10, 080, 49~pages,
  \href{https://arxiv.org/abs/1701.08264}{arXiv:1701.08264}.

\bibitem{Killingback:1986rd}
Killingback T.P., World-sheet anomalies and loop geometry,
  \href{https://doi.org/10.1016/0550-3213(87)90229-X}{\textit{Nuclear Phys.~B}}
  \textbf{288} (1987), 578--588.

\bibitem{KM13}
Kottke C., Melrose R., Equivalence of string and fusion loop-spin structures,
  \href{https://arxiv.org/abs/1309.0210}{arXiv:1309.0210}.

\bibitem{Kud56}
Kudo T., A transgression theorem,
  \href{https://doi.org/10.2206/kyushumfs.9.79}{\textit{Mem. Fac. Sci.
  Ky\=usy\=u Univ.~A}} \textbf{9} (1956), 79--81.

\bibitem{Kur96}
Kuribayashi K., On the vanishing problem of string classes, \textit{J.~Austral.
  Math. Soc. Ser.~A} \textbf{61} (1996), 258--266.

\bibitem{KY98}
Kuribayashi K., Yamaguchi T., The vanishing problem of the string class with
  degree {$3$}, \textit{J.~Austral. Math. Soc. Ser.~A} \textbf{65} (1998),
  129--142.

\bibitem{Lan67}
Landweber P.S., Fixed point free conjugations on complex manifolds,
  \href{https://doi.org/10.2307/1970612}{\textit{Ann. of Math.}} \textbf{86}
  (1967), 491--502.

\bibitem{Lan68}
Landweber P.S., Conjugations on complex manifolds and equivariant homotopy of
  {$MU$}, \href{https://doi.org/10.1090/S0002-9904-1968-11917-2}{\textit{Bull.
  Amer. Math. Soc.}} \textbf{74} (1968), 271--274.

\bibitem{Lau04}
Laures G., {$K(1)$}-local topological modular forms,
  \href{https://doi.org/10.1007/s00222-003-0355-y}{\textit{Invent. Math.}}
  \textbf{157} (2004), 371--403.

\bibitem{Lau16}
Laures G., Characteristic classes in {$TMF$} of level {$\Gamma_1(3)$},
  \href{https://doi.org/10.1090/tran/6575}{\textit{Trans. Amer. Math. Soc.}}
  \textbf{368} (2016), 7339--7357,
  \href{https://arxiv.org/abs/1304.3588}{arXiv:1304.3588}.

\bibitem{LO16}
Laures G., Olbermann M., {$TMF_0(3)$}-characteristic classes for string
  bundles, \href{https://doi.org/10.1007/s00209-015-1551-3}{\textit{Math.~Z.}}
  \textbf{282} (2016), 511--533,
  \href{https://arxiv.org/abs/1403.7301}{arXiv:1403.7301}.

\bibitem{LO18}
Laures G., Olbermann M., Cannibalistic classes of string bundles,
  \href{https://doi.org/10.1007/s00229-017-0978-8}{\textit{Manuscripta Math.}}
  \textbf{156} (2018), 273--298,
  \href{https://arxiv.org/abs/1701.01287}{arXiv:1701.01287}.

\bibitem{LS19}
Laures G., Schuster B., Towards a splitting of the {$K(2)$}-local string
  bordism spectrum, \href{https://doi.org/10.1090/proc/14190}{\textit{Proc.
  Amer. Math. Soc.}} \textbf{147} (2019), 399--410,
  \href{https://arxiv.org/abs/1710.03427}{arXiv:1710.03427}.

\bibitem{lawson2015shimura}
Lawson T., The {S}himura curve of discriminant 15 and topological automorphic
  forms, \href{https://doi.org/10.1017/fms.2015.1}{\textit{Forum Math. Sigma}}
  \textbf{3} (2015), e3, 32~pages,
  \href{https://arxiv.org/abs/1301.3233}{arXiv:1301.3233}.

\bibitem{Law18}
Lawson T., Secondary power operations and the {B}rown--{P}eterson spectrum at
  the prime~2, \href{https://doi.org/10.4007/annals.2018.188.2.3}{\textit{Ann.
  of Math.}} \textbf{188} (2018), 513--576,
  \href{https://arxiv.org/abs/1703.00935}{arXiv:1703.00935}.

\bibitem{LN12a}
Lawson T., Naumann N., Commutativity conditions for truncated
  {B}rown--{P}eterson spectra of height~2,
  \href{https://doi.org/10.1112/jtopol/jtr030}{\textit{J.~Topol.}} \textbf{5}
  (2012), 137--168, \href{https://arxiv.org/abs/1101.3897}{arXiv:1101.3897}.

\bibitem{LN14}
Lawson T., Naumann N., Strictly commutative realizations of diagrams over the
  {S}teenrod algebra and topological modular forms at the prime~2,
  \href{https://doi.org/10.1093/imrn/rnt010}{\textit{Int. Math. Res. Not.}}
  \textbf{2014} (2014), 2773--2813,
  \href{https://arxiv.org/abs/1203.1696}{arXiv:1203.1696}.

\bibitem{Lee:2022spd}
Lee Y., Yonekura K., Global anomalies in {$8d$} supergravity,
  \href{https://doi.org/10.1007/jhep07(2022)125}{\textit{J.~High Energy Phys.}}
  \textbf{2022} (2022), no.~7, 125, 31~pages,
  \href{https://arxiv.org/abs/2203.12631}{arXiv:2203.12631}.

\bibitem{LMSM86}
Lewis Jr. L.G., May J.P., Steinberger M., McClure J.E., Equivariant stable
  homotopy theory, \textit{Lecture Notes in Math.}, Vol.~1213,
  \href{https://doi.org/10.1007/BFb0075778}{Springer}, Berlin, 1986.

\bibitem{Lud23}
Ludewig M., The spinor bundle on loop space,
  \href{https://arxiv.org/abs/2305.12521}{arXiv:2305.12521}.

\bibitem{HA}
Lurie J., Higher algebra, 2017,
  \url{https://www.math.ias.edu/~lurie/papers/HA.pdf}.

\bibitem{MG95}
Mahowald M., Gorbounov V., Some homotopy of the cobordism spectrum {$M{\rm
  O}\langle 8\rangle$}, in Homotopy Theory and Its Applications ({C}ocoyoc,
  1993), \textit{Contemp. Math.}, Vol.~188,
  \href{https://doi.org/10.1090/conm/188/02237}{American Mathematical Society},
  Providence, RI, 1995, 105--119.

\bibitem{MH02}
Mahowald M., Hopkins M., The structure of 24 dimensional manifolds having
  normal bundles which lift to~{$B{\rm O}[8]$}, in Recent Progress in Homotopy
  Theory ({B}altimore, {MD}, 2000), \textit{Contemp. Math.}, Vol.~293,
  \href{https://doi.org/10.1090/conm/293/04944}{American Mathematical Society},
  Providence, RI, 2002, 89--110.

\bibitem{MR09}
Mahowald M., Rezk C., Topological modular forms of level~3,
  \href{https://doi.org/10.4310/PAMQ.2009.v5.n2.a9}{\textit{Pure Appl.
  Math.~Q.}} \textbf{5} (2009), 853--872,
  \href{https://arxiv.org/abs/0812.2009}{arXiv:0812.2009}.

\bibitem{mathew2013homology}
Mathew A., The homology of tmf,
  \href{https://doi.org/10.4310/HHA.2016.v18.n2.a1}{\textit{Homology Homotopy
  Appl.}} \textbf{18} (2016), 1--29,
  \href{https://arxiv.org/abs/1305.6100}{arXiv:1305.6100}.

\bibitem{Mclaughlin}
McLaughlin D.A., Orientation and string structures on loop space,
  \href{https://doi.org/10.2140/pjm.1992.155.143}{\textit{Pacific~J. Math.}}
  \textbf{155} (1992), 143--156.

\bibitem{Mei22}
Meier L., Additive decompositions for rings of modular forms,
  \href{https://doi.org/10.25537/dm.2022v27.427-488}{\textit{Doc. Math.}}
  \textbf{27} (2022), 427--488,
  \href{https://arxiv.org/abs/1710.03461}{arXiv:1710.03461}.

\bibitem{Mei23}
Meier L., Connective models for topological modular forms of level {$n$},
  \href{https://doi.org/10.2140/agt.2023.23.3553}{\textit{Algebr. Geom.
  Topol.}} \textbf{23} (2023), 3553--3586,
  \href{https://arxiv.org/abs/2104.12649}{arXiv:2104.12649}.

\bibitem{Ochanine}
Ochanine S., Sur les genres multiplicatifs d\'efinis par des int\'egrales
  elliptiques,
  \href{https://doi.org/10.1016/0040-9383(87)90055-3}{\textit{Topology}}
  \textbf{26} (1987), 143--151.

\bibitem{PW88}
Pilch K., Warner N.P., String structures and the index of the {D}irac-{R}amond
  operator on orbifolds,
  \href{https://doi.org/10.1007/BF01466769}{\textit{Comm. Math. Phys.}}
  \textbf{115} (1988), 191--212.

\bibitem{pressley1986loop}
Pressley A., Segal G., Loop groups, \textit{Oxford Math. Monogr.}, Oxford University
  Press, New York, 1986.

\bibitem{Qui69}
Quillen D., On the formal group laws of unoriented and complex cobordism
  theory, \href{https://doi.org/10.1090/S0002-9904-1969-12401-8}{\textit{Bull.
  Amer. Math. Soc.}} \textbf{75} (1969), 1293--1298.

\bibitem{QZ25}
Quinn R., Zhu Q., Multiplicative equivariant {T}hom spectra \& structured real
  orientations, \href{https://arxiv.org/abs/2512.15573}{arXiv:2512.15573}.

\bibitem{Rez07}
Rezk C., Supplementary notes for Math 512 (Version 0.18), northwestern
  University~/ University of Illinois at Urbana–Champaign, 2007,
  \url{https://rezk.web.illinois.edu/512-spr2001-notes.pdf}.

\bibitem{Sati:2010dc}
Sati H., Geometric and topological structures related to {M}-branes~{II}:
  {T}wisted string and string{$^c$} structures,
  \href{https://doi.org/10.1017/S1446788711001261}{\textit{J.~Aust. Math.
  Soc.}} \textbf{90} (2011), 93--108,
  \href{https://arxiv.org/abs/1007.5419}{arXiv:1007.5419}.

\bibitem{Sati:2010ip}
Sati H., Twisted topological structures related to {M}-branes,
  \href{https://doi.org/10.1142/S0219887811005567}{\textit{Int.~J. Geom.
  Methods Mod. Phys.}} \textbf{8} (2011), 1097--1116,
  \href{https://arxiv.org/abs/1008.1755}{arXiv:1008.1755}.

\bibitem{Sati:2021uhj}
Sati H., Schreiber U., M/{F}-theory as {$Mf$}-theory,
  \href{https://doi.org/10.1142/S0129055X23500289}{\textit{Rev. Math. Phys.}}
  \textbf{35} (2023), 2350028, 131~pages,
  \href{https://arxiv.org/abs/2103.01877}{arXiv:2103.01877}.

\bibitem{SS25}
Sati H., Schreiber U., Flux quantization, in Encyclopedia of Mathematical
  Physics, \href{https://doi.org/10.1016/B978-0-323-95703-8.00078-1}{Academic
  Press}, Oxford, 2025, 281--324,
  \href{https://arxiv.org/abs/2402.18473}{arXiv:2402.18473}.

\bibitem{SSS12}
Sati H., Schreiber U., Stasheff J., Twisted differential string and fivebrane
  structures, \href{https://doi.org/10.1007/s00220-012-1510-3}{\textit{Comm.
  Math. Phys.}} \textbf{315} (2012), 169--213,
  \href{https://arxiv.org/abs/0910.4001}{arXiv:0910.4001}.

\bibitem{Seg70}
Segal G.B., Cohomology of topological groups, in Symposia {M}athematica,
  {V}ol.~{IV} ({INDAM}, {R}ome, 1968/69), Academic Press, London, 1970,
  377--387.

\bibitem{Seg75}
Segal G.B., Classifying space of a topological group in the {G}el'fand--{F}uks
  sense, \href{https://doi.org/10.1007/BF01075450}{\textit{Funct. Anal. Appl.}}
  \textbf{9} (1975), 131--133.

\bibitem{Sen23}
Senger A., Obstruction theory and the level~{$n$} elliptic genus,
  \href{https://doi.org/10.1112/S0010437X23007406}{\textit{Compos. Math.}}
  \textbf{159} (2023), 2000--2021,
  \href{https://arxiv.org/abs/2203.13743}{arXiv:2203.13743}.

\bibitem{Sen17}
Senger A., The {B}rown--{P}eterson spectrum is not {$\mathbb{E}_{2 ( p^2 + 2
  )}$} at odd primes,
  \href{https://doi.org/10.1016/j.aim.2024.109996}{\textit{Adv. Math.}}
  \textbf{458} (2024), 109996, 33~pages,
  \href{https://arxiv.org/abs/1710.09822}{arXiv:1710.09822}.

\bibitem{Sto12}
Stojanoska V., Duality for topological modular forms,
  \href{https://doi.org/10.4171/DM/368}{\textit{Doc. Math.}} \textbf{17}
  (2012), 271--311, \href{https://arxiv.org/abs/1105.3968}{arXiv:1105.3968}.

\bibitem{ST05}
Stolz S., Teichner P., The spinor bundle on loop space, 2005,
  \url{https://people.mpim-bonn.mpg.de/teichner/Math/ewExternalFiles/MPI.pdf}.

\bibitem{ST04}
Stolz S., Teichner P., What is an elliptic object?, in Topology, Geometry and
  Quantum Field Theory, \textit{London Math. Soc. Lecture Note Ser.}, Vol. 308,
  \href{https://doi.org/10.1017/CBO9780511526398.013}{Cambridge University
  Press}, Cambridge, 2004, 247--343.

\bibitem{stolz2011supersymmetric}
Stolz S., Teichner P., Supersymmetric field theories and generalized
  cohomology, in Mathematical Foundations of Quantum Field Theory and
  Perturbative String Theory, \textit{Proc. Sympos. Pure Math.}, Vol.~83,
  \href{https://doi.org/10.1090/pspum/083/2742432}{American Mathematical
  Society}, Providence, RI, 2011, 279--340,
  \href{https://arxiv.org/abs/1108.0189}{arXiv:1108.0189}.

\bibitem{Sto67}
Stong R.E., On complex-spin manifolds,
  \href{https://doi.org/10.2307/1970357}{\textit{Ann. of Math.}} \textbf{85}
  (1967), 526--536.

\bibitem{stong}
Stong R.E., Appendix: {C}alculation of {$\Omega^{{\rm Spin}}_{11}(K(Z,4))$}, in
  Workshop on Unified String Theories ({S}anta {B}arbara, {C}alif., 1985),
  World Scientific Publishing Co., Singapore, 1986, 430--437.

\bibitem{Tachikawa:2021mvw}
Tachikawa Y., Topological modular forms and the absence of a heterotic global
  anomaly, \href{https://doi.org/10.1093/ptep/ptab060}{\textit{PTEP. Prog.
  Theor. Exp. Phys.}} \textbf{2022} (2022), 04A107, 8~pages,
  \href{https://arxiv.org/abs/2103.12211}{arXiv:2103.12211}.

\bibitem{TY23}
Tachikawa Y., Yamashita M., Anderson self-duality of topological modular forms,
  its differential-geometric manifestations, and vertex operator algebras,
  \href{https://arxiv.org/abs/2305.06196}{arXiv:2305.06196}.

\bibitem{TY:2021mby}
Tachikawa Y., Yamashita M., Topological modular forms and the absence of all
  heterotic global anomalies,
  \href{https://doi.org/10.1007/s00220-023-04761-2}{\textit{Comm. Math. Phys.}}
  \textbf{402} (2023), 1585--1620, {C}orrection,
  \href{https://doi.org/10.1007/s00220--023--04804--8}{\textit{Comm. Math.
  Phys.}} \textbf{402} (2023), 2131,
  \href{https://arxiv.org/abs/2108.13542}{arXiv:2108.13542}.

\bibitem{Tachikawa:2023nne}
Tachikawa Y., Yamashita M., Yonekura K., Remarks on {${\rm mod}\text{-}2$}
  elliptic genus,
  \href{https://doi.org/10.1007/s00220-024-05202-4}{\textit{Comm. Math. Phys.}}
  \textbf{406} (2025), 16, 28~pages,
  \href{https://arxiv.org/abs/2302.07548}{arXiv:2302.07548}.

\bibitem{TY25}
Tachikawa Y., Yonekura K., On invariants of two-dimensional minimally
  supersymmetric field theories,
  \href{https://arxiv.org/abs/2508.04916}{arXiv:2508.04916}.

\bibitem{TZ24}
Tachikawa Y., Zhang H.Y., On a {$\mathbb Z_3$}-valued discrete topological term
  in 10d heterotic string theories,
  \href{https://doi.org/10.21468/scipostphys.17.3.077}{\textit{SciPost Phys.}}
  \textbf{17} (2024), 077, 26~pages,
  \href{https://arxiv.org/abs/2403.08861}{arXiv:2403.08861}.

\bibitem{Tho62}
Thomas E., The torsion {P}ontryagin classes,
  \href{https://doi.org/10.2307/2034967}{\textit{Proc. Amer. Math. Soc.}}
  \textbf{13} (1962), 485--488.

\bibitem{Tod73}
Toda H., Cohomology {${\rm mod}$} {$3$} of the classifying space {$BF_{4}$} of
  the exceptional group {${\bf F}_{4}$},
  \href{https://doi.org/10.1215/kjm/1250523438}{\textit{J.~Math. Kyoto Univ.}}
  \textbf{13} (1973), 97--115.

\bibitem{Tok24}
Tokic L., {$K(2)$}-local splittings of finite {G}alois extensions of
  {$MU\langle6\rangle$} and {$MString$},
  \href{https://arxiv.org/abs/2411.16280}{arXiv:2411.16280}.

\bibitem{Wal10}
Waldorf K., Multiplicative bundle gerbes with connection,
  \href{https://doi.org/10.1016/j.difgeo.2009.10.006}{\textit{Differential
  Geom. Appl.}} \textbf{28} (2010), 313--340,
  \href{https://arxiv.org/abs/0804.4835}{arXiv:0804.4835}.

\bibitem{Wal12a}
Waldorf K., Transgression to loop spaces and its inverse, {I}: {D}iffeological
  bundles and fusion maps, \textit{Cah. Topol. G\'eom. Diff\'er. Cat\'eg.}
  \textbf{53} (2012), 162--210,
  \href{https://arxiv.org/abs/0911.3212}{arXiv:0911.3212}.

\bibitem{Wal12b}
Waldorf K., Transgression to loop spaces and its inverse, {III}: {G}erbes and
  thin fusion bundles,
  \href{https://doi.org/10.1016/j.aim.2012.08.016}{\textit{Adv. Math.}}
  \textbf{231} (2012), 3445--3472,
  \href{https://arxiv.org/abs/1109.0480}{arXiv:1109.0480}.

\bibitem{Wal15}
Waldorf K., String geometry vs. spin geometry on loop spaces,
  \href{https://doi.org/10.1016/j.geomphys.2015.07.003}{\textit{J.~Geom.
  Phys.}} \textbf{97} (2015), 190--226,
  \href{https://arxiv.org/abs/1403.5656}{arXiv:1403.5656}.

\bibitem{Wal16}
Waldorf K., Spin structures on loop spaces that characterize string manifolds,
  \href{https://doi.org/10.2140/agt.2016.16.675}{\textit{Algebr. Geom. Topol.}}
  \textbf{16} (2016), 675--709,
  \href{https://arxiv.org/abs/1209.1731}{arXiv:1209.1731}.

\bibitem{Wal16a}
Waldorf K., Transgression to loop spaces and its inverse,~{II}: {G}erbes and
  fusion bundles with connection,
  \href{https://doi.org/10.4310/AJM.2016.v20.n1.a4}{\textit{Asian~J. Math.}}
  \textbf{20} (2016), 59--115,
  \href{https://arxiv.org/abs/1004.0031}{arXiv:1004.0031}.

\bibitem{Wal23}
Waldorf K., String structures and loop spaces, in Encyclopedia of Mathematical
  Physics. {V}ol.~4. {T}opology \& Geometry, Academic Press, Amsterdam, 2025,
  186--204, \href{https://arxiv.org/abs/2312.12998}{arXiv:2312.12998}.

\bibitem{Wan08}
Wang B.-L., Geometric cycles, index theory and twisted {$K$}-homology,
  \href{https://doi.org/10.4171/JNCG/27}{\textit{J.~Noncommut. Geom.}}
  \textbf{2} (2008), 497--552,
  \href{https://arxiv.org/abs/0710.1625}{arXiv:0710.1625}.

\bibitem{WG18}
Wang Q.R., Gu Z.C., Towards a complete classification of symmetry-protected
  topological phases for interacting fermions in three dimensions and a general
  group supercohomology theory,
  \href{https://doi.org/10.1103/PhysRevX.8.011055}{\textit{Phys. Rev.~X}}
  \textbf{8} (2018), 011055, 29~pages,
  \href{https://arxiv.org/abs/1703.10937}{arXiv:1703.10937}.

\bibitem{WG20}
Wang Q.R., Gu Z.C., Construction and classification of symmetry-protected
  topological phases in interacting fermion systems,
  \href{https://doi.org/10.1103/PhysRevX.10.031055}{\textit{Phys. Rev.~X}}
  \textbf{10} (2020), 031055, 64~pages,
  \href{https://arxiv.org/abs/1811.00536}{arXiv:1811.00536}.

\bibitem{Wil15}
Wilson D., Orientations and topological modular forms with level structure,
  \href{https://arxiv.org/abs/1507.05116}{arXiv:1507.05116}.

\bibitem{Wil82}
Wilson W.S., Brown--{P}eterson homology: an introduction and sampler,
  \textit{CBMS Reg. Conf. Ser. Math.}, Vol.~48, Conference Board of the
  Mathematical Sciences, Washington, DC, 1982.

\bibitem{Witten:1987f}
Witten E., Elliptic genera and quantum field theory,
  \href{https://doi.org/10.1007/BF01208956}{\textit{Comm. Math. Phys.}}
  \textbf{109} (1987), 525--536.

\bibitem{Witten:1987cg}
Witten E., The index of the {D}irac operator in loop space, in Elliptic Curves
  and Modular Forms in Algebraic Topology ({P}rinceton, {NJ}, 1986),
  \textit{Lecture Notes in Math.}, Vol.~1326,
  \href{https://doi.org/10.1007/BFb0078045}{Springer}, Berlin, 1988, 161--181.

\bibitem{Yon22}
Yonekura K., Heterotic global anomalies and torsion {W}itten index,
  \href{https://doi.org/10.1007/jhep10(2022)114}{\textit{J.~High Energy Phys.}}
  \textbf{2022} (2022), no.~10, 114, 38~pages,
  \href{https://arxiv.org/abs/2207.13858}{arXiv:2207.13858}.

\bibitem{Zil77}
Ziller W., The free loop space of globally symmetric spaces,
  \href{https://doi.org/10.1007/BF01390161}{\textit{Invent. Math.}} \textbf{41}
  (1977), 1--22.

\end{thebibliography}
\end{document}